\documentclass[a4paper,12pt]{article}

\usepackage{tikz}
\usetikzlibrary{arrows,backgrounds,calc,trees}
\usepackage{amsmath}
\usepackage{amsfonts}
\usepackage{fullpage}
\usepackage{amsthm}
\usepackage{tikz}
\usepackage{enumerate}
\usetikzlibrary{matrix}

\usepackage[affil-it]{authblk}
\usepackage{graphicx}
\usepackage{caption}
\usepackage{subcaption}
\usepackage{multirow}
\usepackage{hyperref}
\usepackage{cleveref}
\usepackage{dcolumn}
\usepackage{bm}
\usepackage{algorithm2e}

\graphicspath{{./figures_revision/}}

\newcommand\restr[2]{{
  \left.\kern-\nulldelimiterspace 
  #1 
  \vphantom{\big|} 
  \right|_{#2} 
  }}

\newcounter{dummy} \numberwithin{dummy}{section}
\newtheorem{remark}[dummy]{Remark}
\newtheorem{lemma}[dummy]{Lemma}

\newtheorem{definition}[dummy]{Definition}
\newtheorem{theorem}[dummy]{Theorem}

\newtheorem{corollary}[dummy]{Corollary}

\newtheorem{proposition}[dummy]{Proposition}

\newcommand{\divg}{\textup{div}  }

\title{A dynamic Laplacian for identifying Lagrangian coherent structures on weighted Riemannian manifolds}
\author{Gary Froyland and Eric Kwok}
\affil{School of Mathematics and Statistics\\ University of New South Wales \\ Sydney NSW 2052, Australia}
\date{\today}

\begin{document}

\maketitle
\numberwithin{equation}{section}
\begin{abstract}
Transport and mixing in dynamical systems are important properties for many physical, chemical, biological, and engineering processes. The detection of transport barriers for dynamics with general time dependence is a difficult, but important problem, because such barriers control how rapidly different parts of phase space (which might correspond to different chemical or biological agents) interact. The key factor is the growth of interfaces that partition phase space into separate regions. The paper \cite{froyland14} introduced the notion of \textit{dynamic isoperimetry}:  the study of sets with persistently small boundary size (the interface) relative to enclosed volume, when evolved by the dynamics. Sets with this minimal boundary size to volume ratio were identified as level sets of dominant eigenfunctions of a \textit{dynamic Laplace operator}.

In this present work we extend the results of \cite{froyland14} to the situation where the dynamics (i) is not necessarily volume-preserving, (ii) acts on initial agent concentrations different from uniform concentrations, and (iii) occurs on a possibly curved phase space. Our main results include generalised versions of the dynamic isoperimetric problem, the dynamic Laplacian, Cheeger's inequality, and the Federer-Fleming theorem.
We illustrate the computational approach with some simple numerical examples.
\end{abstract}
\numberwithin{equation}{section}

\section{Introduction}

The mathematics of transport in nonlinear dynamical systems has received considerable attention for more than two decades, driven in part by applications in fluid dynamics, atmospheric and ocean dynamics, molecular dynamics, granular flow and other areas. We refer the reader to \cite{ottino89,wiggins90,meiss92,aref02,wiggins05} for reviews of transport and transport-related phenomena. Early attempts to characterise transport barriers in fluid dynamics include time-dependent invariant manifolds (such as lobe-dynamics \cite{wiggins90}) and finite-time Lyapunov exponents \cite{pierrehumbert91, pierrehumbert93, doerner99, haller02, shadden05}.
More recently, in two-dimensional area-preserving flows, \cite{blackhole} proposed finding closed curves whose time-averaged length is stationary under small perturbations;  this aim is closest in spirit\footnote{The extension \cite{haller3d} of \cite{blackhole} to three dimensions is less aligned with \cite{froyland14}, as \cite{haller3d} asks for uniform expansion in all directions in the two-dimensional tangent space to potential LCS surfaces, whereas the approach of \cite{froyland14} in three-dimensions is simply concerned with surface growth without a uniform expansion restriction.} to the predecessor work of this paper \cite{froyland14}, though the latter theory applies in arbitrary finite dimensions and the curves need not be closed.
In parallel to these efforts, the notion of almost-invariant sets \cite{dellnitz99} in autonomous systems spurred the development of probabilistic methods to transport based around the transfer operator. In relation to transport barriers, numerical observations \cite{froyland09} indicated connections between the boundaries of almost-invariant sets and invariant manifolds of low-period points. Transfer operator techniques were later extended to dynamical systems with general time dependence, with the introduction of coherent sets as the time-dependent analogues of almost-invariant sets. \cite{froyland10, froyland13}. Topological approaches to phase space mixing have also been developed \cite{gouillart06}, including connections with almost-invariant sets \cite{grover12}.

In \cite{froyland14}, Froyland introduced the notion of a dynamic isoperimetric problem, namely searching for subsets of a manifold whose boundary size to enclosed volume is minimised in a time-averaged sense under general time-dependent nonlinear dynamics. Solutions to this problem were constructed from eigenvectors of a dynamic Laplace operator, a time-average of pullbacks of Laplace operators under the dynamics. It was shown in \cite{froyland14} that the dynamic Laplace operator arises as a zero-diffusion limit of the transfer operator constructions for finite-time coherent sets in \cite{froyland13}. This result demonstrated that finite-time coherent sets (those sets that maximally resist mixing over a finite time interval), also had the persistently small boundary length to enclosed volume ratio property; intuitively this is reasonable because diffusive mixing between sets can only occur through their boundaries. Thus, finite-time coherent sets have dual minimising properties:  slow mixing (probabilistic) and low boundary growth (geometric). The theory in \cite{froyland14} was restricted to the situation where the advective dynamics was volume-preserving, and to tracking the transport of a uniformly distributed tracer in Euclidean space. \emph{In the present work, we extend the results of \cite{froyland14} in three ways:  (i) to dynamics that is not volume preserving, (ii) to tracking the transport of nonuniformly distributed tracers, and (iii) to dynamics operating on curved manifolds.}

We now begin to be more specific about the results of the present paper. Let $M$ denote a connected $r$-dimensional compact $C^\infty$ Riemannian manifold and $\Gamma$ denote a $C^\infty$ hypersurface disconnecting $M$ into submanifolds $M_1, M_2$;  that is $\{M_1,M_2,\Gamma\}$ is a partition of $M$. For example, $M$ could be the unit square $[0,1]^2\subset \mathbb{R}^2$ and $\Gamma$ either a curve from a boundary point to another boundary point or a closed curve. On $M$ we place a Riemannian metric $m$ and a probability measure $\mu_r$.
The size of a set $M_1\subset M$ is given by $\mu_r(M_1)$ and by a process of inducing explained in the next section, we develop a measure $\mu_{r-1}$ to determine the size of ($r-1$)-dimensional objects such as $\Gamma$. To continue our trivial example, if $\mu_r=\mu_2$ is 2-dimensional Lebesgue measure, then $\mu_{r-1}=\mu_1$ is 1-dimensional Lebesgue measure, which can be used to measure curve length in $\mathbb{R}^2$. In order to track the transport of nonuniformly distributed passive tracers (e.g.\ chemical concentrations in fluids, air mass in the atmosphere, salt in the ocean), we require a general probability measure $\mu_r$ that represents the initial distribution to be tracked.
Similarly, in order to estimate the amount of material that can be ejected through the boundary at any given time, we require the measure $\mu_{r-1}$ to compute boundary size.

Let us suppose that the dynamics over a finite time duration is given by $T:M\to N$, where $T$ can be a single transformation, the concatenation of several maps over several discrete time steps, or the flow map for a time-dependent vector field over some duration $\tau$. The following brief discussion checks the boundary at the initial and final times, but in the case of continuous time, one may continuously check the boundary size as described in Section \ref{sec:mts}. The manifold $N$ is equipped with a Riemannian metric $n$ (which need not be the pushforward of $m$), and a probability measure $\nu_r:=\mu_r\circ T^{-1}$ (which must be the pushforward of $\mu_r$). Conservation of mass enforces the definition $\nu_r:=\mu_r\circ T^{-1}$, but in many applications we may not want $m$ and $n$ to be related by $T$. Continuing our example, if $M=N=[0,1]^2$ and $T$ is not area-preserving, then $\mu_2$ (2-dimensional Lebesgue measure) will be transformed by $T$ to a probability measure $\nu_2$ with a non-constant density. Furthermore, since $M=N=[0,1]^2\subset \mathbb{R}^2$, we have in this example that $m=n$ is the standard Euclidean metric.

Given a disconnecting hypersurface $\Gamma$, we compute the dynamic Cheeger constant
\begin{equation}
\label{cheeger0}
\mathbf{H}^D(\Gamma):=\frac{\mu_{r-1}(\Gamma)+\nu_{r-1}(T\Gamma)}{2\min\{\mu_r(M_1),\mu_r(M_2)\}},
\end{equation}
and wish to minimise $\mathbf{H}^D(\Gamma)$ over all smooth $\Gamma$ disconnecting $M$. The numerator of \eqref{cheeger0} quantifies the boundary size of $\Gamma$ \emph{and} its image $T\Gamma$;  thus, minimising over all smooth $\Gamma$ finds the interface $\Gamma$ which has minimal combined length, both before and after evolution (and in the later continuous versions, throughout evolution). The denominator of (\ref{cheeger0}) is a standard normalisation condition in (static) isoperimetric problems to avoid trivial solutions and ensure that both $M_1$ and $M_2$ are of macroscopic size.
Equation (\ref{cheeger0}) is a natural generalisation of equation (1) \cite{froyland14} for non-volume-preserving dynamics.

Beyond the generalised dynamic isoperimetric problem described above, our main contributions are firstly the formulation of a dynamic Sobolev constant (a functional version of the dynamic Cheeger constant) in our general setting and a corresponding proof of a dynamic version of the celebrated Federer-Fleming theorem (see e.g p.131 \cite{chavel01} for the classical static statement and Theorem 3.1 \cite{froyland14} for the dynamic statement in the volume-preserving, uniform density, flat manifold setting), which equates the geometric Cheeger constant with the functional Sobolev constant. Secondly, we define a generalised version of the dynamic Laplace operator constructed in \cite{froyland14}. In our general setting (see Section \ref{sec:dfop} for details), the operator is
\begin{equation}
\label{dynlap0}
\triangle^D:=\frac{1}{2}\left(\triangle_\mu+\mathcal{L}^*\triangle_\nu\mathcal{L}\right),
\end{equation}
where $\triangle_\mu, \triangle_\nu$ are weighted Laplace-Beltrami operators, weighted by $\mu_r, \nu_r$ respectively.
The operator $\mathcal{L}:L^2(M,m,\mu_r)\to L^2(N,n,\nu_r)$ is simply $\mathcal{L}f=f\circ T^{-1}$ and $\mathcal{L}^*f=f\circ T$.
See Section \ref{sec:lts} for continuous time versions of $\triangle^D$.
The specialisation of (\ref{dynlap0}) to the volume-preserving, unweighted setting may be found in equation (15) \cite{froyland14}.
A related construction is considered in \cite{karrasch16} from the point of view of heat flow, where they search for a single metric for a Laplace-Beltrami operator, rather than solving an isoperimetric-type problem, and follow ideas of \cite{thiffeault} to consider flow in Lagrangian coordinates and make connections to almost-invariant sets subjected to time-dependent diffusion.

We prove a dynamic version of the well-known Cheeger inequality in our generalised setting (see \cite{cheeger69} for the classic (static) Cheeger inequality and Theorem 3.2 \cite{froyland14} for the dynamic Cheeger inequality in the volume-preserving, uniform density, flat manifold setting), which bounds the Cheeger constant above in terms of the dominant nontrivial eigenvalue of $\triangle^Df=\lambda f$ (with natural Neumann-like boundary conditions). Finally, we prove that
\begin{equation}
\label{cvgce0}
\lim_{\epsilon\to 0} \frac{(\mathcal{L}_\epsilon^*\mathcal{L}-\mbox{Id})f}{\epsilon^2}=c\cdot \triangle^Df,
\end{equation}
in a sense made precise in Section \ref{sec:3}, where Id is the identity, $\mathcal{L}_\epsilon$ is an $\epsilon$-mollified version of $\mathcal{L}$, used to compute finite-time coherent sets in \cite{froyland13} and $c$ is an explicit constant.
Because singular vectors of $\mathcal{L}_\epsilon$ (eigenvectors of $\mathcal{L}^*_\epsilon\mathcal{L}_\epsilon$) are used in \cite{froyland13}, and eigenvectors of $\triangle^D$ are used in the present work, this result shows that in the small perturbation limit, the purely probabilistic constructions of \cite{froyland13} coincide with the purely geometric constructions of the present paper.
The limit (\ref{cvgce0}) generalises Theorem 5.1 \cite{froyland14} to the setting of dynamics that need not be volume preserving, to the tracking of weighted tracers, and to curved domains.

The paper is arranged as follows.
In Section \ref{sec:2} we provide relevant background material from differential geometry.
Section \ref{sec:dcc} describes the dynamic isoperimetric problem on weighted manifolds and states the dynamic Federer-Fleming theorem.
Section \ref{sec:di} details the dynamic Laplace operator on weighted manifolds and states the dynamic Cheeger inequality.
In section \ref{sec:3}, we state the convergence result \eqref{cvgce0}.
Section \ref{sec:num} contains illustrative numerical experiments and most of the proofs are contained in the appendices.
In comparison to \cite{froyland14}, Theorems \ref{thm:dff}, \ref{thm:spec}, \ref{thm:wci}, and \ref{thm:3.1} in this work generalise respectively Theorems 3.1, 4.1, 3.2, and 5.1 in \cite{froyland14}.

\section{Primer on differential geometry}\label{sec:2}

Let $M$ be a compact, connected $r$-dimensional $C^\infty$ Riemannian manifold. We denote the boundary of $M$ by $\partial M$. If $\partial M$ is non-empty, then we assume that $\partial M$ is $C^\infty$. We are interested in tracking the masses of the $r$ and $r-1$ dimensional subsets of $M$ as this manifold is transformed by a general smooth dynamical system. We now give a brief introduction of the key tools in differential geometry for performing the above task; additional details are provided in Section \ref{sec:dfg} of the appendix.

Recall that to compute the $r$-dimensional volume of the objects in $M$, one considers a metric tensor on the tangent space $\mathcal{T}_xM$ at the point $x\in M$. In particular, the Riemannian metric $m$ on $M$ associates each point $x\in M$ with a symmetric bilinear form $m( . , . )_x:\mathcal{T}_xM\times \mathcal{T}_xM\to \mathbb{R}$, yielding a volume form $\omega_m^r$ on $M$ (see Appendix \ref{sec:df} for more details). The differential $r$-form $\omega_m^r$ defines an $r$-dimensional volume measure on any measurable subset $U\subset M$ by $V_m(U):=\int_U \omega_m^r$. To describe the mass distribution of the objects in $M$, we consider a weighted Riemannian manifold $(M, m, \mu_r)$, where $\mu_r$ is an absolutely continuous probability measure with respect to $V_m$; that is, there exist $h_\mu\in L^1(M, V_m)$ such that
\begin{equation*}
 \mu_r(U)=\int_U d\mu_r=\int_U h_\mu\cdot\omega_m^r,
\end{equation*}
for all measurable $U\subset M$, and $\mu_r(M)=1$. Since any subset of $M$ with $\mu_r$ measure zero has no physical impact, without loss of generality we assume that the density $h_\mu$ is uniformly bounded away from zero.


Let $(N, n, \nu_r)$ be another weighted Riemannian manifold, where $N$ is a compact, connected $r$-dimensional $C^\infty$ Riemannian manifold, $n$ the Riemannian metric on $N$, and $\nu_r$ an absolutely continuous probability measure with respect to $V_n$. As before, we shall assume that the density $h_\nu$ of $\nu_r$ is uniformly bounded away from zero. Consider a general dynamical system $T:M\to N$ that acts as a $C^\infty$-diffeomorphism from $M$ onto $N$. For the purpose of modeling physical processes, we assume that no mass is lost under transport; that is, the measure $\mu_r$ on $M$ is transformed under the action of $T$ to $\nu_r:=\mu_r\circ T^{-1}$. Because the densities $h_\mu$, $h_\nu$ are uniformly bounded away from zero, $\nu_r=\mu_r\circ T^{-1}$, and $T$ is a diffeomorphism, the nondegeneracy of the metrics $n, m$ implies that the Jacobian associated with $T$ must be uniformly bounded above and away from zero (see Appendix \ref{sec:B3}). We emphasise that $n$ is not necessarily the push-forward of $m$, and that $T$ is not an isometry from $(M, m)$ to $(N, n)$ in general. 

Let $\mathcal{T}M$ denote the tangent bundle of $M$; that is, $\mathcal{T}M:=\cup_{x\in M}\{x\}\times \mathcal{T}_xM$. A vector field $\mathcal{V}$ on $M$ is a section of the bundle $\mathcal{T}M$; that is the image of $x\in M$ under $\mathcal{V}$ is the tangent vector $\mathcal{V}_x\in \mathcal{T}_xM$. For $k\geq 1$, we denote the space of $k$-times continuously differentiable vector fields on $M$ by $\mathcal{F}^k(M)$. For a pair $\mathcal{V},\mathcal{W}\in \mathcal{F}^k(M)$, one can view $m(\mathcal{V}, \mathcal{W}):M\to \mathbb{R}$ as a $C^k$ function on $M$ given by $m(\mathcal{V}, \mathcal{W})(x)=m(\mathcal{V}_{x}, \mathcal{W}_{x})_x$ for all $x\in M$. Denote by $\mathcal{T}^*M$ the dual bundle of $\mathcal{T}M$; that is the cotangent bundle $\mathcal{T}^*M:\cup_{x\in M}\{x\}\times\mathcal{T}^*_xM$, where $\mathcal{T}_x^*M$ is the vector dual of $\mathcal{T}_xM$. The covector fields on $N$ are sections of the bundle $\mathcal{T}^*M$.

 It is practical to associate the diffeomorphism $T:M\to N$ with the linear tangent map $T_*$ that takes vector fields on $M$ to vector fields on $N$, which we now define. Let $\gamma_x:(-\epsilon, \epsilon) \to M$ be a family of parameterised curves in $M$, with $\gamma_x(0)=x\in M$. Suppose for each $x\in M$ that $\mathcal{V}_x\in \mathcal{T}_xM$ is tangent to the curve $\gamma_x$ at $x$. The action of $\mathcal{V}_x$ on a differentiable function $f$ at each point $x\in M$ is defined to be the number
\begin{equation}\label{eq:mathcalV}
 \restr{\mathcal{V}_xf}{x}:=\restr{\frac{\partial (f\circ \gamma_x)}{\partial t}}{t=0};
\end{equation}
that is $\restr{\mathcal{V}_xf}{x}$ measures the initial rate of change of $f$ along a curve with tangent $\mathcal{V}_x$ at the point $x$. The local push-forward map $(T_*)_{x_0}:\mathcal{T}_{x_0}M\to \mathcal{T}_{Tx_0}N$ is defined at a fixed point $x_0\in M$ as
\begin{equation*}
 \restr{[(T_*)_{x_0}\mathcal{V}_{x_0}]g}{Tx_0}:=\restr{\mathcal{V}_{x_0}(g\circ T)}{x_0},
\end{equation*}
for all $g\in C^k(N, \mathbb{R})$. The collection of local push-forward maps define a linear tangent map $T_*:\mathcal{F}^k(M)\to \mathcal{F}^{k}(N)$ via
\begin{equation}\label{eq:T_*}
[(T_*\mathcal{V})g](Tx):=\restr{[(T_*)_{x}\mathcal{V}_{x}]g}{Tx},
\end{equation}
for all $x\in M$, and $g\in C^k(N, \mathbb{R})$.

Next, we define the linear cotangent map $T^*$ that takes covector fields on $N$ to covector fields on $M$ as follows. Given a vector field $\mathcal{V}$ on $M$, the action of $\mathcal{V}$ on a differentiable function $f$ on $M$ is a function $\mathcal{V}f:M\to \mathbb{R}$ given by $\mathcal{V}f(x):=\restr{\mathcal{V}_xf}{x}$. By the duality of the tangent and cotangent spaces, the cotangent vector fields are differential $1$-forms $df$ that map vector fields on $M$ to functions on $M$ via $df(\mathcal{V}):=\mathcal{V}f$. The cotangent mapping on differential $1$-forms is defined by
\begin{equation}\label{eq:T^*}
 [T^*(dg)]\mathcal{V}:=dg(T_*\mathcal{V})=\mathcal{V}(g\circ T)=(T_*\mathcal{V})g,
\end{equation}
for all $\mathcal{V}\in \mathcal{F}^k(M)$ and $g\in C^k(N, \mathbb{R})$. One can associate the cotangent mapping $T^*$ with an exterior product of $p$-forms, $1\leq p\leq r$ (see Appendix \ref{sec:df}). In particular, since the metric tensor $n$ is a symmetric $2$-form on $\mathcal{F}^k(N)$, one defines the \emph{pullback metric} of $n$ by
\begin{equation}\label{eq:pbm}
 T^*n(\mathcal{V}_1, \mathcal{V}_2)(x):=n(T_*\mathcal{V}_1, T_*\mathcal{V}_2)(Tx),
\end{equation}
for all vector fields $\mathcal{V}_1, \mathcal{V}_2$ on $M$, and each point $x\in M$; that is, the pullback metric $T^*n$ is defined in such a way that $T$ is an isometry from $(M, T^*n)$ to $(N, n)$.


To compute the co-dimension $1$ volume of $(r-1)$-dimensional subsets of $(M, m, \mu_r)$ and $(N, n, \nu_r)$, one uses the induced Riemannian metric. Suppose $\Gamma$ is a compact $C^\infty$ co-dimension $1$ subset of $M$. The embedding $\Phi:\Gamma \hookrightarrow M$ induces a Riemannian metric on $\Gamma$ via the pullback metric associated with $\Phi$; that is $\Phi^*m$ is the induced metric on $\Gamma$. Let $\omega_m^{r-1}$ denote the $(r-1)$-dimensional volume form corresponding to the induced metric $\Phi^*m$ (i.e\ $\omega_m^{r-1}=\omega^r_{\Phi^*m}$). To describe the distribution of mass on $\Gamma$, we define the $(r-1)$-dimensional measure $\mu_{r-1}$ on $M$ by
\begin{equation}\label{eq:omega_m^{r-1}}
 \mu_{r-1}(\Gamma):=\int_{\Gamma}h_\mu\cdot\omega^{r-1}_m,
\end{equation}
where $h_\mu$ is the density of $\mu_r$; the measure $\mu_{r-1}$ captures the mass distribution on $\Gamma$ via $h_\mu$. Similarly, the co-dimension $1$ mass distribution on a $C^\infty$, compact subset of $N$ is captured by the $(r-1)$-dimensional measure $\nu_{r-1}$ via the density $h_\nu$ of $\nu_r$. 
We now provide an example to demonstrate that the $\mu_{r-1}$ measure on certain hypersurfaces can be significantly increased under the action of a transformation $T$.

\subsection{Shear on a two-dimensional cylinder}\label{example1}

 Let $M=[0, 4)/\sim\times[0,1]$ be a $2$-dimensional cylinder in $\mathbb{R}^2$, where $\sim$ is identification at interval endpoints; that is, $M$ is periodic in the first coordinate with period $4$. The Riemannian metric  $e$ on $M$ is given by the Kronecker delta $\delta_{ij}$, so that the volume form $\omega_e^2$ on $M$ is $\omega_m^2=dx_1dx_2$. To form a weighted Riemannian manifold $(M, e, \mu_2)$, we set the density $h_\mu$ of $\mu_2$ to be a positive and periodic function $h_\mu(x_1, x_2)=\frac{1}{8}(\sin(\pi x_1)+2)$.

 Consider the hypersurface $\Gamma=\{x\in M: x_1=1.5,3.5\}$;  we choose this surface because it is the solution of the classical ``static'' isoperimetric problem defined by minimising (\ref{cheeger0}) without the second term in the numerator.
 The curve $\Gamma$ is two vertical lines on $M$ that pass over regions with minimal density $h_\mu$ as shown in Figure \ref{fig:static1a}. One can compute $\mu_1(\Gamma)$ analytically by noting that the induced Riemannian metric on $\Gamma$ is given by $dx_2$; thus
\begin{equation*}
 \mu_1(\Gamma)=\int_{0}^1 h_\mu(1.5, x_2)\,dx_2+\int_{0}^1 h_\mu(3.5, x_2)\,dx_2=0.25.
\end{equation*}
 \begin{figure}[h]
         \begin{subfigure}[t]{0.5\textwidth}
                 \includegraphics[width=\textwidth]{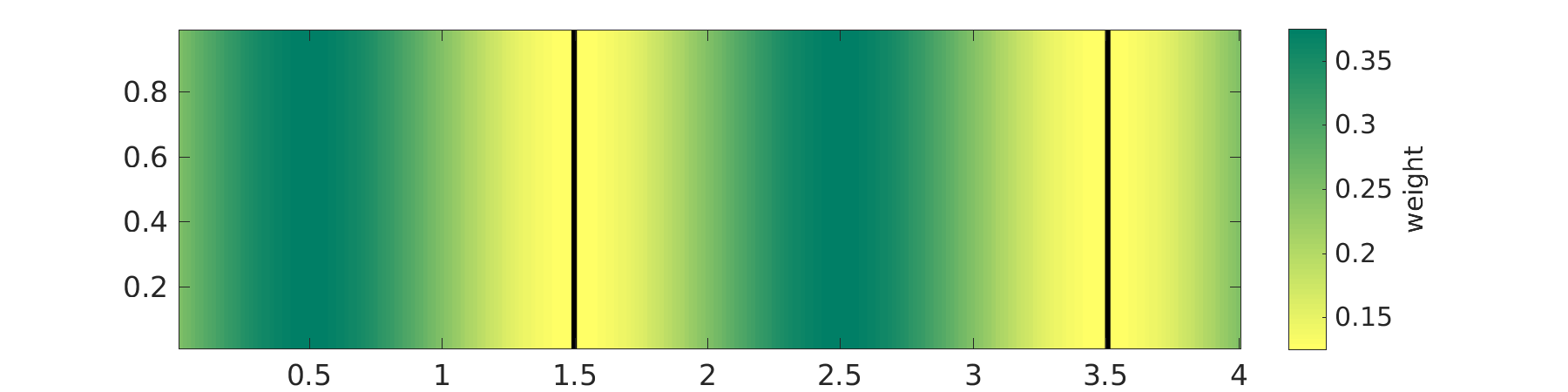}
                 \caption{}\label{fig:static1a}  
         \end{subfigure}
         \begin{subfigure}[t]{0.5\textwidth}
                 \includegraphics[width=\textwidth]{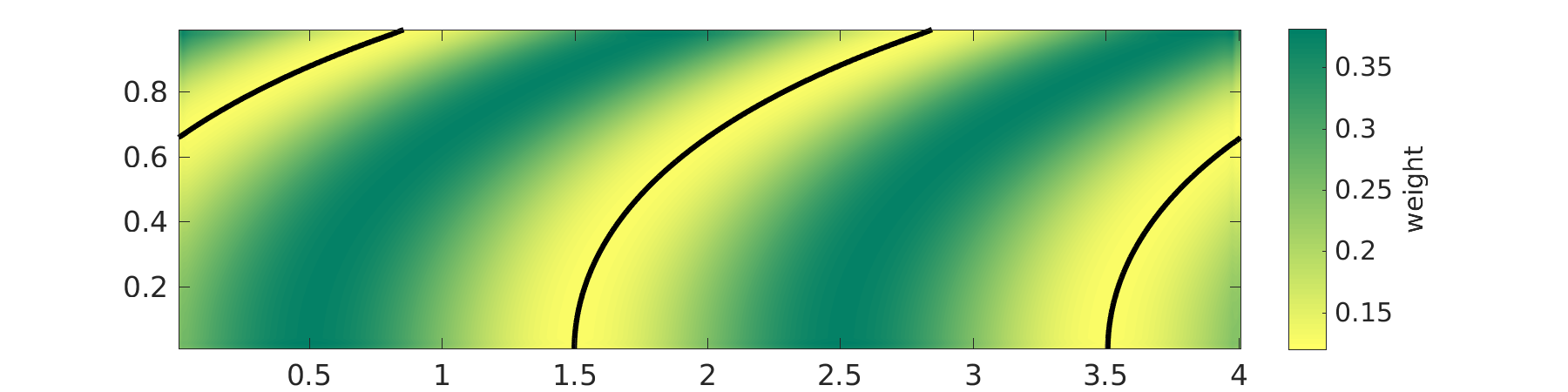}
                 \caption{}\label{fig:static2a}  
         \end{subfigure}
         \caption{Deformation of $2$-dimensional cylinder under nonlinear shear $T$. $(a)$ Colours are values of $h_\mu$, and black lines are the hypersurface $\Gamma$. $(b)$ Values of $h_\mu\circ T^{-1}$, and $T\Gamma$.} \label{fig:static1}
\end{figure}

Let us now apply the following transformation to $M$,
\begin{equation*}
 T(x_1, x_2)=\left(x_1+\frac{\cosh\left(2 x_2\right)-1}{2}, x_2\right),
\end{equation*}
where the first coordinate is computed modulo $4$. The map $T$ is a nonlinear horizontal shear. The hypersurface $\Gamma$ is transformed to $T\Gamma$ under the action of $T$ as shown in Figure \ref{fig:static2a}. The shearing magnitude $(\cosh(2x_2)-1)/2$ is chosen to simplify the analytical computation of $\nu_1(T\Gamma)$. It is easy to verify that $T$ is area-preserving. Since $T$ is area-preserving and $\nu_2=\mu_2\circ T^{-1}$, one has
\begin{equation*}
 \int_{TM} h_\nu\,dx_1dx_2=\int_{M} h_\mu\,dx_1dx_2=\int_{TM} h_\mu\circ T\,dx_1dx_2,
\end{equation*}
which implies $h_\nu=h_\mu\circ T$ in this example.

To compute the $\nu_1$ measure on $T\Gamma$, we parametrise the curve $T\Gamma$ by $T\Gamma=(\sigma_c(t), t)$ for $t\in [0,1]$, where $\sigma_c(t)=c+\frac{\cosh{(2t)-1}}{2}$, for $c=1.5, 3.5$. Furthermore, by using the fact that $h_\nu=h_\mu\circ T$, one has
\begin{equation*}
h_\nu(\sigma_c(t), t)=h_\mu(c, t)=\restr{\frac{\sin(\pi c)+2}{8}}{c=1.5,3.5}=\frac{1}{8},
\end{equation*}
 for all $t\in [0, 1]$. Therefore
\begin{align*}
 \nu_{1}(T\Gamma)
 &= \int_0^1 \sqrt{1+\left|\frac{\partial \sigma_{1.5}}{\partial t}(t)\right|^2} \cdot h_\nu(\sigma_{1.5}(t), t) \,dt+\int_0^1 \sqrt{1+\left|\frac{\partial \sigma_{3.5}}{\partial t}(t)\right|^2}\cdot h_\nu(\sigma_{3.5}(t), t) \,dt\\
  &=\frac{2}{8}\int_0^1 \sqrt{1+\left|\frac{\partial \sigma_{1.5}}{\partial t}\right|^2}\,dt\\
  &=\frac{2}{8}\int_0^1 \sqrt{1+\sinh^2(2 t)} \,dt=0.4534.
\end{align*}
Thus the $\nu_1$ measure of $T\Gamma$ is almost double that of the $\mu_1$ measure of $\Gamma$.
Correspondingly, the numerator in (\ref{cheeger0}) will be undesirably large.
In Section \ref{sect631} we show how to use our new machinery to find an improved choice for $\Gamma$ that takes into account both the weight $h_\mu$ \textit{and} the dynamics of $T$.

\section{The dynamic isoperimetric problem on weighted manifolds}\label{sec:dcc}

Our goal is to detect Lagrangian coherent structures on the weighted Riemannian manifold $(M, m, \mu_r)$; i.e.\ subsets of $M$ that resist mixing with the surrounding phase space by having persistently small boundary size to internal size.
Following \cite{froyland14}, we introduce a version of the \emph{dynamic isoperimetric problem}, generalised to the situation where the dynamics need not be volume preserving, and occurs on a possibly weighted, possibly curved manifold.

Let $\Gamma$ be a compact $C^\infty$-hypersurface in $M$ that disconnects $M$ into two disjoint open subsets $M_{1}$ and $M_{2}$ with $M_{1}\cup\Gamma\cup M_{2}=M$.
To begin with, we model the dynamics as a single iterate of $T$.
 The subsets $M_1$ and $M_2$ are transformed into $N_1:=TM_1$ and $N_2:=TM_2$, with $T\Gamma$ the disconnecting surface separating $N_1$ and $N_2$ in $N$. Consider the following optimisation problem:
\begin{definition}\label{def:cc}
 Define the \emph{dynamic Cheeger ratio} $\mathbf{H}^D$ by
 \begin{equation}\label{eq:cr}
 \mathbf{H}^D(\Gamma)=\frac{\mu_{r-1}(\Gamma)+\nu_{r-1}(T\Gamma)}{2\min\{\mu_r(M_1),\mu_r(M_2)\}}.
 \end{equation}
 The \emph{dynamic isoperimetric problem} is defined by the optimisation problem
  \begin{equation}\label{eq:cc}
 \mathbf{h}^D=\inf_\Gamma \{\mathbf{H}^D(\Gamma)\},
 \end{equation}
 where $\Gamma$ varies over all $C^\infty$-hypersurfaces in $M$ that partition $M$ into $M=M_1\cup \Gamma \cup M_2$. The number $\mathbf{h}^D$ is called the \emph{dynamic Cheeger constant}.
\end{definition}
Note that by the definition of $\nu_r$, one has $\mu_r(M_1)=\nu_r(N_1)$ and $\mu_r(M_2)=\nu_r(N_2)$.
Importantly, one does \emph{not} have $\mu_{r-1}(\Gamma)=\nu_{r-1}(T\Gamma)$ in general, because $n$ is not necessary the push-forward of $m$ (see also the direct computation in Section \ref{example1}). Thus, one could rewrite \eqref{eq:cr} as
\begin{equation}
\label{eq:cc2}
 \mathbf{H}^D(\Gamma)=\frac{\mu_{r-1}(\Gamma)}{2\min\{\mu_r(M_1),\mu_r(M_2)\}}+\frac{\nu_{r-1}(T\Gamma)}{2\min\{\nu_r(TM_1),\nu_r(TM_2)\}}.
\end{equation}
 By searching over all $C^\infty$-hypersurfaces $\Gamma$ in $M$ to minimise $\mathbf{H}^D(\Gamma)$, the first ratio term of \eqref{eq:cc2} attempts to minimise mixing between the subsets $M_1$ and $M_2$ across the boundary $\Gamma$, through the mechanism of small co-dimensional $1$ mass $\mu_{r-1}(\Gamma)$ at the initial time, and small  co-dimensional $1$ mass $\nu_{r-1}(T\Gamma)$ at the final time.
 Having a persistently small boundary is consistent with slow mixing in the presence of small magnitude diffusion, and is also consistent with measures of mixing adapted to purely advective dynamics such as the mix-norm \cite{MathewMezicPetzold} and negative index Sobolev space norms \cite{Thiffeault_review_nonlinearity2012}.
 The reason for the constraint $\min\{\mu_r(M_1), \mu_r(M_2)\}$ is to ensure that  $M_1$ and $M_2$ found, \emph{both} have macroscopic $r$-dimensional mass to avoid trivial solutions.
 Thus, the optimal solution for \eqref{eq:cc} is a $C^\infty$-hypersurface that represents an excellent candidate for a Lagrangian coherent structure, in the sense that the corresponding subsets $M_1$ and $M_2$ are able to retain their resistance to mixing in the presence of the prescribed dynamics $T$.

  To see why this problem is a truly dynamic problem, consider the $2$-dimensional flat cylinder $[0, 4)/\sim\times[0,1]$ described in Section \ref{example1}.
  The hypersurface $\Gamma=\{x\in M: x_1=1.5, 3.5\}$ partitions $M$ into two disjoint subsets $M_1=(1.5, 3.5)\times [0,1]$ and $M_2=[0, 1.5)\times [0,1]\cup(3.5, 4)\times [0,1]$, forming the partition $M=M_1\cup \Gamma \cup M_2$. It is straightforward to calculate $\mu_2(M_1)=\mu_2(M_2)=0.5$. We note that $\Gamma$ is optimally minimising for the first ratio term of \eqref{eq:cc2}; thus mixing is minimised between $M_1$ and $M_2$. However, under the action of $T$, the $\nu_1$ measure on $T\Gamma$ is almost doubled (from 0.25 to 0.4534). Thus, the sets $M_1$ and $M_2$ are not able to maintain their resistance to mixing, and therefore are poor candidates for LCSs.

\subsection{Multiple discrete time steps and continuous time}\label{sec:mts}

The ``single iterate'' problem described above can easily be extended to multiple discrete time steps or continuous time.
Let $\{(M^i, m^i, \mu_{r, i})\}_{i=1}^k$ be $k$, $r$-dimensional weighted Riemannian manifolds, where each $M^1, M^2, \ldots, M^k$ is $C^\infty$, compact, and connected. For each $1\leq i\leq k$, define co-dimension $1$ measures $\mu_{r-1, i}$ on $M^i$ via the densities $h_i$ of each $\mu_{r, i}$ analogous to \eqref{eq:omega_m^{r-1}}. Let us now consider a composition of several maps $T_1, T_2, \ldots, T_{k-1}$, such that $T_i(M^i)=M^{i+1}$, $T_0$ the identity and $\mu_{r,i}=\mu_{r, i+1}\circ T_i$, for $i=1, 2, \ldots k-1$. Denoting $T^{(i)}=T_i\circ \cdots\circ T_2\circ T_1$, $i=1,\ldots, k-1$. These maps might arise, for example, as time-$\tau$ maps of a time-dependent flow. If we wish to track the evolution of a coherent set under these maps, penalising the boundary of the evolved set $T^{(i)}(\Gamma)$ after the application of each $T_i$, then we can define
\begin{equation}\label{eq:Hdisc}
 \mathbf{H}_k^D(\Gamma):=\frac{\frac{1}{k}\sum_{i=0}^{k-1} \mu_{r-1, i+1}(T^{(i)}\Gamma)}{\min\{\mu_{r, 1}(M_1), \mu_{r, 1}(M_2)\}},
\end{equation}
and consider the time-discrete dynamic optimisation problem
\begin{equation}\label{eq:hdisc}
 \mathbf{h}_k^D:=\inf_{\Gamma} \mathbf{H}^D_k(\Gamma),
\end{equation}
as a natural generalisation of $\mathbf{h}^D$.

 In continuous time, we consider an evolving Riemannian manifold $M(t)$, $t\in [0, \tau]$ under a (possibly time-dependent) ODE $\dot{x}=F(x, t)$, where $F(x, t)$ is $C^\infty$ at each $x\in M(t)$; i.e the initial manifold $M(0)$ is transformed under the smooth flow maps $T^{(t)}:M(0)\to M(t)$ arising from $F$ for each $t\in [0, \tau]$. We denote the Riemannian metric on $M(t)$ by $m^t$, and define absolutely continuous probability measures $\mu_{r, t}$ on $M(t)$ for each $t\in [0, \tau]$; one has an evolving weighted Riemannian manifold $(M(t), m^t, \mu_{r, t})$. Note that the metrics $m^t$ need not be related for different $t$. For all $t\in [0, \tau]$, we assume $\mu_{r, 0}=\mu_{r, t}\circ T^{(t)}$ on $M(t)$. Define
\begin{equation}\label{eq:Hcont}
 \mathbf{H}_{[0,\tau]}^D(\Gamma):=\frac{\frac{1}{\tau}\int_0^\tau \mu_{r-1, t}(T^{(t)}\Gamma)\,dt}{\min\{\mu_{r, 0}(M_1(t)), \mu_{r, 0}(M_2(t))\}},
\end{equation}
and
\begin{equation}\label{eq:hcont}
 \mathbf{h}_{[0,\tau]}^D:=\inf_{\Gamma} \mathbf{H}^D_{[0, \tau]}(\Gamma),
\end{equation}
as a time-continuous generalisation of $\mathbf{h}^D$.
See Section 3.3 \cite{froyland14} for analogous constructions in the unweighted, volume-preserving situation.

\subsection{A dynamic Federer-Fleming theorem on weighted manifolds}\label{sec:dff}
 Our first result on dynamic isoperimetry is the dynamic version of the Federer-Fleming theorem, which links $\mathbf{h}^D$ with a function-based optimisation problem.
The \emph{gradient} of $f\in C^1(M, \mathbb{R})$ denoted by $\nabla_m f$ is a vector field satisfying
\begin{equation}\label{eq:grad}
 m( \nabla_m{f}, \mathcal{V})=\mathcal{V}f,
\end{equation}
for all $\mathcal{V}\in \mathcal{F}^k(M)$.
Following \cite{froyland14}, we define the \emph{dynamic Sobolev constant}:
\begin{definition}
 Define the \textup{dynamic Sobolev constant} $\mathbf{s}^D$ by
 \begin{equation}\label{eq:sd}
  \mathbf{s}^D =\inf_{f} \frac{\int_M |\nabla_m f|_m\, d\mu_r+\int_{N} |\nabla_n (f\circ T^{-1}) |_n\, d\nu_r}{2\inf_{\alpha\in \mathbb{R}} \int_M |f-\alpha|_m\, d\mu_r}
 \end{equation}
 where $f:M\to \mathbb{R}$ varies over all $C^\infty$ functions on $M$.
\end{definition}

The dynamic Sobolev constant $\mathbf{s}^D$ defined above admits the following geometric interpretation: consider the numerator of $\mathbf{s}^D$, one can show (by Lemma \ref{thm:ca} in the appendix) that
\begin{equation*}\label{eq:sdt}
 \int_M |\nabla_m f|_m \,d\mu_r=\int_{-\infty}^\infty \mu_{r-1}(\{f=t\})\,dt,
\end{equation*}
and,
\begin{align*}
 \int_N |\nabla_n \mathcal{L}f|_n \,d\nu_r
 &=\int_{-\infty}^\infty \nu_{r-1}(\{\mathcal{L}f=t\})\,dt.\\
 &=\int_{-\infty}^\infty \nu_{r-1}(T\{f=t\})\,dt,
\end{align*}
where the final equality is due to Proposition \ref{thm:Lf} in the appendix. Furthermore, there is a deep connection between $\mathbf{s}^D$ and the dynamic Cheeger constant $\mathbf{h}^D$. One has

\begin{theorem}[Dynamic Federer-Fleming theorem]\label{thm:dff}
 Let $(M, m, \mu_r)$ and $(N, n, \nu_r)$ be weighted Riemannian manifolds, where $M$ and $N$ are $C^\infty$, compact and connected. Let $T:M\to N$ be a $C^\infty$ diffeomorphism, with $\nu_r=\mu_r\circ T^{-1}$. Assume the density of $\mu_r$ is $C^\infty$ and uniformly bounded away from zero.  Define $\mathbf{h}^D$ and $\mathbf{s}^D$ by \eqref{eq:cc} and \eqref{eq:sd} respectively. Then
 \begin{equation}
 \label{eq:4.11}
 \mathbf{s}^D=\mathbf{h}^D.
 \end{equation}
\end{theorem}

\begin{proof}
The inequality $\mathbf{s}^D\geq\mathbf{h}^D$ is a straightforward modification of the corresponding result in \cite{froyland14} (Theorem 3.1). The other direction is deferred to the appendix.
\end{proof}

Furthermore, in the notation of Section \ref{sec:mts} one can define the continuous time-step dynamic Sobolev constant for continuous-time dynamics by
\begin{equation*}
\mathbf{s}^D_{[0, \tau]}=\inf_f  \frac{\frac{1}{\tau}\int_0^\tau\left(\int_{M(t)}|\nabla_{m^t} (f\circ T^{(-t)})|_{m^t}\,d\mu_{r, t}\right)\,dt}{\inf_\alpha \int_{M(0)} |f-\alpha|\,d\mu_{r, 0}}.
\end{equation*}
Again by the linearity of our construction, it is straightforward to obtain a dynamic Federer-Fleming theorem for continuous-time dynamics; that is $$\mathbf{s}^D_{[0, \tau]}=\mathbf{h}^D_{[0, \tau]}.$$
The proof is obtained by a straightforward modification of the proof of Theorem \ref{thm:dff} analogous to the continuous-time modification in the proof of Corollary 3.3 in \cite{froyland14}.

\section{The dynamic Laplace operator on weighted manifolds}\label{sec:di}

In this section, we further develop the theory of dynamic isoperimetry established for $\mathbb{R}^r$ in \cite{froyland14}, to obtain results that hold on weighted, non-flat Riemannian manifolds $(M, m, \mu_r)$ for non-volume-preserving dynamics. More precisely, we define the dynamic Laplace operator and state and prove dynamic versions of Cheeger's inequality.
The dynamic Laplace operator will be the key object in the computation of solutions of the dynamic isoperimetry problem.

\subsection{The dynamic Laplace-Beltrami operator}\label{sec:dfop}
Classical isoperimetric theory has deep connections with the Laplace-Beltrami operator (see \cite{buser82,chavel84,kawohl03,milman09}). It is well known that one can recover certain geometrical information about a manifold $M$ from the spectrum of Laplace-Beltrami operator \cite{reuter06,rustamov07}. In this work, our domain of interest is a weighted Riemannian manifold $(M, m, \mu_r)$. The dynamics $T$ maps $M$ onto $N=T(M)$.  The geometric properties of $N$ can be drastically different to $M$, and we are motivated to construct an operator on $(M, m, \mu_r)$ whose spectrum reveals important geometric structures on both $M$ and $N$.

For an unweighted Riemannian manifold $M$, the standard Laplace-Beltrami operator is defined as the composition of the divergence with the gradient \cite{chavel84}, with $\nabla_m$ defined by (\ref{eq:grad}).
Let $U\subseteq M$ be open, with $C^\infty$ boundary $\partial U$ and unit normal bundle $\mathbf{n}$ along $\partial U$; i.e.\ for $\mathcal{W}\in \mathcal{F}^k(\partial U)$, $m(\mathcal{W}, \mathbf{n})(x)=0$ for all $x\in \partial U$.
The \emph{divergence} of $\mathcal{V}\in \mathcal{F}^1(M)$, denoted by $\divg_m \mathcal{V}$ is a function satisfying
\begin{equation}\label{eq:div}
\int_U \divg_m \mathcal{\mathcal{V}}\cdot \omega_m^r:=\int_{\partial U}m(\mathcal{V}, \mathbf{n})\cdot \omega_m^{r-1},
 \end{equation}
for all open $U\subseteq M$. The \emph{Laplace-Beltrami} operator acting on a function $f\in C^2(M, \mathbb{R})$ is defined by $\triangle_mf:=\divg_m (\nabla_m f)$.

Recall that in the setting of a weighted Riemannian manifold $(M, m, \mu_r)$, if $h_\mu$ is the density of $\mu_r$, then when computing weighted volumes, the volume form $\omega_m^r$ is scaled by $h_\mu$ at each point in $M$. According to the definition \eqref{eq:grad}, the gradient does not depend on the weight $h_\mu$, because the metric $m$ is independent of $h_\mu$. However, the divergence given by \eqref{eq:div} does depend on $h_\mu$ because it is defined in terms of $\omega_m^r$. We define the \emph{weighted divergence} $\divg_\mu$ of a $\mathcal{V}\in \mathcal{F}^1(M)$ for $(M, m, \mu_r)$ by
 \begin{equation}\label{eq:wdiv}
  \divg_\mu \mathcal{V}:=\frac{1}{h_\mu}\divg_m (h_\mu\mathcal{V}),
 \end{equation}
where the density $h_\mu$ of $\mu_r$ is assumed to be $C^1(M, \mathbb{R})$.
Note that by \eqref{eq:div}
 \begin{equation}\label{eq:wdivth}
  \int_U(\divg_{\mu} \mathcal{V})\cdot h_\mu\omega_m^r= \int_U \divg_m{(h_\mu \mathcal{V})}\cdot\omega_m^r=\int_{\partial U} m( h_\mu \mathcal{V}, \mathbf{n}) \cdot \omega_m^{r-1}=\int_{\partial U} m(\mathcal{V}, \mathbf{n})\cdot h_\mu\omega_m^{r-1}.
 \end{equation}
Hence, the definition \eqref{eq:wdiv} for weighted divergence is analogous to the unweighted version \eqref{eq:div}.

Now as a consequence of \eqref{eq:wdiv} and the well-known fact that $\divg_m(h_\mu\mathcal{V})=\divg_m(\mathcal{V})+m(\nabla_m h_\mu, \mathcal{V})$ (see e.g equation (13) p.3 in \cite{chavel84}), one has the following definition for the \emph{weighted Laplacian} on a weighted Riemannian manifold $(M, m, \mu_r)$:
\begin{equation}\label{def:wlp}
\triangle_\mu f:=\divg_\mu {(\nabla_m f)}=\frac{1}{h_\mu} \divg_m{(h_\mu \nabla_m f)}=\triangle_m f + \frac{m( \nabla_m h_\mu, \nabla_m f)}{h_\mu},
\end{equation}
for all $f\in C^2(M, \mathbb{R})$. Analogous to \eqref{def:wlp}, one forms the weighted Laplacian $\triangle_\nu$ on $N$ with respect to the metric $n$ and density $h_\nu$ for the weighted Riemannian manifold $(N, n, \nu_r)$.

We now describe the construction of the \emph{dynamic} version of $\triangle_\mu$. The crux of the construction is the push-forward and pullback of functions between $L^2(M, m, \mu_r)$ and $L^2(N, n, \nu_r)$. To track the transformation of a function in $L^1(M, V_m)$ under $T$, the standard tool in dynamical systems is the Perron-Frobenius operator $\mathcal{P}:L^{1}(M, V_m)\to L^1(N, V_n)$ given by
\begin{equation}\label{eq:P-F}
\int_U \mathcal{P}h\cdot\omega_n^r=\int_{T^{-1}U} h\cdot\omega_m^r,
\end{equation}
for all measurable $U\subset N$. For a point-wise definition of $\mathcal{P}$, see \eqref{eq:P-F2} in the appendix. Recalling $h_\nu \in L^1(N, V_n)$ is the density of $\nu_r$ with respect to $\omega_m^r$, and the fact that $\nu_r=\mu_r\circ T^{-1}$, one has
\begin{equation}\label{eq:CalP}
\int_U \mathcal{P}h_\mu\cdot\omega_n^r=\int_{T^{-1}U}h_\mu\cdot\omega_m^r=\mu_r(T^{-1}U)=\nu_r(U)=\int_U h_\nu\cdot \omega_n^r,
\end{equation}
for all measurable $U$ in $N$. Therefore, $h_\nu=\mathcal{P}h_\mu$. We define the push-forward operator $\mathcal{L}:L^2(M, m, \mu_r)\to L^2(N, n, \nu_r)$ (from \cite{froyland13}) by
\begin{equation}\label{eq:pfw}
 \mathcal{L}f:=\frac{\mathcal{P}(f\cdot h_\mu)}{h_\nu}.
\end{equation}
\begin{lemma}
\label{lem:welldefL_init}
The operator $\mathcal{L}:L^2(M, m, \mu_r)\to L^2(N, n, \nu_r)$ is well defined, may be expressed as $\mathcal{L}f=f\circ T^{-1}$, and has adjoint $\mathcal{L}^*g=g\circ T$.
\end{lemma}
The proof of this result is given in the appendix (see Lemma \ref{lem:welldefL}).
\begin{definition}\label{def:dwc}
Assume the density of $\mu_r$ is $C^1(M, \mathbb{R})$. Define the \emph{dynamic Laplacian} $\triangle^D:C^2(M, \mathbb{R})\to C^0(M, \mathbb{R})$ by
\begin{equation}\label{eq:dwc}
 \triangle^D:=\frac{1}{2}(\triangle_\mu+\mathcal{L}^*\triangle_\nu \mathcal{L}),
\end{equation}
where the weighted Laplacians $\triangle_\mu, \triangle_\nu$ are given by \eqref{def:wlp}, and $\mathcal{L}$, $\mathcal{L}^*$ are defined above.
\end{definition}

The first term in the RHS of \eqref{eq:dwc} is the weighted Laplacian $\triangle_\mu$ on $f\in C^2(M, \mathbb{R})$. The second term pushes $f$ forward by $\mathcal{L}$ to the function $\mathcal{L}f$. This is then followed by the application of the weighted Laplacian $\triangle_\nu$ to the function $\mathcal{L}f$. The weighted Laplacian $\triangle_\nu$ provides geometric information on the weighted manifold $(N, n, \nu_r)$. The result is finally pulled back to a continuous function on $M$ via $\mathcal{L}^*$. For example, consider the familiar setting of $(M, e, \ell)$, where $M$ is an open subset of $\mathbb{R}^r$, with $\ell$ the Lebesgue measure on $M$ and $e$ the standard Euclidean metric (i.e\ on each point in $M$, $e_{ij}=\delta_{ij}$ for all $1\leq i, j, k\leq r$). If $T:M\to N$ is volume preserving, then in the standard Euclidean coordinates $\{x_i\}_{i=1}^r$ for $M$, and $\{y_i\}_{i=1}^r$ for $N$, one has
\begin{equation*}
 \triangle^Df=\frac{1}{2}\sum_{i=1}^r \left( \frac{\partial^2 f}{\partial x_i^2} +  \frac{\partial^2 (f\circ T^{-1})}{\partial y_i^2}\circ T \right),
\end{equation*}
for all $f\in C^2(M, \mathbb{R})$, where $\triangle^D$ is precisely the definition of the dynamic Laplacian in \cite{froyland14} (where it is denoted by $\hat{\triangle}$).

 Corollary \ref{thm:dwlp2} in the appendix provides an alternate representation of $\triangle^D$:
 \begin{equation}\label{eq:dwlp2}
 \triangle^Df=\frac{1}{2}\left(\triangle_m +\mathcal{L}^*\triangle_n \mathcal{L}  \right)f+\frac{1}{2}\left(\frac{m(\nabla_m h_\mu, \nabla_m f)}{h_\mu}+\frac{n(\nabla_n h_\nu, \nabla_n \mathcal{L}f)\circ T}{h_\nu\circ T}\right).
\end{equation}
The effect of the densities $h_\mu$, $h_\nu$ is completely captured by the terms in the second parentheses of \eqref{eq:dwlp2}.
Finally and importantly,  we have
\begin{proposition}
\label{PropLB}
The operator $\triangle^D$ may be represented as
 \begin{equation}\label{eq:laplace3}
  \triangle^D=\frac{1}{2}(\triangle_\mu+\triangle_{\tilde{\mu}})f,
 \end{equation}
 where $\triangle_{\tilde{\mu}}$ is the weighted Laplace-Beltrami operator on $M$ defined by \eqref{def:wlp} with respect to the metric $T^*n$ and density $\mathcal{L}^*h_\nu=h_\nu\circ T$.
 \end{proposition}
 For the proof, see Corollary \ref{thm:dwlp3} in the appendix.
 We briefly discuss some special cases of Proposition \ref{PropLB}.
 If $(M, m)=(N,n)$, then $\triangle_{\tilde\mu}$ in (\ref{eq:laplace3}) is the weighted Laplace-Beltrami operator on $M$ with respect to the metric $T^*m$ and density $\mathcal{L}^*h_\nu=\frac{h_\mu}{|\det{J_T}|}$, where $J_T$ is the Jacobian matrix associated with $T$ (see \eqref{eq:pdiffT}).
If $N=T(M)\subset \mathbb{R}^d$, $m=n=e$, and $T$ is volume preserving, then $\triangle_{\tilde\mu}$ is the Laplace-Beltrami operator on $M$ with respect to the metric $T^*e$ and density $\mathcal{L}^*h_\nu={h_\mu}$.
Finally, if $h_\mu\equiv 1$ (uniform density) and $T$ is volume preserving, one is in the setting of \cite{froyland14}, and $\triangle_{\tilde\mu}$ in (\ref{eq:laplace3}) is the unweighted Laplace-Beltrami operator with respect to the metric $T^*e$.


\subsection{Continuous time}\label{sec:lts}
We now describe a time-continuous version of (\ref{eq:dwc}).
For the time-continuous case, let $(M(t), m^t, \mu_{r, t})$ be an evolving weighted Riemannian manifold as in Section \ref{sec:mts}, with flow maps $T^{(t)}:M(0)\to M(t)$ arising from a (possible time-dependent) ODE $\dot{x}=F(x, t)$, where $F(x, t)$ is $C^\infty$ at each $x\in M(t)$. We define a time-continuous Perron-Frobenius operator $\mathcal{P}_t:L^1(M(0), \mu_{r, 0})\to L^1(M(t), \mu_{r, t})$ by $\int_{M(t)}\mathcal{P}_t f\cdot \omega_{m^t}^r=\int_{M(0)} f\cdot \omega_{m^0}^r$ for all $t\in [0, \tau]$. One now has the time-continuous push-forward operator $\mathcal{L}_t:L^2(M(0), m^0, \mu_{r, 0})\to L^2(M(t), m^t, \mu_{r, t})$ given by
\begin{equation}
\mathcal{L}_t f:=\frac{\mathcal{P}_t(f\cdot h)}{\mathcal{P}_t h},
\end{equation}
for all $t\in [0, \tau]$, where $h$ is the density of the initial measure $\mu_{r, 0}$.

Define the time-continuous generalisation of $\triangle^D$ as
\begin{equation}\label{eq:lts2}
\triangle^D_{[0, \tau]}f:=\frac{1}{\tau}\int_{0}^\tau \mathcal{L}_t^*\triangle_{\mu, t}\mathcal{L}_tf\,dt,
\end{equation}
where $ \triangle_{\mu, t}$ is the weighted Laplacian given by \eqref{def:wlp}, with respect to the metric $m^t$ and weight $\mathcal{P}_t h$ for each $t\in [0, \tau]$. Furthermore, by a straightforward modification of Corollary \ref{thm:dwlp3} in the appendix, one has
\begin{equation*}
\mathcal{L}^*_t\triangle_{\mu, t}\mathcal{L}_t =\triangle_{\tilde{\mu}, t},
\end{equation*}
for each $t\in [0, \tau]$, where $\triangle_{\tilde{\mu}, t}$ is a weighted Laplacian on $M$ defined by \eqref{def:wlp} with respect to the metric $(T^{(t)})^*(m^t)$ and density $\mathcal{P}_t h\circ T^{(t)}$. Hence, one may express \eqref{eq:lts2} as
\begin{equation}
\label{eq:lapt}
 \triangle_{[0, \tau]}^Df=\frac{1}{\tau}\int_0^\tau \triangle_{\tilde{\mu}, t}f\, dt.
\end{equation}


\subsection{Spectral theory and a dynamic Cheeger inequality on weighted manifolds}\label{sec:di1}

In standard isoperimetric theory for a compact, connected Riemannian manifold $M$, one may use the spectrum of the Laplace-Beltrami operator $\triangle_m$ to reveal geometric information about $M$. Variational properties characterise the spectrum of $\triangle_m$ (see e.g.\ p.13 in \cite{chavel84} or p.210 in \cite{mcowen96}). Extensions of these variational properties, which carry dynamic information, can be developed for dynamic Laplacian on a compact subset of $\mathbb{R}^r$, under volume-preserving dynamics as in Theorem 3.2 in \cite{froyland14}. Here, we generalise Theorem 3.2 in \cite{froyland14} to weighted, non-flat Riemannian manifolds, subjected to non-volume-preserving dynamics. 

 \begin{theorem}\label{thm:spec}
Let $(M, m,\mu_r)$ and $(N, n, \nu_r)$ be weighted Riemannian manifolds, where $M$ and $N$ are $C^\infty$, compact and connected. Let $T:M\to N$ be a $C^\infty$-diffeomorphism such that $\nu_r=\mu_r\circ T^{-1}$. Define $\triangle^D$ and $T^*n$ by \eqref{eq:dwc} and \eqref{eq:pbm} respectively. Consider the eigenvalue problem
\begin{equation}\label{eq:spec1}
\triangle^D\phi=\lambda \phi,
\end{equation}
with initial Neumann-type boundary condition
  \begin{equation}\label{eq:bc}
  m([\nabla_m +\nabla_{T^*n}]\phi, \mathbf{n})(x)=0, \quad\forall x\in \partial M,
  \end{equation}
  where $\mathbf{n}$ is the normal bundle along $\partial M$. Assume the density of $\mu_r$ is $C^\infty$ and uniformly bounded away from zero.
  \begin{enumerate}
 \item{The eigenvalues of $\triangle^D$ form a decreasing sequence $0= \lambda_1> \lambda_2> \lambda_3>\dots$ with $\lambda_k\to -\infty$, as $k\to \infty$.}
 \item{The corresponding eigenfunctions $\phi_1, \phi_2,\ldots$ are  in $C^\infty(M, \mathbb{R})$, the eigenfunction $\phi_1$ is constant,  and eigenfunctions corresponding to distinct eigenvalues are pairwise orthogonal in $L^2(M, m, \mu_r)$}. 
 \item{Let $\langle \cdot,\cdot\rangle_\mu$ denote the inner-product on $L^2(M, m, \mu_r)$, and $|\,\cdot\,|_m=\sqrt{m(\cdot,\cdot)}$ the norm on tangent spaces induced by the metric tensor $m$. Define $S_0=L^2(M, m, \mu_r)$ and $S_k:=\{f\in L^2(M,m,\mu_r): \langle f, \phi_i\rangle_\mu=0\textup{ for } i = 1,\ldots, k\}$, for $k=1, 2, \ldots$, then
   \begin{align}\label{eq:3.7}
     \lambda_k
      &=-\inf_{f\in S_{k-1}}\frac{\int_M |\nabla_m f|_m^2 \, d\mu_r+\int_{N} |\nabla_n\mathcal{L}f|_n^2 \,d\nu_r}{2\int_M f^2\, d\mu_r}\\
      &=-\inf_{f\in S_{k-1}}\frac{\int_M \left(|\nabla_m f|_m^2+|\nabla_{T^*n} f|_{T^*n}^2\right) \,d\mu_r}{2\int_M f^2\, d\mu_r},
   \end{align}
   where $\mathcal{L}$ is given by $\eqref{eq:pfw}$. Moreover, the infimum of \eqref{eq:3.7} is attained by $f=\phi_k$.
 }
 \end{enumerate}
 \end{theorem}
 \begin{proof}
  See appendix.
 \end{proof}


Equation \eqref{eq:3.7} shows that the eigenvalues of $\triangle^D$ take on larger negative values when $|\nabla_m f|_m$ is large with respect to $\mu_r$ and $|\nabla_n\mathcal{L}f|_n$ is large with respect to $\nu_r$.
To obtain $\lambda_k$ close to zero, one needs $f$ and $\mathcal{L}f$ to have low gradient, and particularly in regions of high $\mu_r$ and $\nu_r$ mass, respectively.
Compare this to \eqref{eq:sd} and the display equations below \eqref{eq:sd}, which make connections with level sets of $f$ and the push-forward $\mathcal{L}f$.
Another way to state that $\lambda_k$ is close to zero is to say that one needs the level sets of $f$ and $\mathcal{L}f$ to be not large with respect to $\mu_{r-1}$ and $\nu_{r-1}$, respectively.
This probably means a combination of not being large according to $\omega^{r-1}_m$ and $\omega^{r-1}_n$ (e.g.\ if $M$ is two-dimensional, the level sets are generally a \emph{small number} of \emph{short} curves produced by an $f$ which is not very oscillatory), and avoiding high density areas of $(M,m,\mu_r)$ and $(N,n,\nu_r)$.

The following theorem provides an upper bound on how bad (how large) the average size of an evolving boundary $\Gamma$ can be;  it bounds above the geometric quantity $\mathbf{h}^D$ in terms of $\lambda_2$, the first nontrivial eigenvalue of $\triangle^D$.
The classical ``static'' version of this result, due to Cheeger \cite{cheeger69}, can be intuitively described in terms of heat flow.
Consider heat flow (generated by the Laplace operator) on a solid dumbbell in two dimensions with a narrow neck.
By initialising ``positive heat'' on one side of the dumbbell and ``negative heat'' on the other side, the rate at which the heat flow equilibriates will be slow because of the narrow neck.
The eigenvalue $\lambda_2$ will be close to zero because of this slow equilibriation.
Of course, the narrow neck means that it is possible to very cheaply partition the dumbbell $M$ into two pieces $M_1$, $M_2$, with $\Gamma$ cutting across the neck.
Cheeger  showed that a small $\lambda_2$ \emph{implied} a small $\mathbf{h}^D$ (a cheap way of disconnecting $M$).
Theorem \ref{thm:wci} injects general nonlinear dynamics into these ideas, and extends Theorem 3.2 \cite{froyland14} to weighted manifolds and non-volume-preserving dynamics.
In terms of heat flow, we are effectively averaging the heat flow geometry across the time duration over which our dynamics acts;  see also \cite{karrasch16} for a treatment of metastability using heat flow in Lagrangian coordinates.

\begin{theorem}[Dynamic Cheeger inequality]\label{thm:wci}
Let $(M, m, \mu_r)$ and $(N,n, \nu_r)$ be weighted Riemannian manifolds, where $M$ and $N$ are $C^\infty$, compact and connected. Let $\triangle^D$ and $\mathbf{H}^D$ be defined by \eqref{eq:dwc} and \eqref{eq:cr} respectively. Assume the density of $\mu_r$ is $C^\infty$ and uniformly bounded away from zero. If $\lambda_2$ is the smallest magnitude nonzero eigenvalue of the eigenproblem \eqref{eq:spec1}-\eqref{eq:bc} with eigenfunction $\phi_2$, then
\begin{equation}\label{eq:wci}
\mathbf{h}^D\leq \inf_{t\in (-\infty, \infty)}\mathbf{H}^D(\{\phi_2=t\})\leq 2\sqrt{-\lambda_2}.
\end{equation}
\end{theorem}
\begin{proof}
See appendix.
\end{proof}

By the linearity of our construction with respect to time, it is straightforward to use variational properties to characterise the spectrum of $\triangle^D_{[0, \tau]}$ (see (\ref{eq:lapt})) as in Theorem \ref{thm:spec}. Moreover, by a modification (see Appendix \ref{sec:multip} for details), one can obtain a continuous-time dynamic Cheeger inequality $$\mathbf{h}_{[0, \tau]}^D\leq 2\sqrt{-\lambda_{2, \tau}},$$ where $\lambda_{2,\tau}$ is the second eigenvalue of $\triangle^D_{[0, \tau]}$ defined in (\ref{eq:lapt}). 

To motivate our strategy for obtaining a good feasible solution $\Gamma$ to the minimisation (\ref{eq:cc}) we first note the equivalence of $\mathbf{h}^D$ and $\mathbf{s}^D$ given by Theorem \ref{thm:dff};  indeed in the proof of Theorem \ref{thm:dff} one selects $\Gamma$ from level sets of functions $f$ used in the RHS of (\ref{eq:sd}).
Second, part 3 of Theorem \ref{thm:spec} shows that $\lambda_2$ solves an $L^2$ version of the $L^1$ minimisation in the definition of the Sobolev constant in (\ref{eq:sd}).
In fact, as $L^1$ optimisation is often more difficult than $L^2$ optimisation, one chief reason to introduce the $L^2$ minimisation is to obtain the simple eigenvalue characterisation (\ref{eq:3.7}).
While we cannot easily solve the $L^1$ optimisation of (\ref{eq:sd}) we can solve its $L^2$ version via (\ref{eq:3.7}) to obtain $\lambda_2$ and $\phi_2$.
Defining $\Gamma_t=\{x\in M: \phi_2(x)=t\}$ as the level set corresponding to the value $t$, we search over all level sets of $\phi_2$, selecting the one that gives the lowest value of $\mathbf{H}^D(\Gamma^t)$ as our approximate solution to (\ref{eq:cc}).
Note that for each $t$ we are equivalently inserting $f=\mathbf{1}_{M_{1,t}}$ into (\ref{eq:sd}), where $M_{1,t}=\{x\in M: \phi_2(x)<t\}$, and we are thus \emph{evaluating the best level set according to the $\mathbf{H}^D$ or $L^1$ objective, rather than then $L^2$ objective.
}%
%
This is the content of Algorithm 1 below, which is relatively standard in manifold learning and graph partitioning, and has also been used in \cite{froyland14,FJ15}.
\RestyleAlgo{boxruled}
\LinesNumbered
\begin{algorithm}[ht]
  \caption{Dynamic spectral partitioning\label{alg}}
  Solve the eigenvalue problem $\triangle^D \phi_2=\lambda_2\phi_2$, where $\lambda_2$ is the first non-trivial eigenvalue of $\triangle^D$, with corresponding $C^\infty(M)$ eigenfunction $\phi_2$.

  For each $t\in [\min \phi_2, \max \phi_2]$, partition $M$ into $M=M_{1, t}\cup\Gamma_t\cup M_{2, t}$ via $M_{1,t}=\{x\in M: \phi_2(x)<t\}$, $M_{2, t}=\{x\in M: \phi_2(x)>t\}$, and the $C^\infty$ hypersurface $\Gamma_t=\{x\in M: \phi_2(x)=t\}$.

  Compute $\mathbf{H}^D(\Gamma_t)$ for each $t\in [\min \phi_2, \max \phi_2]$ and extract the optimal $t_0$; the hypersurface $\Gamma_{t_0}$ is an approximate solution to the dynamic isoperimetric problem \eqref{eq:cc}.
\end{algorithm}
\begin{remark}
Algorithm \ref{alg} can be extended to multi-element partitions if one is searching for multiple coherent objects.
Early transfer operator based methods (e.g.\ \cite{dellnitz99,deuflhard00}) proposed the use of a numerical spectral gap as a heuristic for determining the number of almost-invariant sets;  that is, a gap between $\lambda_k$ and $\lambda_{k+1}$ indicates that $k$ is a natural number\footnote{In settings where there is a good functional analytic setup for the transfer operator $\mathcal{P}$, one defines the number of almost-invariant (resp.\ coherent) sets as the number of eigenvalues (resp.\ Lyapunov exponents) outside the essential spectrum \cite{DFS00} (resp.\ \cite{FLQ00}).} of almost-invariant sets to search for.
This idea is commonly used in the transfer operator community and is equally applicable to finite-time coherent sets \cite{froyland13} (where one would look for a gap in the singular value spectrum) and to the dynamic Laplace operator \cite{FJ15}.
Such a heuristic has also been used for eigenvalues of (static) Laplace-Beltrami operators and their discrete graph-based counterparts in manifold learning (see e.g.\ the review \cite{vonluxburg07}), where it is called the eigengap heuristic.
Once an estimate of a natural number $k\ge 1$ of coherent objects has been determined in this way, one embeds the eigenfunctions $\phi_{2},\ldots,\phi_{k+1}$ in ($k-1$)-dimensional Euclidean space, as per e.g.\ \cite{shimalik00}).
One can then employ standard clustering methods to identify $k$ distinct coherent objects $M_1,\ldots,M_k$.
In the case of weighted manifolds, the balancing of the $\mu_r$ measures of the sets $M_1,\ldots,M_k$ is important.
This could be achieved by, for example, weighted fuzzy clustering, analogous to the algorithm in \cite{froyland05}.
\end{remark}

\section{Geometry and probability:  linking finite-time coherent sets with dynamic isoperimetry} \label{sec:3}

We demonstrate that the probabilistic approach for identifying coherent structures in \cite{froyland13} is tightly connected to the dynamic Laplacian given by \eqref{eq:dwc}, extending Theorem 5.1 \cite{froyland14} to the non-volume-preserving, weighted manifold setting.
Let $(M, e, \mu_r)$ and $(N, e, \nu_r)$ be weighted Riemannian manifolds, where $\mu_r, \nu_r$ are absolutely continuous probability measures with respect to the Lebesgue measure $\ell$, $e$ the Euclidean metric, and $M$ a compact, $r$-dimensional subset of $\mathbb{R}^r$.
In this setting  it was shown in \cite{froyland13} that one could apply localised smoothing operators before and after the application of the transfer operator, followed by a normalisation, to obtain an operator $\mathcal{L}_\epsilon$ (see (\ref{eq:3.2}) below), and that the leading sub-dominant singular vectors of the operator $\mathcal{L}_\epsilon$ corresponded to finite-time coherent sets.

Theorem 5.1 in \cite{froyland14} states that if $T$ is volume preserving and $\ell=\mu_r=\nu_r$, then
\begin{equation}\label{eq:pwL_epL_ep}
\lim_{\epsilon\to 0} \frac{(\mathcal{L}_\epsilon^*\mathcal{L}_\epsilon-I)f}{\epsilon^2}(x)=\frac{1}{2}(\triangle_e+\mathcal{L}^*\triangle_e\mathcal{L})f(x),
\end{equation}
for all $x\in M\subset \mathbb{R}^r$, where $\mathcal{L}$ and $\mathcal{L}^*$ are composition with $T^{-1}$ and $T$, respectively, and
$\mathcal{L}_\epsilon^*$ is the adjoint of $\mathcal{L}_\epsilon$ with respect to a weighted inner-product (see \eqref{eq:3.3} below).

In the following, we first generalise the above constructions to a weighted Riemannian manifold setting. We then improve the point-wise convergence \eqref{eq:pwL_epL_ep} to a uniform convergence over all $f\in C^3(M, \mathbb{R})$. Define $Q:\mathbb{R}^+\to\mathbb{R}$ with support in the open interval $(0, 1)$, such that for any vector $x=(x_1, x_2, \ldots, x_r)\in \mathbb{R}^r$
\begin{equation}
\int_{E_1(0)} x_ix_jQ(|x|)d\ell(x)=
\left\{
	\begin{array}{ll}
		0  & \mbox{if } i\neq j \\
		c  & \mbox{if } i=j
		\end{array}\label{eq:diffusion4},
	\right.
\end{equation}
for some fixed constant $c$. For $\epsilon>0$, let $Q_{m, \epsilon}(x, z):=\epsilon^{-r}Q(\textup{dist}_m(x, z)/\epsilon)$ be a family of functions, where $\textup{dist}_m$ is the Riemannian distance function on $M$ with respect to the metric $m$. For open subsets $X\subset X_\epsilon\subseteq M$ and each $\epsilon>0$, define the \emph{diffusion operator} $\mathcal{D}_{X, \epsilon}:L^1(X, V_m)\to L^1(X_\epsilon, V_m)$ by
\begin{equation}\label{eq:diffusion1}
 \mathcal{D}_{X, \epsilon}f(x)=\int_X Q_{m, \epsilon}(x, y)f(y)\cdot \omega_m^r(y),
\end{equation}
for all $x\in M$. If necessary we rescale $\mathcal{D}_{X, \epsilon}$ so that $\mathcal{D}_{X, \epsilon}\mathbf{1}_X=\mathbf{1}_{X_\epsilon}$, where $\mathbf{1}$ is the characteristic function; i.e.\ we assume $\int_0^\epsilon Q(x/\epsilon)dx=\epsilon^r$ for all $\epsilon>0$. One can interpret $\mathcal{D}_{X, \epsilon}$ as a mollifier on $f$, that averages $f$ at the point $x\in X$ over the $\epsilon$-neighbourhood of $x$ according to the distribution $Q$. Similarly for $Y'_\epsilon\subset Y_\epsilon \subseteq N$ we define a local diffusion operator $\mathcal{D}_{Y'_\epsilon,\epsilon}:L^{1}(Y'_\epsilon, V_n)\to L^{1}(Y_\epsilon, V_n)$ by $\mathcal{D}_{Y'_\epsilon, \epsilon}f(x):=\int_{Y'_\epsilon} Q_{n, \epsilon}(x, y)f(y)\cdot \omega_n^r(y)$.

Recall the definition of the Perron-Frobenius operator $\mathcal{P}$ given by \eqref{eq:P-F}. Set $Y'_\epsilon=TX_\epsilon$, one has an advection-diffusion process between $L^1(X, V_m)$ and $L^1(Y_\epsilon, V_n)$, given by the following diagram:
\begin{equation}\label{eq:map}
 L^{1}(X, V_m)\overset{\mathcal{D}_{X,\epsilon}}{\longrightarrow} L^{1}(X_\epsilon, V_m)\overset{\mathcal{P}}{\longrightarrow}
 L^{1}(Y'_\epsilon, V_n)\overset{\mathcal{D}_{Y'_\epsilon,\epsilon}}{\longrightarrow}L^{1}(Y_\epsilon,V_n).
\end{equation}

We form $\mathcal{P}_\epsilon:L^1(X, V_m)\to L^1(Y_\epsilon, V_n)$ according to \eqref{eq:map} via the composition $\mathcal{P}_\epsilon f:=\mathcal{D}_{Y'_\epsilon,\epsilon}\circ\mathcal{P}\circ\mathcal{D}_{X,\epsilon}f$. Normalising $\mathcal{P}_\epsilon$ yields the operator
 \begin{equation}\label{eq:3.2}
 \mathcal{L}_\epsilon f(y):=\restr{\frac{\mathcal{P}_\epsilon(f\cdot h_\mu)}{\mathcal{P}_\epsilon h_\mu}}{y}=\int_X k_\epsilon(x, y)f(x)\,d\mu_r(x),
\end{equation}
where
\begin{equation*}
k_\epsilon(x, y):=\frac{\int_{X_\epsilon}Q_{n, \epsilon}(y, Tz) Q_{m, \epsilon}(z, x)\cdot \omega_m^r(z)}{\int_X\left(\int_{X_\epsilon} Q_{n, \epsilon}(y, Tz) Q_{m, \epsilon}(z, x)\cdot\omega_m^r(z)\right) \, d\mu_r(x)}.
\end{equation*}
Let $h_{\nu_\epsilon}=\mathcal{P}_\epsilon h_\mu$, and define $\nu_{r,\epsilon}:=dh_{\nu_\epsilon}/dV_n$. If $k_\epsilon(x, y)\in L^2(X\times Y_\epsilon, \mu_r\times \nu_{r, \epsilon})$ then $\mathcal{L}_\epsilon:L^2(X, \mu_r)\to L^2(Y_\epsilon, \nu_{r, \epsilon})$ is compact (by Lemma 1 in \cite{froyland13}).

By obvious modification of the arguments in \cite{froyland13}, one can verify that the adjoint operator $\mathcal{L}^*_\epsilon:L^2(Y_\epsilon, \nu_{\epsilon, r})\to L^2(X, \mu_r)$ is given by the composition
\begin{equation}\label{eq:3.3}
 \mathcal{L}^*_\epsilon g  =\mathcal{D}^*_{X,\epsilon}\circ\mathcal{L}^*\circ\mathcal{D}^*_{Y'_\epsilon,\epsilon}g.
\end{equation}
Let $\mathbf{1}$ denote the characteristic function. Note that $\mathcal{L}_\epsilon \mathbf{1}_X=\mathbf{1}_{Y_\epsilon}$ and $\mathcal{L}_\epsilon^* \mathbf{1}_{Y_\epsilon}=\mathbf{1}_X$, hence the leading singular values $\mathcal{L}_\epsilon$ is always $1$, with corresponding left and right singular vectors $\mathbf{1}_X$ and $\mathbf{1}_{Y_\epsilon}$ (by Proposition 2 in \cite{froyland13}).

By construction, with a suitable choice for $Q$ the leading singular value of $\mathcal{L}_\epsilon$ is always $1$, and the second leading singular vector of $\mathcal{L}_\epsilon$ is used to partition $X\subset M$ into finite-time coherent sets in \cite{froyland13}. The operator $\mathcal{L}_\epsilon$ applies local diffusion on $X\subset M$, before and after $X$ is transformed into $Y_\epsilon$ under the action $T$. Similarly, the operator $\mathcal{L}_\epsilon^*$ applies local diffusion on $Y_\epsilon$, before and after $Y_\epsilon\subset N$ is pulled-back to $X\subset M$ under $T$. Therefore, if $X$ contains finite-time coherent sets (and $Y_\epsilon$ contains their images), then there will be a tendency for the boundaries of these coherent sets to be small both before and after advection in order to minimise diffusive mixing through their boundaries.
The reason for adding diffusion is to give compactness of $\mathcal{L}_\epsilon$ acting on $L^2$, ensuring the singular values of $\mathcal{L}_\epsilon$ close to $1$ are isolated, and to detect subsets of $X$ and $Y_\epsilon$ that have small boundary both before and after the application of $T$; see Section $4$ in \cite{froyland13} for details.

An interesting question is ``what happens in the limit $\epsilon\to 0$?''
The composition $\mathcal{L}_\epsilon^*\mathcal{L}_\epsilon$ is approximately the identity for small $\epsilon$, which appears to provide no dynamical information.
However, by subtracting the identify and rescaling by $\epsilon^2$, one can extract the next term in an $\epsilon$ expansion of $\mathcal{L}_\epsilon^*\mathcal{L}_\epsilon$.
The following result generalises Theorem 5.1 in \cite{froyland14} for $\mathbb{R}^r$ to the case of non-flat weighted Riemannian manifolds; subjected to non-volume-preserving dynamics.

\begin{theorem}
\label{thm:3.1}
  Let $(M, m, \mu_r)$ and $(N, n, \nu_r)$ be weighted Riemannian manifolds, where $M$ and $N$ are $C^\infty$, compact and connected. Let $T:M\to N$ be a $C^\infty$ diffeomorphism. Assume $\nu_r=\mu_r\circ T^{-1}$, and the density of $\mu_r$ is $C^3$. Define $\triangle^D$ by \eqref{eq:dwc}, and $\mathcal{L}_\epsilon$ and its adjoint $\mathcal{L}^*_\epsilon$ by \eqref{eq:3.2} and \eqref{eq:3.3} respectively. There exists a constant $c$ such that
  \begin{equation}\label{eq:3.4}
   \lim_{\epsilon\to 0}\left(\sup_{\|f\|_{C^3(M, \mathbb{R})}\leq 1}\left \|\frac{(\mathcal{L}_\epsilon^* \mathcal{L}_\epsilon - I)f}{\epsilon^2}-c\cdot \triangle^D f \right \|_{C^0(M, \mathbb{R})}\right)=0,
 \end{equation}
 where the constant $c$ is as in \eqref{eq:diffusion4}.
\end{theorem}
\begin{proof}
 See appendix.
\end{proof}

As in the analogous result for small magnitude diffusion presented in Theorem 5.1 \cite{froyland14}, one now has a geometric interpretation of finite-time coherent sets considered in \cite{froyland13}.
 Due to Theorem \ref{thm:3.1}, given $\epsilon$ sufficiently small, the action of the operator $\mathcal{L}_\epsilon^*\mathcal{L}_\epsilon-I$ is approximated by the action of the dynamic Laplacian $\triangle^D$.
 Thus, one has a dual interpretation of finite-time coherent sets as defined probabilistically in \cite{froyland13} to minimise global mixing (including now in the weighted, non-volume-preserving situation), and as defined geometrically in \cite{froyland14} and the present paper using the notion of dynamical isoperimetry to force small boundary size under nonlinear dynamics.

\section{Numerical experiments}\label{sec:num}

In this section we use Theorems \ref{thm:spec} and \ref{thm:wci} to compute solutions to the dynamic isoperimetric problem \eqref{eq:cc}. Our examples will showcase Lagrangian coherent structures on weighted domains with non-volume-preserving dynamics. To keep the numerics simple, we do not explicitly model curvature in the examples. We consider $2$-dimensional weighted, flat Riemannian manifolds $(M, e, \mu_2)$ and $(N, e, \nu_2)$, where $M$ and $N$ are $2$-dimensional compact subsets of $\mathbb{R}^2$, and $e$ is the Euclidean metric. We consider measures $\mu_2$ with smooth densities $h_\mu$ that are uniformly bounded away from zero, and nonlinear dynamics $T:M\to N$ such that $\nu_2=\mu_2\circ T^{-1}$.  Before we give the specific details on the $2$-dimensional weighted Riemannian manifolds $(M, e, \mu_2)$, $(N, e, \nu_2)$ and the transformations $T$, we outline the numerical discretisation of the weighted Laplacian $\triangle^D$ defined by \eqref{eq:dwc} and the operator $\mathcal{L}$.
We have employed a very simple low-order method, but in principle any standard operator approximation method can be used instead.
We note that Froyland and Junge \cite{FrJu17} have recently developed higher-order methods with low data requirements to accurately compute the spectrum and eigenfunctions of the dynamic Laplacian and extract the dominant LCSs.


\subsection{Numerical approximation for $\mathcal{L}$ and $\mathcal{L}^*$}\label{sec:nCalL}
To obtain a numerical approximation for $\mathcal{L}$, we start with tracking the time evolution of the density $h_\mu$ under $T$. To achieve this, we numerically estimate the Perron-Frobenius operator $\mathcal{P}$ using Ulam's method \cite{ulam60}. We follow the construction of \cite{froyland10}: partition $M$ and $N$ into the collections of small boxes $\{B_1, \ldots, B_I\}$ and $\{C_1,\ldots, C_J\}$ respectively, and let $P$ be the transition matrix of volume transport between the boxes in $M$ and boxes in $N$ under the action of $T$. We numerically estimate the entries of $P$ by computing
\begin{equation}\label{eq:matrixP}
P_{ij}=\frac{\#\{z_{i,q}\in B_i:T(z_{i, q})\in C_j\}}{\#\{z_{i, q}\in B_i\}},
\end{equation}
where $z_{i, q}$, $q=1,\ldots, Q$ are $Q$ uniformly distributed test points in the box $B_i$. The matrix $P$ is a row-stochastic matrix, where the $(i, j)^{th}$ entry estimates the conditional probability of a randomly chosen point in $B_i$ entering $C_j$ under the application of $T$. The connection between the matrix $P$ and the operator $\mathcal{P}$ is as follows. Denote by $\pi_I:L^1(M, e, V_m)\to \textup{sp} \{\mathbf{1}_{B_1},\ldots,\mathbf{1}_{B_I}\}$ and $\theta_J:L^1(N, e, \nu_r)\to \textup{sp} \{\mathbf{1}_{C_1},\ldots,\mathbf{1}_{C_J}\}$ the orthogonal Ulam projections formed by taking expectations on partition elements. Define $\mathcal{P}_{I,J}:=\theta_J\circ \mathcal{P}$. One has $\mathcal{P}_{I,J}:\textup{sp} \{\mathbf{1}_{B_1},\ldots,\mathbf{1}_{B_I}\}\to \textup{sp} \{\mathbf{1}_{C_1},\ldots,\mathbf{1}_{C_J}\}$, so that $P$ is the matrix representation of $\mathcal{P}_{I,J}$ under left multiplication.

We discretise the density $h_\mu$ of $\mu_r$ to a column vector $\mathbf{u}$ of length $I$, by setting $u_i=\mu_r(B_i)$. If some sets $B_i$ have zero reference measure, then we remove them from our collection as there is no mass to be transported. We therefore assume that $u_i>0$ for all $i=1, \ldots I$. To approximate the density $h_\nu$ of $\nu_r$, we use the fact that $h_\nu=\mathcal{P}h_\mu$ (by \eqref{eq:CalP}). Thus $\mathbf{v}=P^\top \mathbf{u}$ is the numerical approximation of $h_\nu$. We assume $v_j>0$ (if $v_j=0$, then we remove the corresponding sets $C_j$ because they represent $\nu_r(C_j)=0$).

To numerically estimate $\mathcal{L}$ given by \eqref{eq:pfw}, we use the matrix $P$ and the vectors $\mathbf{u}$ and $\mathbf{v}$. In particular, the components of $\theta_J (\mathcal{L}f)$ are approximated by
\begin{equation}\label{eq:dL}
 [\mathcal{L}f]_j\approx\sum_{i=1}^I \frac{P_{ji}(f_i  u_i)}{v_j},
\end{equation}
where $f_i$ are the components of the vector $\mathbf{f}:=\pi_I f$. Define the $I\times J$ matrix $\tilde{P}$ by
\begin{equation}\label{eq:matrixP2}
\tilde{P}_{ij}:=P_{ij}u_i/v_j.
\end{equation}
Then \eqref{eq:dL} is equivalent to $\theta_J (\mathcal{L}f)\approx \tilde{P}^\top \mathbf{f}$; that is the matrix $\tilde{P}$ under left multiplication is the numerical approximation of $\mathcal{L}$. To numerically estimate $\mathcal{L}^*$ from $\mathcal{L}$, we note by definition $\langle  \mathcal{L}f,g\rangle_\mu=\langle f, \mathcal{L}^* g\rangle_\nu$, for all $f\in L^2(M, m, \mu_r)$ and $g\in L^2(N, n, \nu_r)$. Hence,
\begin{equation*}
\sum_{i=1}^I f_i\cdot [\mathcal{L}^*g]_i\cdot u_i
\approx\langle f, \mathcal{L}^*g\rangle_\mu
=\langle \mathcal{L}f, g\rangle_\nu
\approx\sum_{i=1}^I\sum_{j=1}^J \tilde{P}_{ij} f_i\cdot g_j\cdot v_j=\sum_{i=1}^I f_i \cdot \sum_{j=1}^J P_{ij}g_j \cdot u_i,
\end{equation*}
where $[\mathcal{L}^*g]_i$ and $g_j$ are the components of the vectors $\pi_I(\mathcal{L}f)$ and $\theta_J g$ respectively. Therefore, we have
\begin{equation}\label{eq:dL*}
[\mathcal{L}^*g]_i\approx \sum_{j=1}^J P_{ij} g_j.
\end{equation}
The operator $\mathcal{L}^*$ is numerically estimated by the matrix $P$ under right multiplication.

\subsection{Finite-difference estimate for $\triangle^D$}\label{sec:nlp}
To numerically solve the eigenvalue problem $\triangle_\mu f=\lambda f$ on $(M, e, \mu_r)$, we discretise $\triangle_\mu$ using the second equality of \eqref{def:wlp}; that is
\begin{equation}\label{eq:wlpe1}
\triangle_\mu f=\frac{1}{h_\mu}\divg_e(h_\mu \nabla_e f).
\end{equation}
In preparation for the numerical approximations for our $2$-dimensional examples, which will be a rectangle, cylinder or torus, we construct a $K$ by $L$ grid system for $M$. Let $(x_1, x_2)$ be Euclidean coordinates on $M$. We cover $M$ with $I$ grid boxes $\{B_i\}_{i=1}^I$ of uniform size $b_{x_1}\times b_{x_2}$ (one can easily consider the more general case of nonuniform box sizes), and re-index the boxes $\{B_i\}_{i=1}^I$ to $\{B_{k, l}\}_{1\leq k\leq K, 1\leq l\leq L}$, indexing the $x_1$-direction with $k$, and the $x_2$-direction with $l$; clearly $K\times L=I$. Let $f_{k, l}$ and $\mu_{k,l}$ denote the components of discrete functions $\mathbf{f}$ and $\mathbf{u}$ respectively.

We employ standard finite-difference schemes to obtain numerical approximations for the RHS of \eqref{eq:wlpe1}. Starting with the approximation of $h_\mu\nabla_e f$, one has in Euclidean coordinates $(x_1, x_2)$, the vector $h_\mu\nabla_e f=h_\mu(\partial f/\partial x_1, \partial f/\partial x_2)$. To compute the derivatives $\partial f/\partial x_1$ and $\partial f/\partial x_2$ numerically, we apply the standard central-difference technique to obtain on the grid box $B_{k,l}$,
\begin{equation*}
 \frac{\partial f}{\partial x_1}\approx \frac{f_{k+1, l}-f_{k-1, l}}{2b_{x_1}}\quad\mbox{and}\quad  \frac{\partial f}{\partial x_2}\approx\frac{f_{k, l+1}-f_{k, l-1}}{2b_{x_2}},
\end{equation*}
thus on the grid box $B_{k,l}$
\begin{equation}\label{eq:appgrad}
 h_\mu\nabla_ef\approx  \left( u_{k,l}\frac{f_{k+1, l}-f_{k-1, l}}{2b_{x_1}}, u_{k,l}\frac{f_{k, l+1}-f_{k, l-1}}{2b_{x_2}}  \right).
\end{equation}

Next, we numerically solve the divergence $\divg_e$ applied to the RHS of \eqref{eq:appgrad}. By central-difference approximations, one has on the grid box $B_{k, l}$
\begin{align}\label{eq:appdivgrad}
  \triangle_\mu f=\frac{1}{h_\mu}(\divg_e (h_\mu\nabla_ef)) \approx \frac{1}{u_{k,l}}\left[ u_{k+1,l}\frac{f_{k+2, l}-f_{k, l}}{4b^2_{x_1}}-u_{k-1,l}\frac{f_{k-2, l}-f_{k, l}}{4b^2_{x_1}}\right.\notag\\
  \left.+u_{k,l+1}\frac{f_{k, l+2}-f_{k, l}}{4b_{x_2}^2}-u_{k,l-1}\frac{f_{k, l}-f_{k, l-2}}{4b_{x_2}^2}\right].
\end{align}
Denote the resulting finite-difference approximation of $\triangle_\mu$ by the $I\times I$ matrix $\mathbf{L}_\mu$. Rearranging \eqref{eq:appdivgrad}, then $\mathbf{L}_\mu$ applied to the vector
\begin{equation}\label{eq:approxf}
\mathbf{f}:=(f_{1, 1}, f_{2, 1}, \ldots, f_{K, 1}, f_{1, 2}, f_{2, 2}\ldots, f_{1, L}, f_{2, L}\ldots, f_{K, L}),
\end{equation}
is a vector of length $I$ with components
\begin{align}\label{eq:dlp}
 [\mathbf{L}_\mu \mathbf{f}]_{k+K(l-1)} = \frac{1}{4b_{x_1}^2}\frac{u_{k+1, l}}{u_{k,l}}f_{k+2, l}+\frac{1}{4b_{x_1}^2}\frac{u_{k-1, l}}{u_{k,l}}f_{k-2, l}+\frac{1}{4b_{x_2}^2}\frac{u_{k, l+1}}{u_{k, l}}f_{k, l+2}
 +\frac{1}{4b_{x_2}^2}\frac{u_{k, l-1}}{u_{k, l}}f_{k, l-2}\notag\\
 -\bigg(\frac{1}{4b_{x_1}^2}\frac{u_{k+1, l}+u_{k-1, l}}{u_{k, l}}+ \frac{1}{4b_{x_2}^2}\frac{u_{k, l+1}+u_{k, l-1}}{u_{k, l}} \bigg)f_{k, l},
\end{align}
for $1\leq k\leq K$, $1\leq l\leq L$. Note that if $u_{k,l}$ is constant for all $1\leq k\leq K$ and $1\leq l\leq L$, then the expression \eqref{eq:dlp} becomes the standard $5$-point stencil Laplace matrix.

To treat the numerical approximation of $\triangle_\mu$ at the boundary of $M$, we apply the usual Neumann boundary condition $e(\nabla_e \varphi, \mathbf{n})_x=0$ for all $x\in \partial M$ (where $\mathbf{n}$ is unit normal to $\partial M$). This Neumann boundary condition is imposed by symmetric reflection \cite{strikwerda04} in the above modified finite-difference scheme as follows: consider the grid boxes $B_{1, l}$ for $1\leq l\leq L$; one has a boundary on the left side edge of each of these grid boxes. By construction, the unit normal $\mathbf{n}$ along the left side edge of the grid boxes $\{B_{1, l}\}_{l=1}^{L}$ is given by $(-1,0)$. Therefore, the boundary condition $e(\nabla_e \varphi, \mathbf{n})_x=0$ is satisfied by reflecting the artificial $f_{0, l}=f_{2,l}$, $f_{-1, l}=f_{1, l}$ and $u_{0, l}=u_{2, l}$ for all $1\leq l\leq L$. One applies similar symmetric reflections to all $B_{k,l}$ at the boundary of $M$.

By definition \eqref{eq:dwc}, and the numerical approximations we obtained for $\triangle_\mu$, $\triangle_\nu$, $\mathcal{L}$ and $\mathcal{L}^*$, one has the finite-difference approximation for the weighted dynamic Laplacian $\mathbf{L}^D$ given by
\begin{equation}\label{eq:discretelaplace}
\mathbf{L}^D=\mathbf{L}_\mu+P\mathbf{L}_\nu\tilde{P}^\top,
\end{equation}
where the matrices $P$ and $\tilde{P}$ are given by \eqref{eq:matrixP} and \eqref{eq:matrixP2} respectively, and $\mathbf{L}_\mu$, $\mathbf{L}_\nu$ by \eqref{eq:dlp}. We note that the matrices $\mathbf{L}_\mu, \mathbf{L}_\nu, P$ and $\tilde{P}$ are sparse and consequently $\mathbf{L}^D$ is sparse. One can numerically solve the finite dimensional eigenvalue problem $\mathbf{L}^D\mathbf{f}=\lambda \mathbf{f}$ for small eigenvalues $\lambda$, and in particular $\lambda_2$ and corresponding eigenfunction $\phi_2$. To find a good solution $\Gamma$ to the dynamic isoperimetric problem \eqref{eq:cc}, one can use the level sets of $\phi_2$ as candidates for $\Gamma$ as in Algorithm \ref{alg}.

\subsection{Case study 1: dynamics on a cylinder}\label{sec:nwc}
We now demonstrate our technique on a weighted $2$-dimensional cylinder $(M, e, \mu_2)$, where $M=[0, 4)/\sim\times[0,1]$ and $h_\mu(x_1, x_2)=\frac{1}{8}(\sin(\pi x_1)+2)$ as in Section \ref{example1}. We set our computational resolution for $M$ to be $K\times L=256\times 64$ square grid boxes $B_{k,l}$ of side length $b=1/64$, and select the number of test points in each grid box to be $Q=400$\footnote{One can also use far fewer points per box and still obtain good results.}. We consider two different types of nonlinear transformations $T_1$ and $T_2$ acting on $M$:
\begin{align}
 T_1(x_1, x_2)&=\left(x_1+\frac{\cosh\left(2 x_2\right)-1}{2}, x_2\right),\label{eq:T_1}\\
 T_2(x_1, x_2)&=\left(x_1+x_2, x_2+0.1x_2\sin(2\pi x_2)\right)\label{eq:T_2},
\end{align}
where the first coordinate is computed modulo $4$ in both cases. The map $T_1$ is the area-preserving, non-linear horizontal shear from the example considered in Section \ref{example1}. The map $T_2$ is a linear horizontal shear, composed with vertical area-distortion; i.e.\
\begin{equation*}
T_2(x_1, x_2)=\hat{T}_2(x_1+x_2, x_2),
\end{equation*}
where
\begin{equation*}
\hat{T}_2(x_1, x_2)=(x_1, x_2+0.1x_2\sin(2\pi x_2)),
\end{equation*}
compresses the mass distribution of $M$ in towards the horizontal line $x_2=0.5$.
\subsubsection{The transformation $T_1$ on $M$}
\label{sect631}
We optimally partition $M$ using Algorithm \ref{alg} to find a good solution $\Gamma$ to the dynamic isoperimetric problem \eqref{eq:cc}. First, we consider the dynamic Laplacian for $T_1$ acting on $M$. In step $1$ of Algorithm \ref{alg}, we construct the matrix $\mathbf{L}^D$ given by \eqref{eq:discretelaplace} as the numerical approximation of $\triangle^D$ via the finite-difference scheme outlined in Section \ref{sec:nlp}, and numerically solve the finite-dimensional eigenproblem $\mathbf{L}^D\phi=\lambda \phi$. The leading numerical eigenvalues $\lambda_1$, $\lambda_2,\ldots \lambda_7$ of $\mathbf{L}^D$ are $0$, $-0.6046$, $-1.3739$, $-2.3221$, $-3.2886$, $-3.4091$, $-3.7056\ldots$.  The components of the numerical eigenvector $\phi_2$ corresponding to $\lambda_2$ takes on at most $256\times 64$ unique values; at most one value on each of the grid box. Step $2$ of Algorithm \ref{alg}, generates partitions of $M=M_{1, t}\cup \Gamma_t\cup M_{2, t}$ from level sets of $\phi_2$. Finally, one computes $\mathbf{H}^D(\Gamma_t)$ for each $t$, and finds the optimal $\Gamma_{t_0}$ as a solution to the dynamic optimisation problem \eqref{eq:cc}; the results are shown in Figure \ref{fig:dyn1}.

\begin{figure}[h]
         \begin{subfigure}[h]{0.5\textwidth}
                 \includegraphics[width=\textwidth]{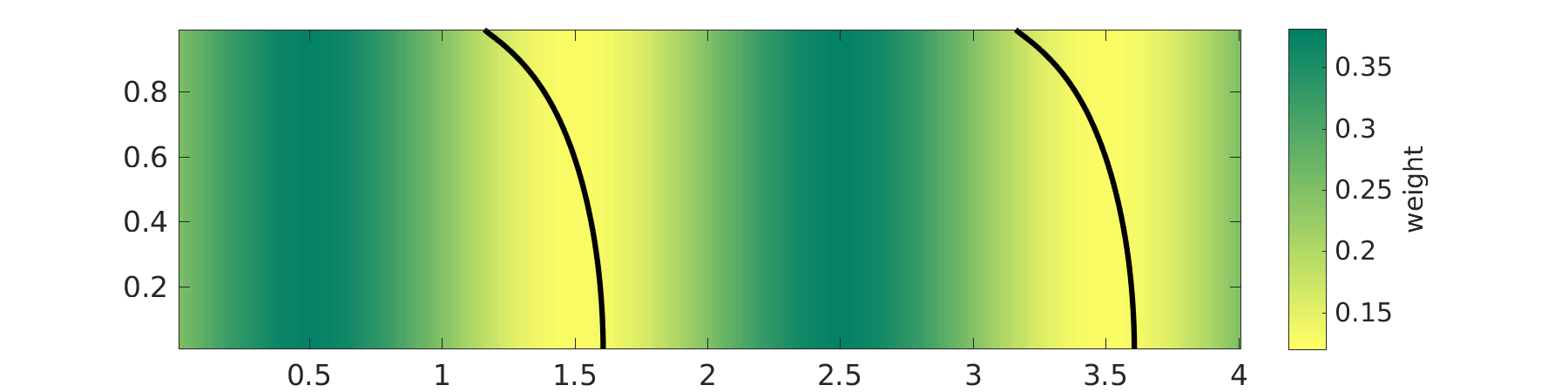}
                 \caption{}\label{fig:dyn1a}
         \end{subfigure}
         \begin{subfigure}[h]{0.5\textwidth}
                 \includegraphics[width=\textwidth]{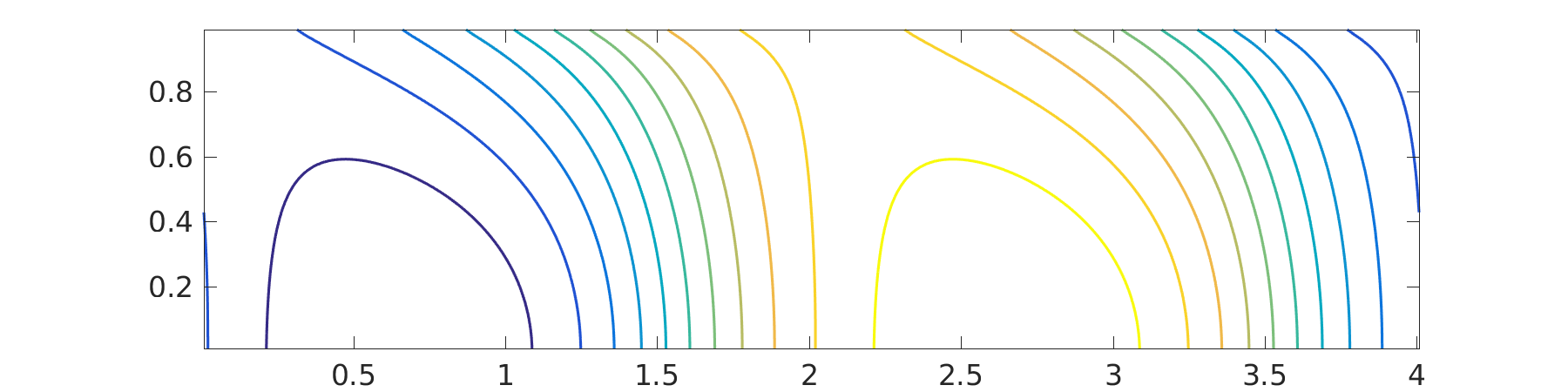}
                 \caption{}\label{fig:dyn1b}
         \end{subfigure}
         \begin{subfigure}[h]{0.5\textwidth}
                 \includegraphics[width=\textwidth]{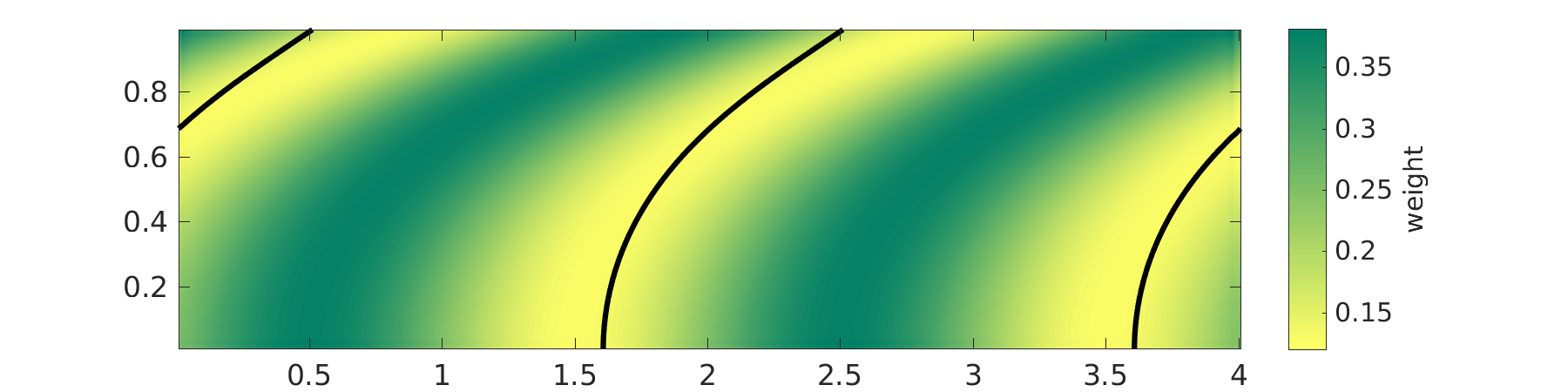}
                 \caption{}\label{fig:dyn2a}
         \end{subfigure}
         \begin{subfigure}[h]{0.5\textwidth}	
                 \includegraphics[width=\textwidth]{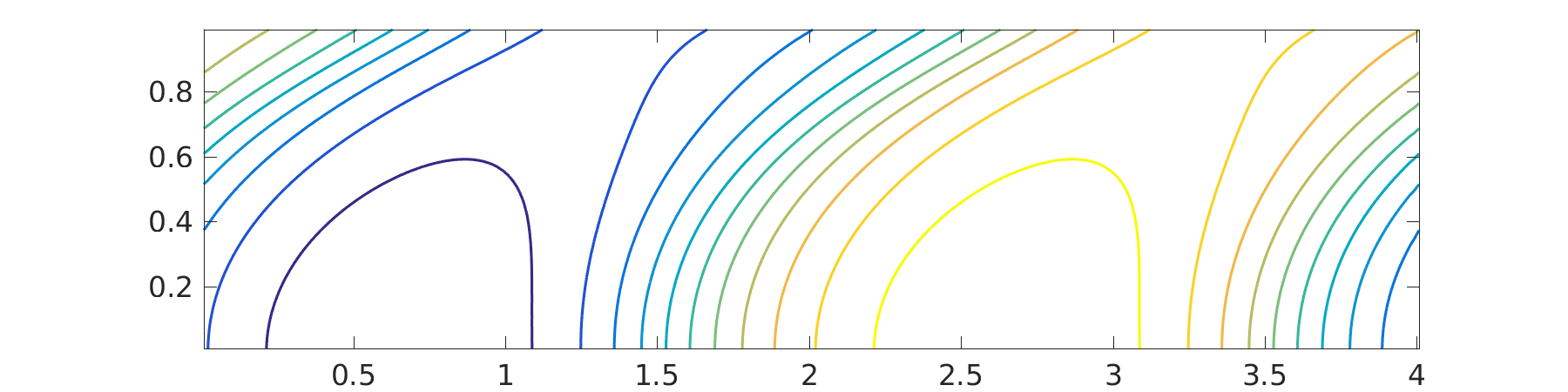}
                 \caption{}\label{fig:dyn2b}
         \end{subfigure}
         \caption{Partition of $M$ using the eigenvector $\phi_2$ of $\triangle^D$ under non-linear shear $T_1$ given by \eqref{eq:T_1}. (a) Colours are values of $h_\mu$, and black lines are the level surface $\Gamma_{t_0}=\{\phi_2=-1.211\times 10^{-5}\}$. (b) The level surfaces of $\phi_2$. (c) Colours are the values of $h_\nu$, and black lines are the level surface $T_1\Gamma_{t_0}$. (d) The level surfaces of $\mathcal{L}\phi_2$.}\label{fig:dyn1}
\end{figure}

It was found that the hypersurface $\Gamma_{t_0}$ is $\mathbf{H}^D$ \emph{minimising} for $t_0=-1.211\times 10^{-5}$; see figures \ref{fig:dyn1a} and \ref{fig:dyn2a}. Note that the densities $h_\mu$ form a region of low $\mu_1$-mass in $M$ about the lines $\{x\in M: x_1=1.5, 3.5\}$. Thus, to minimise the $\mu_1$-mass of the hypersurface $\Gamma_{t_0}$ in $M$, it is advantageous to have $\Gamma_{t_0}$ as short curves in close proximity to the vertical lines $\{x\in M: x_1=1.5, 3.5\}$. Moreover, to effectively counter the shearing imposed by $T_1$ so that the size of $\nu_1$-mass of $T_1\Gamma_{t_0}$ stays persistently small in $TM$, the curve $\Gamma_{t_0}$ bends horizontally towards the left progressively more as $x_2$ approaches $1$ from $0$.  The $\mu_1$-mass of $\Gamma_{t_0}$ and its image under $T_1$ are $\mu_1(\Gamma_{t_0})=0.3088$ and $\nu_1(T_1\Gamma_{t_0})=0.3815$; the level surface $\Gamma_{t_0}$ experiences significantly reduced deformation under the action of $T_1$, compared to the results of Section \ref{example1} shown by Figure \ref{fig:static2a}. Moreover, the partition $M=M_{1,t_0}\cup \Gamma_{t_0}\cup M_{2, t_0}$ has a perfectly balanced $\mu_2$-mass distribution between $M_{1, t_0}$ and $M_{2,t_0}$.  One has $\mathbf{H}^D(\Gamma_{t_0})=0.6903$, and this solution is a suitable candidate for LCSs on $(M, e, \mu_2)$.

\subsubsection{The transformation $T_2$ on $M$}
We repeat the above numerical experiment on the $2$-cylinder, replacing the transformation $T_1$ with $T_2$, and setting the initial mass density $h_\mu$ to be uniformly distributed on $M$. The leading numerical eigenvalues $\lambda_1$, $\lambda_2,\ldots \lambda_6$ of $\mathbf{L}^D$ are $0, -0.7747\pm 0.0092i, -3.0900\pm 0.0199i, 3.8702, -4.5674\pm 0.0250i$.
In this example, although $\mathbf{L}_\mu$ and $\mathbf{L}_\nu$ have real eigenvalues to numerical precision, when combined to form $\mathbf{L}^D$, one obtains small imaginary parts.
The eigenvalues $\lambda_2, \lambda_3$ should be real and equal (i.e.\ $\lambda_2$ has multiplicity 2), because of the symmetry obtained by translating all objects in the $x$-coordinate direction.
As before, we apply Algorithm \ref{alg} to partition $M$ using the level surfaces of the second eigenfunction $\phi_2$ corresponding to $\lambda_2=-0.7747$; the results are shown in Figure \ref{fig:dyn4}.
\begin{figure}[h]
         \begin{subfigure}[h]{0.5\textwidth}
                 \includegraphics[width=\textwidth]{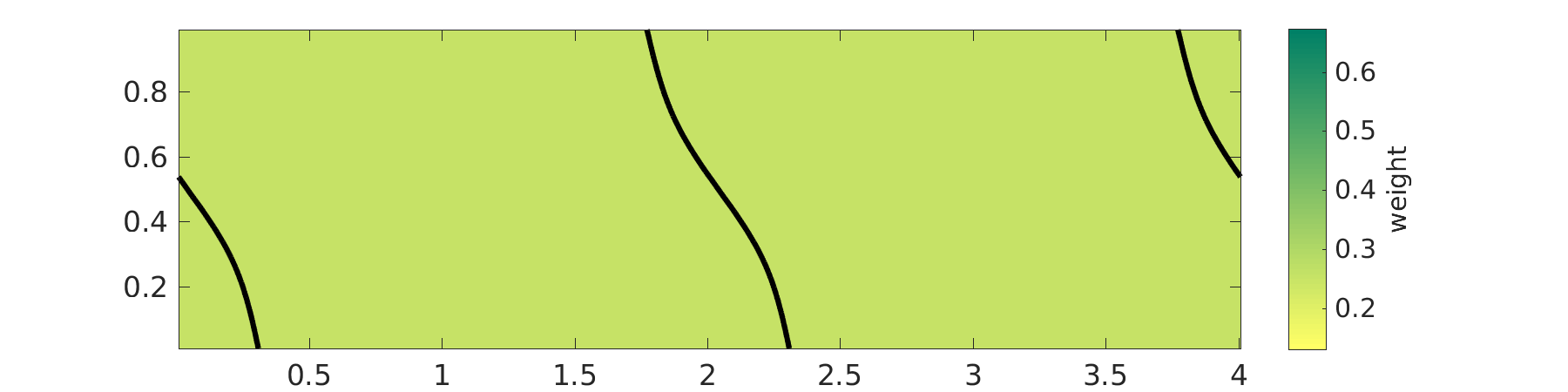}
                 \caption{}\label{fig:dyn3a}
         \end{subfigure}
         \begin{subfigure}[h]{0.5\textwidth}
                 \includegraphics[width=\textwidth]{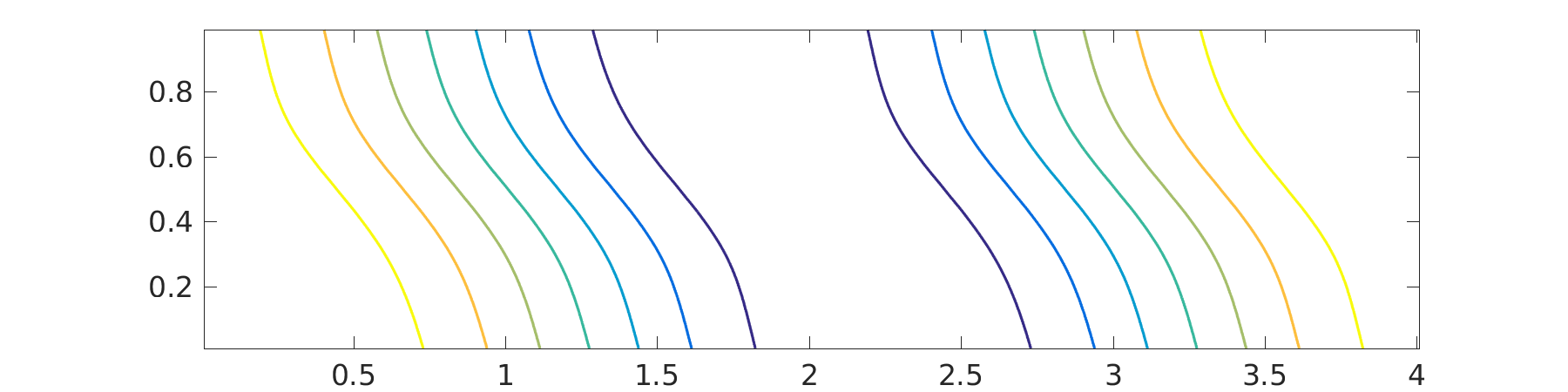}
                 \caption{}\label{fig:dyn3b}
         \end{subfigure}
         \begin{subfigure}[h]{0.5\textwidth}
                 \includegraphics[width=\textwidth]{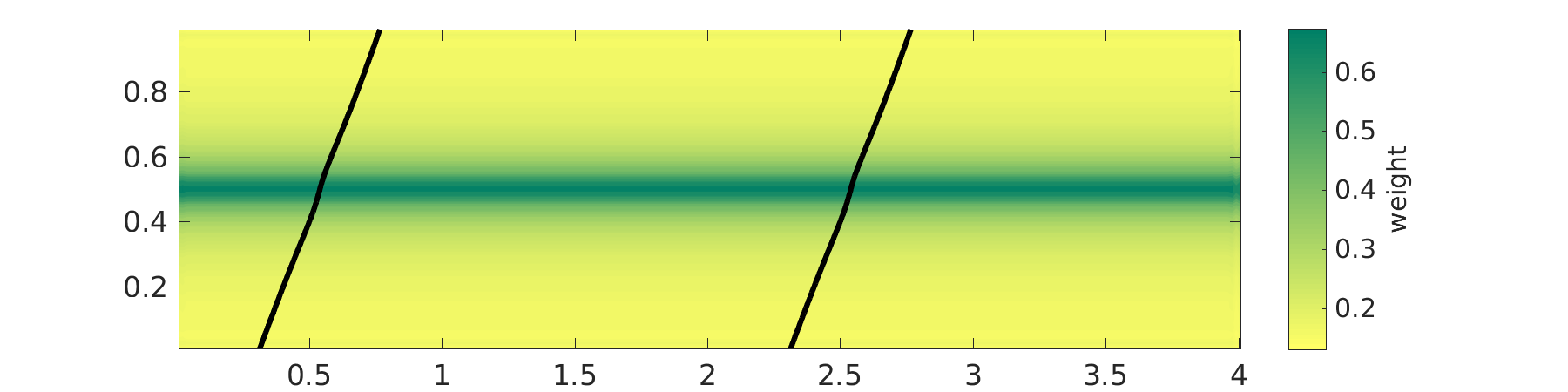}
                 \caption{}\label{fig:dyn4a}
         \end{subfigure}
         \begin{subfigure}[h]{0.5\textwidth}	
                 \includegraphics[width=\textwidth]{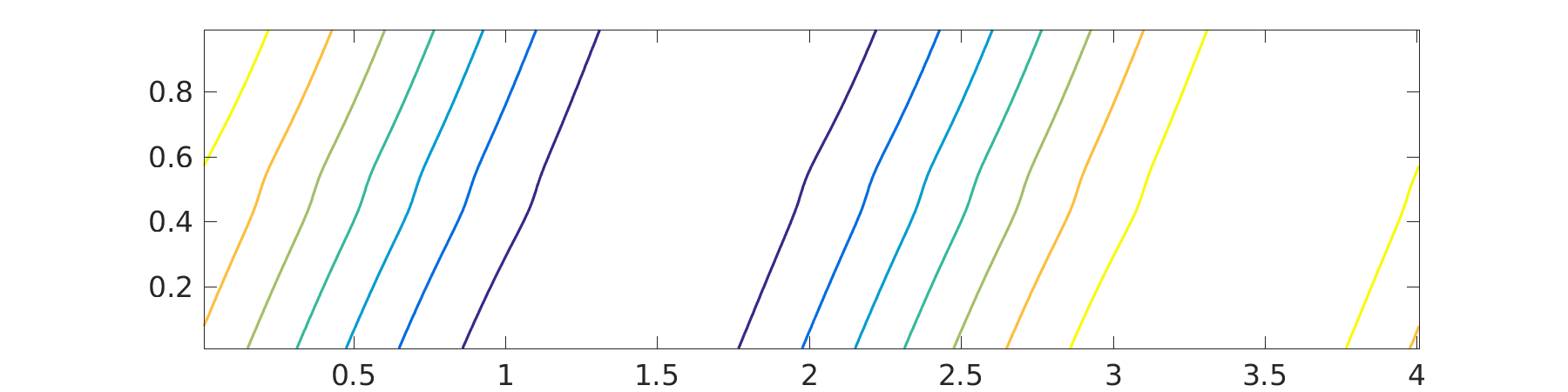}
                 \caption{}\label{fig:dyn4b}
         \end{subfigure}
       \caption{Partition of $M$ using the eigenvector $\phi_2$ of $\triangle^D$ for the nonlinear shear $T_2$ given by \eqref{eq:T_2}. (a) Colours are values of $h_\mu$, and black lines are the level surface $\Gamma_{t_0}=\{\phi_2=-3.4232\times 10^{-5}\}$. (b) The level surfaces of $\phi_2$. (c) Colours are the values of $h_\nu$, and black lines are the level surface $T_2\Gamma_{t_0}$. (d) The level surfaces of $\mathcal{L}\phi_2$.}\label{fig:dyn4}
\end{figure}

It was found that the hypersurface $\Gamma_{t_0}$ is $\mathbf{H}^D$ \emph{minimising} for $t_0=-3.4232\times 10^{-5}$; see figures \ref{fig:dyn3a} and \ref{fig:dyn4a}. The co-dimension $1$ mass of $\Gamma_{t_0}$ and its image under $T_2$ are $\mu_1(\Gamma_{t_0})=0.5682$ and $\nu_1(T_2\Gamma_{t_0})=0.5435$, respectively. Recall that the action of $T_2$ on $M$ has the effect of compressing the mass distribution towards the horizontal line $x_2=0.5$. To avoid a large $\nu_1$-mass of $T_2\Gamma_{t_0}$ on $T_2M$, one makes appropriate compromises on the $\mu_1$-mass of $\Gamma_{t_0}$ in $M$. For example, in Figure \ref{fig:dyn4a} as the black curves approach the dark green, high density region, they become more vertical so as to traverse this high density region using a shorter curve length and reducing their $\nu_1$-mass. This necessitates $\Gamma_{t_0}$ being slightly curved and having slightly greater $\mu_1$-mass. Once again the partition $M=M_{1,t_0}\cup \Gamma_{t_0}\cup M_{2, t_0}$ has a perfectly balanced $\mu_2$-mass distribution of $M_{1, t_0}$ and $M_{2,t_0}$.  One has $\mathbf{H}^D(\Gamma_{t_0})=1.1117$.

\subsection{Case study 2: dynamics on a torus}\label{sec:nwc2}

A dynamic isoperimetric partition of the classical so-called standard map $(x,y)\mapsto(x+y,y+8\sin(x+y))\pmod{2\pi}$ on the 2-torus was calculated in \cite{froyland14}.
Despite the strong nonlinearity, the lines of constant $x+y$ value are clearly mapped to lines of constant $x$ value, and this fact was automatically exploited by the eigenfunctions of the dynamic Laplace operator to obtain the optimal solution to the dynamic isoperimetric problem (see Figures 7 and 8 \cite{froyland14}).

To demonstrate the effects of lack of volume-preservation and nonuniform weights, we now modify the standard map to a version that is not volume-preserving, and replace the uniform weighting on the torus with a nonuniform weighting.
We consider a weighted $2$-dimensional torus $(\mathbb{T}^2, e, \mu_2)$, where $\mathbb{T}^2=2\pi(\mathbb{R}/\mathbb{Z})\times 2\pi(\mathbb{R}/\mathbb{Z})$ and $h_\mu(x_1, x_2)=\frac{1}{8\pi^2}(\sin(x_2-\pi/2)+2)$.
We consider the transformation $T:=T_4\circ T_3$ acting on $M$, where
\begin{align}
 T_3(x_1, x_2)&=\left(x_1+0.3x_1\cos(2x_1), x_2\right),\label{eq:T_3}\\
 T_4(x_1, x_2)&=\left(x_1+x_2, x_2+8\sin(x_1+x_2)\right)\label{eq:T_4},
\end{align}
computed modulo $2\pi$. The map $T_3$ distorts the area of $\mathbb{T}^2$ in the horizontal direction, and $T_4$ is the ``standard map''.

We set our computational resolution for $M$ to be $K\times L=128\times 128$ square grid boxes $B_{k,l}$ of side length $b=1/64$, and select the number of test points in each grid box to be $Q=400$.
We optimally partition $M$ using Algorithm \ref{alg}. The leading numerical eigenvalues $\lambda_1$, $\lambda_2,\ldots \lambda_7$ of $\mathbf{L}^D$ are $0$, $-0.3584$, $-0.3751$, $-1.0750$, $-1.1349$, $-1.4358$, $-1.4966$. We generate partitions of $M$ using the level surfaces of the second eigenfunction $\phi_2$ corresponding to $\lambda_2=-0.3584$; the results are shown in Figure \ref{fig:dyn5}.

\begin{figure}[h]
\centering
         \begin{subfigure}[h]{0.24\textwidth}
                 \includegraphics[width=\textwidth]{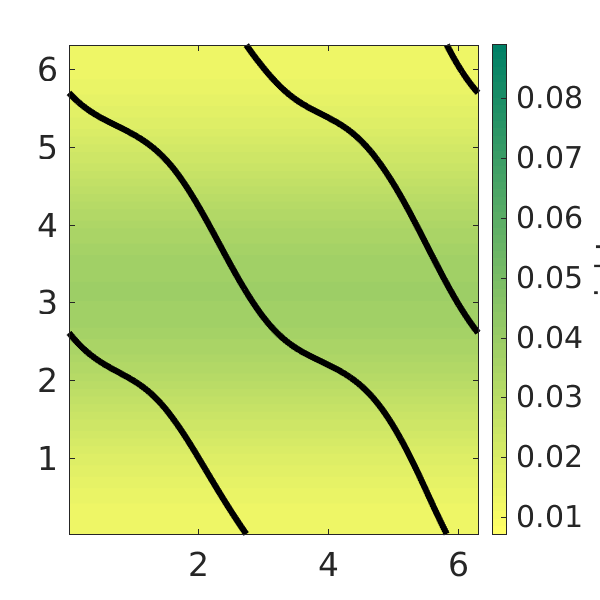}
                 \caption{}\label{fig:dyn5a}
         \end{subfigure}
         \begin{subfigure}[h]{0.24\textwidth}
                 \includegraphics[width=\textwidth]{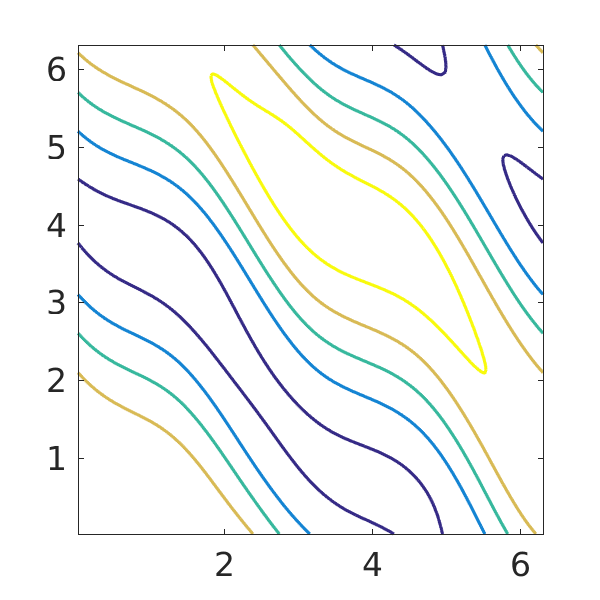}
                 \caption{}\label{fig:dyn5b}
         \end{subfigure}
         \begin{subfigure}[h]{0.24\textwidth}
                 \includegraphics[width=\textwidth]{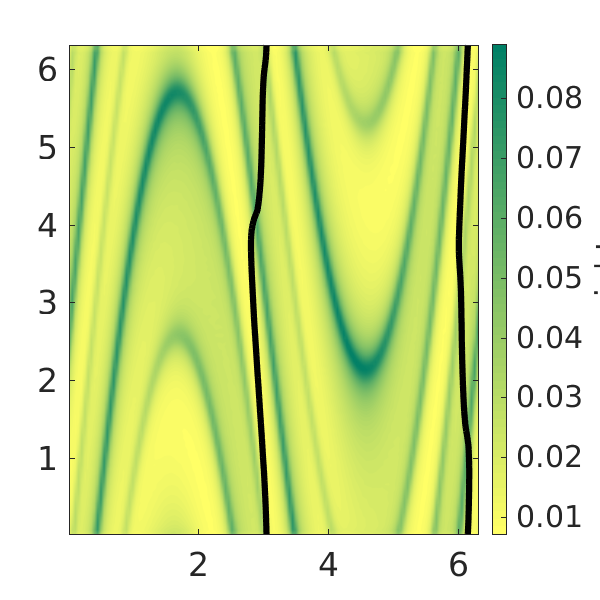}
                 \caption{}\label{fig:dyn6a}
         \end{subfigure}
         \begin{subfigure}[h]{0.24\textwidth}	
                 \includegraphics[width=\textwidth]{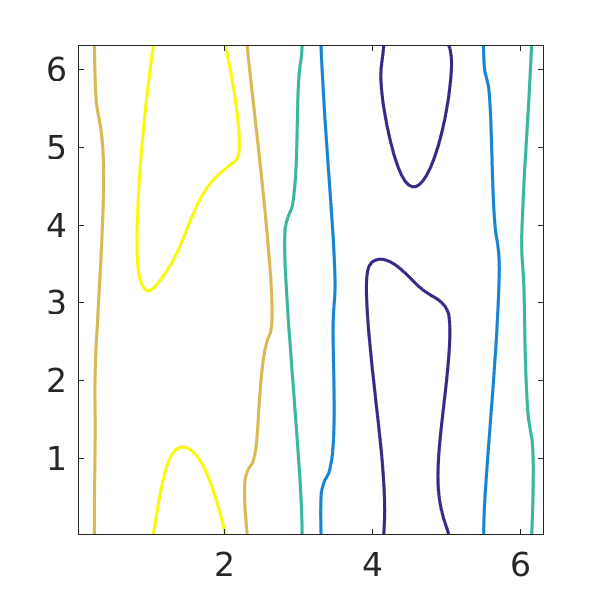}
                 \caption{}\label{fig:dyn6b}
         \end{subfigure}
       \caption{Partition of $\mathbb{T}^2$ using the eigenvector $\phi_2$ of $\triangle^D$ under $T=T_4\circ T_3$ given by \eqref{eq:T_3} and \eqref{eq:T_4}. (a) Colours are values of $h_\mu$, and black lines are the level surface $\Gamma_{t_0}=\{\phi_2=-4.5492\times 10^{-5}\}$. (b) The level surfaces of $\phi_2$. (c) Colours are the values of $h_\nu$, and black lines are the level surface $T\Gamma_{t_0}$. (d) The level surfaces of $\mathcal{L}\phi_2$.}\label{fig:dyn5}
\end{figure}

It was found that the hypersurface $\Gamma_{t_0}$ is $\mathbf{H}^D$ \emph{minimising} for $t_0=-4.5492\times 10^{-5}$; see figures \ref{fig:dyn5a} and \ref{fig:dyn6a}.
The black curves in these figures approximately follow lines of constant $x+y$ value and constant $x$ value, respectively, (as was the case in Figure 7 \cite{froyland14}).
However, here the curves are additionally optimised to take into account the nonuniform $\mu_2$ and non-volume preserving nature of the dynamics, which we explain below.

The $\mu_1$-mass on $\Gamma_{t_0}$ and the $\nu_1$-mass of its image under $T$ are $\mu_1(\Gamma_{t_0})=0.4584$ and $\nu_1(T\Gamma_{t_0})=0.2375$, respectively. Similar to the results of the previous case study of $T_2$ acting on $M$, one makes appropriate compromises on the $\mu_1$-mass of $\Gamma_{t_0}$ in $\mathbb{T}^2$, to ensure that the $\nu_1$-mass of $T\Gamma_{t_0}$ in $\mathbb{T}^2$ remains small. For example, in Figure \ref{fig:dyn6a} the black (almost straight) curves attempt to follow the yellow, low density regions, but when they have to cross the dark green, high density regions, the curves briefly turn to cross these high density regions at a sharper angle. While slightly increasing the curve length, this behaviour reduces the $\nu_1$-mass of the curves.

The partition $\mathbb{T}^2=M_{1,t_0}\cup \Gamma_{t_0}\cup M_{2, t_0}$ is almost perfectly balanced, with $\mu_2(M_1)=0.49994$ and $\mu_2(M_2)=0.5006$.  One has $\mathbf{H}^D(\Gamma_{t_0})=0.6968$. Despite the highly nonlinear nature of $T$, evident in the distribution of $h_\nu$, shown in Figure \ref{fig:dyn6a}, one can find curves that are rather short according to $\nu_1$, both before and after the application of $T$. Therefore this solution is a suitable candidate for LCSs on $(\mathbb{T}^2, e, \mu_2)$, for the finite-time (single) application of $T$.


\section{Conclusions}

The dynamic isoperimetric theory initiated in \cite{froyland14} was concerned with identifying subsets of $\mathbb{R}^r$ with persistently least boundary size to volume ratio under general nonlinear volume-preserving dynamics.
The motivation for this theory was that the boundaries of such sets have optimality properties desired in Lagrangian coherent structures.
In the present work we have extended the constructions and theoretical results of \cite{froyland14} to weighted, non-flat Riemannian manifolds and to possibly non-volume preserving dynamics.
This entailed developing a nontrivial generalisation of the dynamic isoperimetric problem to weighted manifolds and allowing for non-volume preserving dynamics.
We proved a new dynamic version of the classical Federer-Fleming theorem in this setting, which very tightly links the (geometric) dynamic isoperimetric problem with a (functional) minimisation of a new dynamic Sobolev constant.

We then constructed a weighted dynamic Laplacian, and showed that under a natural Neumann-type boundary condition, the spectrum of this weighted dynamic Laplacian can be completely characterised using variational principles tied to the finite-time dynamics of $T$ and the geometry of the manifold. We additionally proved that a dynamic Cheeger inequality holds on weighted Riemannian manifolds, extending a result from \cite{froyland14} for flat, unweighted manifolds, and volume-preserving dynamics.
We demonstrated numerically that the eigenfunctions of the weighted dynamic Laplacian are able to identify sets with small boundaries that remain small when transformed by general dynamics. Such persistently minimal surfaces are excellent candidates for Lagrangian coherent structures (LCSs) as diffusion across their short boundaries is minimised over a finite-time duration.

Finally, we further developed the connection between two very different methods for detecting transport barriers in dynamical systems, namely the relationship between finite-time coherent sets and LCSs as defined using isoperimetic notions. The connection between these sets, explored in \cite{froyland14} in the flat manifold, volume-preserving dynamics setting, is that in the limit of small diffusion, regions in phase space that minimally mix (a purely probabilistic notion) are linked with sets that have persistently least boundary size (a purely geometric notion).
We further enhanced this link by extending and strengthening the result of \cite{froyland14} to the more general setting of weighted, curved Riemannian manifolds with possibly non-volume preserving dynamics.

Very recently, fast and robust numerical methods have been developed \cite{FrJu17} to accurately compute the spectrum and eigenfunctions of the dynamic Laplacian in rather general situations to extract the dominant LCSs
These include settings where the dynamical system is not known \textit{a priori} and only a limited amount of (possibly corrupted) data is available.


%

\section*{Appendix}
\appendix

Let $(M, m)$ and $(N, n)$ be compact, connected $r$-dimensional $C^\infty$ Riemannian manifolds, where $m, n$ are the Riemannian metric tensors on $M$ and $N$ respectively. Let $T$ be a $C^\infty$-diffeomorphism of $M$ onto $N$. Let $(M, m, \mu_r)$ and $(N, n, \nu_r)$ be weighted Riemannian manifolds, where $\mu_r$ is absolutely continuous probability measures with respect to $V_m$, and $\nu_r=\mu_r\circ T^{-1}$. Denote by $h_\mu$ and $h_\nu$ the densities of the measures $\mu_r$ and $\nu_r$ respectively.

In comparison to the predecessor work \cite{froyland14}, Appendix A is new because we need to treat the weights $h_\mu$ which were simply the constant function in the unweighted setting of \cite{froyland14}.
Much of Appendix B is new, in order to handle a general Riemannian metric, rather than the locally flat metrics in \cite{froyland14}.
Appendix C is required to treat \emph{weighted} Sobolev spaces (the weight was uniform in \cite{froyland14}).
Appendices D--F broadly follow the overall strategy of the corresponding sections in Appendices A, C, B, D, respectively, in \cite{froyland14}, however in the present work there are additional technical issues arising from the weights and the Riemannian metrics.
Moreover, because we need to handle general Riemannian metrics, wherever possible we use a coordinate-free framework, as opposed to the explicit coordinates employed in the simpler locally Euclidean setting of \cite{froyland14}.

\section{Muckenhoupt weights $A_p$ }\label{sec:ap}

For a measurable function $f$ on $M$, the \emph{essential supremum} of $f$ is the number
\begin{equation*}
 \textup{ess sup} f:=\{a\in \mathbb{R}: V_m(f^{-1}(a, \infty))=0\}.
\end{equation*}
Recall from Section \ref{sec:2} that the volume form $\omega_m^r$ on $M$ is given in terms of the volume measure $V_m$ on $M$ via $V_m(U)=\int_U \omega_m^r$ for any measurable subset $U\subset M$. Define the class of $A_p$ weights \cite{turesson00}:

\begin{definition}\label{def:apw}
Let $B_\rho(x)\subset M$ denote the metric ball centered at $x\in M$ with radius $\rho>0$. The density $h_\mu$ of a measure $\mu_r$ is said to be an $A_p$ weight of $(M, m, \mu_r)$, if there exists a constant $C_\mu$ such that for every $x$ and $\rho$, $h_\mu$ satisfies
\begin{equation}\label{eq:apw}
\left(\frac{1}{V_m(B_\rho(x))}\int_{B_\rho(x)} h_\mu\cdot\omega_m^r\right)\left(\frac{1}{V_m(B_\rho(x))}\int_{B_\rho(x)} h_\mu^{-\frac{1}{p-1}}\cdot\omega_m^r \right)^{p-1}\leq C_\mu,
\end{equation}
for $1<p<\infty$, or
\begin{equation}\label{eq:apw2}
\left(\frac{1}{V_m(B_\rho(x))}\int_{B_\rho(x)} h_\mu\cdot\omega_m^r\right)\left(\underset{z\in B_\rho(x)}{\textup{ess sup}}\frac{1}{h_\mu(z)} \right)\leq C_\mu,
\end{equation}
for $p=1$. We call $C_\mu$ the $A_p$ constant of $h_\mu$.
\end{definition}

\begin{proposition}
\label{thm:smoothap}
 Suppose the density $h_\mu$ of the measure $\mu_r$ is Lipschitz and uniformly bounded away from zero. Then $h_\mu$ is an $A_p$ weight for all $1\leq p<\infty$.
\end{proposition}
\begin{proof}
Let $B_\rho(x)\subset M$, and denote by $\textup{dist}_m$ be the Riemannian distance function with respect to the metric $m$. Since $h_\mu$ is Lipschitz and nonnegative, for every $x, z\in M$ one has $h_\mu(z)\leq h_\mu(x)+K\textup{dist}_m(x, z)$ for some $K<\infty$. Hence, for every $x\in M$ and $\rho>0$
\begin{align}\label{eq:smoothap1}
 \frac{1}{V_m(B_\rho(x))}\int_{B_\rho(x)} h_\mu(z)\cdot \omega_m^r(z)
 &\leq \frac{h_\mu(x)+K\rho}{V_m(B_\rho(x))}\int_{B_\rho(x)} \omega_m^r(z)\notag\\
 &=h_\mu(x)+K\rho.
\end{align}
Since $M$ is compact and $B_\rho(x)\subset M$, one has $\rho<\infty$. Also, since $h_\mu$ is Lipschitz, it is bounded on $M$. Hence, the RHS of \eqref{eq:smoothap1} is bounded above by $\sup_{x\in M} (h_\mu(x)+K\rho)$.

In addition, since $h_\mu$ is uniformly bounded away from zero, $h_\mu^{-1}$ and $h_\mu^{-1/(p-1)}$ are bounded for $1< p<\infty$. Hence, there exist constants $\gamma_p$ and $\gamma_1$ such that
\begin{equation}\label{eq:smoothap2}
\frac{1}{V_m(B_\rho(x))}\int_{B_\rho(x)} h_\mu^{-\frac{1}{p-1}}\cdot\omega_m^r\leq \gamma_p,
\end{equation}
$1<p<\infty$, and
\begin{equation}\label{eq:smoothap3}
 \left(\underset{z\in B_\rho(x)}{\textup{ess sup}}\frac{1}{h_\mu(z)} \right) \leq \gamma_1,
\end{equation}

Hence, by \eqref{eq:smoothap1}, \eqref{eq:smoothap2} there are constants $C_\mu=\gamma_p^{1/(p-1)}\cdot \sup_{x\in M}(h_\mu(x)+K \rho)$, such that \eqref{eq:apw} is satisfied for $1<p<\infty$. Similarly, there is a constant $C_\mu=\gamma_1\cdot \sup_{x\in M}(K \rho+h_\mu(x))$, such that \eqref{eq:apw2} is satisfied for $p=1$. Hence, $h_\mu$ is an $A_p$ weight.
\end{proof}

Let $L^1_{\textup{loc}}(M, V_m)$ denote the space of locally integrable functions; that is, if $f\in L^1_{\textup{loc}}(M, V_m)$ then $\int_{B_\rho(x)} f\cdot\omega_m^r<\infty$ for every $x\in M$ and $\rho\in \mathbb{R}^+$.  Given a weighted Riemannian manifold $(M, m, \mu_r)$, we wish to determine the condition on the density $h_\mu$ so that $L^p(M, \mu_r)\subset L_{\textup{loc}}^1(M, V_m)$.

\begin{proposition}\label{thm:loc}
 Let $B_\rho(x)\subset M$ denote the metric ball centered at $x\in M$ with radius $\rho>0$. If $h_\mu^{-1/(p-1)}$ is in $L_{\textup{loc}}^1(M, V_m)$ for $p\in (0, \infty)$, or if for every $x$ and $\rho$
 \begin{equation*}
  \underset{z\in B_\rho(x)}{\textup{ess sup}}\frac{1}{h_\mu(z)}<\infty,
 \end{equation*}
 for $p=1$. Then $L^p(M, m, \mu_r)\subset L_{\textup{loc}}^1(M, V_m)$.
\end{proposition}
\begin{proof}
This result appeared in \cite{turesson00} for the case $M=\mathbb{R}^r$; the arguments for the present version is identical, thus are omitted.


\end{proof}

\section{Additional notes on differential geometry}
\label{sec:dfg}
By a local chart on $M$, we meant a pair $(U, \varphi)$ such that $U$ is an open subset of $M$ and $\varphi:U\to \mathbb{R}^r$ a local $C^\infty$-diffeomorphism. The countable collection of local charts $(U_i, \varphi_i)_{i\in I}$ such that $\cup_{i\in I} U_i$ forms an open cover for $M$ is called an atlas. For a fixed $k \in I$,  one can define a set of local coordinates $(x_1, x_2, \ldots, x_r)$ on $U_k$ as the set of smooth projections of the image of $\varphi_k$ onto the $j^{th}$ coordinate, $1\leq j\leq r$; that is $x_j:U_k\to \mathbb{R}$ is a homeomorphism for each $1\leq j\leq r$. Moreover, the atlas on $M$ defines a local coordinate system for each point $x\in M$. In local coordinates, it is possible to carry out the operation of partial differentiation on a differentiable function $f$ at the point $x\in U_k$, $k\in I$ as
\begin{equation}\label{eq:pdiff}
 \left[\frac{\partial}{\partial x_j}\right]_x f:=\frac{\partial (f\circ \varphi_k^{-1})}{\partial x_j}(\varphi_k(x)),
\end{equation}
for each $1\leq j\leq r$. It is well known (see e.g p.7 in \cite{carmo92}) that the above operation is independent on the choice of $\varphi_k$, and therefore we use the abbreviation $\left[\partial/\partial x_j\right]_x f=\partial f/\partial x_j(x)=\partial_j f(x)$ whenever there is no confusion on whether the partial differentiation is carried out on $\mathbb{R}^r$ or $M$. It is straightforward to verify that the set $\{\partial_i\}_{i=1}^r$ forms a basis for the vector fields on $M$. Hence, one can write the metric tensor $m$ in coordinates as $m_{ij}(x)=m(\partial_i, \partial_j)(x)$.

Given a diffeomorphism $T:M\to N$, and local charts $(U, \varphi)$, $(TU, \vartheta)$ on $M$, $N$ respectively. Observe that $\vartheta\circ T:M\to \mathbb{R}^r$ is smooth. Therefore, it is possible to carry out the operation of partial differentiation on $T$ at the point $x\in U$ as
\begin{equation}\label{eq:pdiffT}
  \left[\frac{\partial}{\partial x_j}\right]_x T:=\frac{\partial (\vartheta\circ T\circ \varphi^{-1})}{\partial x_j}(\varphi(x)).
\end{equation}
One can construct the Jacobian matrix $J_T$ in local coordinates via \eqref{eq:pdiffT}, as a $r\times r$ matrix with entries $(J_T)_{ij}:=\partial_jT_i$, where $T_i$ is the smooth projection of the image of $\vartheta\circ T$ onto the $i^{th}$ coordinate, and we use the abbreviation $\partial T_i/\partial x_j(x)=\partial_j T_i(x)$.

\subsection{Differential forms}\label{sec:df}

Let $(x_1, x_2, \ldots, x_r)$ be local coordinates on $M$. Denote by $dx_i$ the differential $1$-forms dual to the tangent basis $\partial_i$, for each $1\leq i\leq r$. For $p\leq r$,  one can express a differentiable $p$-form $\eta$ in coordinates via the exterior product of $1$-forms
\begin{equation}\label{eq:etac}
 \eta=\sum_{j_1<j_2<\ldots<j_{p}} a_{j_1\ldots j_p}dx_{j_1}\wedge dx_{j_2}\wedge\ldots\wedge dx_{j_p},
\end{equation}
where $a_{j_1\ldots j_p}$ are real-valued functions on $M$.

The exterior derivative on a differentiable $f:M\to\mathbb{R}$ is a $1$-form given by $df=\sum_{i=1}^r \partial_i fdx_i$, and the exterior derivative on the $p$-form $\eta$ defined by \eqref{eq:etac} is a $(p+1)$-form satisfying
\begin{equation*}
  d\eta=\sum_{j_1<j_2<\ldots<j_{p}} d(a_{j_1\ldots j_p})\wedge dx_{j_1}\wedge dx_{j_2}\wedge\ldots\wedge dx_{j_p}.
\end{equation*}
The interior derivative $i(\mathcal{\mathcal{V}})$ on a $p$-form $\eta$ with respect to a vector field $\mathcal{V}$ on $M$, is a $(p-1)$-form satisfying
\begin{equation*}
 [i(\mathcal{\mathcal{V}})\eta](\mathcal{V}_1, \mathcal{V}_2, \ldots, \mathcal{V}_{p-1})=\eta(\mathcal{V}, \mathcal{V}_1, \mathcal{V}_2, \ldots, \mathcal{V}_{p-1}),
\end{equation*}
for all vector fields $\mathcal{V}_1, \mathcal{V}_2,\ldots, \mathcal{V}_{p-1}$ on $M$.

Recall the definition of the tangent and cotangent mappings $T_*$ and $T^*$ associated with $T$, given by \eqref{eq:T_*} and \eqref{eq:T^*} respectively. For the differential $p$-form $\eta$ given by \eqref{eq:etac}, one has
\begin{equation*}\label{eq:T^*eta}
 (T^*\eta)(\mathcal{V}_1, \mathcal{V}_2, \ldots, \mathcal{V}_p)(x)=\eta(T_*\mathcal{V}_1, T_*\mathcal{V}_2, \ldots, T_*\mathcal{V}_p)(Tx),
\end{equation*}
for all vector fields $\mathcal{V}_1, \mathcal{V}_2,\ldots \mathcal{V}_p$ on $M$. Therefore, in coordinates
\begin{align}\label{eq:T^*etac}
 T^*\eta
 &=T^*\left(\sum_{j_1<j_2<\ldots<j_{p}} a_{j_1\ldots j_p}dx_{j_1}\wedge dx_{j_2}\wedge\ldots\wedge dx_{j_p}\right)\notag\\
 &=\sum_{j_1<j_2<\ldots<j_{p}} a_{j_1\ldots j_p}\circ T \cdot T^*dx_{j_1}\wedge T^*dx_{j_1}\wedge T^*dx_{j_2}\wedge\ldots\wedge T^*dx_{j_p}\notag\\
 &=\sum_{j_1<j_2<\ldots<j_{p}} a_{j_1\ldots j_p}\circ T \cdot d(x_{j_1}\circ T)\wedge d(x_{j_1}\circ T)\wedge d(x_{j_2}\circ T)\wedge\ldots\wedge d(x_{j_p}\circ T)
\end{align}
where the last line is due to the fact that $[T^*(df)]\mathcal{V}=\mathcal{V}(f\circ T)=[d(f\circ T)]\mathcal{V}$, for all $f\in C^\infty(M, \mathbb{R})$ and vector fields $\mathcal{V}$ on $M$.

Let $G_m(x)$ be a $r\times r$ matrix with components $m_{ij}(x)$ at the point $x\in M$. The volume form $\omega_m^r$ in the local coordinates $\{x_i\}_{i=1}^r$ is defined by
\begin{equation}\label{eq:omega_m^r}
 \omega_m^r(x):=\sqrt{\det G_m}(x)\cdot dx_1\wedge dx_2\ldots\wedge dx_r,
\end{equation}
for each point $x\in M$. Let $\Gamma$ be a $C^\infty$ co-dimension 1 subset of $M$. Recall from Section \ref{sec:2} that the embedding $\Phi:\Gamma\to M$ induces a Riemannian metric on $\Gamma$ via the pullback metric $\Phi^*m$; that is $\omega_m^{r-1}=\Phi^*\omega_m^r$. The following is a classical result in geometric measure theory (see Theorem I.3.1 in \cite{chavel01}):
\begin{lemma}[co-area formula]\label{thm:ca}
Let $f\in C^1(M, \mathbb{R})$. For an open, connected $U\subset M$ with compact closure, and any function $h:M\to \mathbb{R}^+$ in $L^1(M, V_m)$, one has
\begin{equation}\label{eq:ca}
\int_U |\nabla_m f|_mh\cdot \omega_m^r=\int_\mathbb{R} \left( \int_{f^{-1}\{t\}} h\cdot\omega_m^{r-1}\right)\,dt,
\end{equation}
where $|\cdot|^2_m=m(\cdot,\cdot)$ and $\nabla_m f$ is the gradient of $f$ with respective to the metric $m$; defined by \eqref{eq:grad}.
\end{lemma}
The co-area formula connects the spatial integral over the gradient of a function to the co-dimension one measure on the level sets generated by that function. If the density $h_\mu$ of the absolutely continuous  probability measure $\mu_r$ is a positive function in $L^1(M, V_m)$, then one can apply the co-area formula \eqref{eq:ca} with $h=h_\mu$ to obtain
\begin{equation*}
\int_U |\nabla_m f|_m \cdot h_\mu\omega_m^r= \int_U |\nabla_m f|_m h_\mu\cdot\omega_m^r=\int_\mathbb{R} \mu_{r-1}(f^{-1}\{t\})\,dt,
\end{equation*}
for all measurable $U\subset M$.

\subsection{Differential operators on weighted manifolds}\label{sec:B2}
Recall the definitions of the gradient $\nabla_m$, divergence $\divg_m$ and weighted divergence $\divg_\mu$ given by \eqref{eq:grad}, \eqref{eq:div} and \eqref{eq:wdiv} respectively. One can express $\nabla_mf$ in local coordinates $\{x_1, \ldots, x_r\}$ on $M$ as,
 \begin{equation}
 \label{eq:gradl}
  \nabla_m f=\sum_{i,j=1}^r m^{ij}\partial_i f\partial_j,
 \end{equation}
 for all $f\in C^k(M, \mathbb{R})$, and where $m^{ij}$ is the components $r\times r$ matrix $G^{-1}_m$ (see p.4, equation (22) \cite{chavel84}). As a consequence of Stokes' theorem (see p.124, \cite{spivak65}), the divergence given by \eqref{eq:div} can be written as,
 \begin{equation}\label{eq:div2}
\divg_m\mathcal{V}\cdot \omega_m^r=d[i(\mathcal{V})\omega_m^r],
 \end{equation}
 for all $\mathcal{V}\in \mathcal{F}^k(M)$. Since $\{\partial_i\}_{i=1}^r$ forms a basis for the vector fields on $M$, one can express the vector field $\mathcal{V}$ on $M$ as $\mathcal{V}=\sum_{i=1}^r \mathcal{V}^i \partial_i$. Then \eqref{eq:div2} in local coordinates is (see equation (32) on p.5 in \cite{chavel84}),
 \begin{equation}\label{eq:divl}
 \divg_m\mathcal{V}=\frac{1}{\sqrt{\det{G_m}}}\sum_{i=1}^r\partial_i\left( \sqrt{\det{G_m}}\mathcal{V}^i\right),
 \end{equation}
Hence, the Laplace-Beltrami operator is given in local coordinates by
\begin{equation}\label{eq:laplacel}
 \triangle_m f=\frac{1}{\sqrt{\det{G_m}}}\sum_{i, j=1}^r\partial_i\left(m^{ij}\sqrt{\det{G_m}}  \partial_j f\right).
\end{equation}

Let $f:M\to \mathbb{R}$ be differentiable, and  $\mathcal{V}_1$, $\mathcal{V}_2$ vector fields on $M$. The standard divergence properties (see equation (12) and (13) on p.3, \cite{chavel84}) holds analogously for the weighted divergence \eqref{eq:wdiv}; that is for $h_\mu\in C^1(M, \mathbb{R})$
\begin{align}\label{eq:pdiv1}
 \divg_\mu(\mathcal{V}_1+\mathcal{V}_2)
 &=\frac{1}{h_\mu}\divg(h_\mu\mathcal{V}_1+h_\mu \mathcal{V}_2 )\notag\\
 &=\frac{1}{h_\mu}\divg(h_\mu\mathcal{V}_1)+\frac{1}{h_\mu}\divg(h_\mu \mathcal{V}_2 )\notag\\
 &=\divg_\mu \mathcal{V}_1+\divg_\mu\mathcal{V}_2
\end{align}
and
\begin{align}\label{eq:pdiv2}
 \divg_\mu(f\mathcal{V}_1)
 &=\frac{1}{h_\mu}\divg (h_\mu f\mathcal{V}_1)\notag\\
 &=\frac{f}{h_\mu}\divg(h_\mu \mathcal{V}_1)+\frac{1}{h_\mu}m( \nabla_m f,h_\mu\mathcal{V}_1)\notag\\
 &=f\divg_\mu \mathcal{V}_1+m(\nabla_m f, \mathcal{V}_1).
\end{align}

\subsection{Properties of $T_*$ and $T^*$}\label{sec:B3}

Let $g:N\to \mathbb{R}$ be differentiable. One can express the tangent mapping $T_*$ given by \eqref{eq:T_*} in local coordinates $\{x_1, \ldots, x_r\}$ as
\begin{equation*}
(T_*\mathcal{V})g=\mathcal{V}(g\circ T)=\sum_{i=1}^r \mathcal{V}^i\frac{\partial(g\circ T)}{\partial x_i},
\end{equation*}
where $\mathcal{V}=\sum_{i=1}^r\mathcal{V}^i\partial_i\in \mathcal{T}M$.
The following result computes coordinate representations of the pullback metric $T^*n$.

 \begin{lemma}\label{thm:app3.1}
 Let $n_{ij}$ be the local coordinates representation of the metric tensor $n$. Denote by $G_n$ the $r\times r$ matrix with components $n_{ij}$ and $J_T$ the Jacobian matrix of $T$. We have at the each point $x\in M$
 \begin{equation}\label{eq:app3.1a}
   (G_{T^*n})_{ij}=\left(J_T^\top\cdot G_n\circ T \cdot J_T\right)_{ij}
 \end{equation}
 where $T^*n$ is the pullback metric of $n$ given by \eqref{eq:pbm}.
 \end{lemma}
 \begin{proof}
 Let $(U, \varphi_m)$ be a local chart on $M$, containing the point $x_0\in U$ with corresponding coordinates $\{x_i\}_{i=1}^r$. Then the local chart $(TU, \varphi_n)$ on $N$ contains the point $Tx_0\in N$. Let $\{y_i\}_{i=1}^r$ denote the local coordinates on $TU$. Due to \eqref{eq:T_*} and \eqref{eq:pdiff}, one has for all differentiable $g$ on $N$
 \begin{align*}
\left[T_*\frac{\partial}{\partial x_i}\right]_{Tx_0} g
=\left[\frac{\partial}{\partial x_i}\right]_{x_0} (g\circ T)
  &=\frac{\partial (g\circ T\circ \varphi_m^{-1})}{\partial x_i}(\varphi_m (x_0))\\
  &=\frac{\partial (g\circ \varphi_n^{-1}\circ \varphi_n \circ T\circ \varphi_m^{-1})}{\partial x_i}(\varphi_m (x_0))\\
  &=\sum_{k=1}^r \frac{\partial (\varphi_n\circ T_k\circ \varphi_m^{-1})}{\partial x_i}(\varphi_m (x_0))\frac{\partial (g\circ\varphi_n^{-1})}{\partial y_k}(\varphi_n(Tx_0))\\
  &=\sum_{k=1}^r \left[\frac{\partial }{\partial x_i}\right]_{x_0}T_k\cdot \left[\frac{\partial}{\partial y_k}\right]_{Tx_0}g,
 \end{align*}
 where the last equality is due to \eqref{eq:pdiffT}. Therefore,
 \begin{equation}\label{eq:app3.1a1}
 T_*\frac{\partial}{\partial x_i}=\sum_{k=1}^r (J_T)_{ki}\circ T^{-1} \cdot \frac{\partial}{\partial y_k}
 \end{equation}
 at the point $Tx_0$; that is $T_*(\partial/\partial x_i)$ is a tangent vector in $\mathcal{T}_{Tx_0}(N)$ with components
 $\left(T_*(\partial/\partial x_i)\right)^k=(J_T)_{ki}\circ T^{-1}$, $1\leq k\leq r$. To obtain \eqref{eq:app3.1a} at the point $x_0$, we compute
 \begin{align*}
 (G_{T^*n})_{ij}(x_0)= T^*n\left( \frac{\partial}{\partial x_i}, \frac{\partial}{\partial x_j}\right)(x_0)
  &=n\left( T_*\frac{\partial}{\partial x_i}, T_*\frac{\partial}{\partial x_j}\right)(Tx_0)\\
  &=n\bigg(\sum_{k=1}^r  \left(T_*\frac{\partial}{\partial x_i}\right)^k \frac{\partial}{\partial y_k} , \sum_{l=1}^r \left(T_*\frac{\partial}{\partial x_j}\right)^l\frac{\partial}{\partial y_l}\bigg)(Tx_0)\\
  &=\sum_{k,l=1}^r \restr{\bigg(n_{kl} \cdot \left(T_*\frac{\partial}{\partial x_i}\right)^k \cdot \left(T_*\frac{\partial}{\partial x_j}\right)^l\bigg)}{Tx_0}\\
  &=\sum_{k,l=1}^r(J_T(x_0))_{ki}\cdot n_{kl}(Tx_0)\cdot (J_T(x_0))_{lj}\quad\mbox{by \eqref{eq:app3.1a1}}\\
  &=\bigg(J_T^\top \cdot G_n\circ T \cdot J_T\bigg)_{ij}(x_0).
 \end{align*}
 Since $x_0\in U$ is arbitrary and $(U, \varphi_m)$ is a chart for $M$, we conclude that the above calculations hold for all points in $M$.
 \end{proof}

 \begin{corollary}\label{thm:app3.2}
 Let $n$ be the metric tensor of $N$, with volume form $\omega_n^r$ given by \eqref{eq:omega_m^r}.  Define the co-tangent mapping $T^*$ as in \eqref{eq:T^*}. One has $T^*\omega_n^r=\omega^r_{T^*{n}}$.
 \end{corollary}

 \begin{proof}
 Let $\{x_i\}_{i=1}^r$ and $\{y_i\}_{i=1}^r$ be local coordinates on $M$ and $N$ respectively. Then by \eqref{eq:T^*etac}, one has for each $1\leq i\leq r$,
  \begin{equation*}
  T^*(dy_i)=d(y_i\circ T)=\sum_{j=1}^r \frac{\partial (y_i\circ T)}{\partial x_j}dx_j=\sum_{j,k=1}^r  \frac{\partial T_k}{\partial x_j}\frac{\partial y_i}{\partial y_k}\circ T\cdot dx_j=\sum_{j,k=1}^r\frac{\partial T_k}{\partial x_j} \delta_{ik}\cdot dx_j,
  \end{equation*}
  where $\delta_{ik}$ is the kronecker delta. Therefore, $T^*(dy_i)=\sum_{j=1}^r \partial T_i/\partial x_j \cdot dx_j$. It follows that
  \begin{equation}\label{eq:app3.2a}
   T^*(dy_1\wedge dy_2\wedge\ldots\wedge dy_r)=T^*dy_1\wedge T^*dy_2\wedge\ldots\wedge T^*dy_r=|\det J_T|\cdot dx_1\wedge dx_2\wedge\ldots\wedge dx_r.
  \end{equation}
 Let $G_n$, $G_{T^*n}$ to be the $r\times r$ matrices with entries $n_{ij}$, $(T^*n)_{ij}$ in coordinates $\{y_i\}_{i=1}^r, \{x_i\}_{i=1}^r$ respectively. Then by Lemma \ref{thm:app3.1}, one has  $\det{(G_{T^*n})}=|\det(J_T)|^2\det{(G_n)\circ T}$, which implies
  \begin{align*}
   T^*(\omega_n^r)&=\sqrt{\det G_n}\circ T  \cdot T^*(dy_1\wedge\ldots\wedge dy_r)\quad\mbox{by \eqref{eq:T^*etac}}\\
   &=\sqrt{\det G_n}\circ T \cdot |\det J_T|\cdot dx_1\wedge\ldots\wedge dx_r\quad\mbox{by \eqref{eq:app3.2a}}\\
   &=\sqrt{\det G_{T^*n}} \cdot dx_1\wedge\ldots\wedge dx_r=\omega_{T^*n}^r.
  \end{align*}
 \end{proof}

Recall that $T$ is an isometry from $(M,T^*n)$ to $(N, n)$. Due to Corollary \ref{thm:app3.2} one has
\begin{equation}\label{eq:cov}
 \int_{T(U)}\omega_n^r=\int_{U}\omega^r_{T^*n}=\int_{U}T^*(\omega_n^r),
\end{equation}
for all measurable $U\subset M$. Hence, by the definition of $\mathcal{P}$ given by \eqref{eq:P-F}, one has for all $f\in L^1(M, V_m)$
\begin{align*}
 \int_U f\cdot \omega_m^r
 &=\int_{T(U)} \mathcal{P}f \cdot \omega_n^r\notag\\
 &=\int_{U}\mathcal{P}f\circ T\cdot \omega_{T^*n}^r\quad \mbox{by \eqref{eq:cov}}\notag\\
 &=\int_{U} \mathcal{P}f\circ T\cdot\sqrt{\det{G_{T^*n}}}\cdot dx_1\wedge\ldots\wedge dx_r\notag\\
 &=\int_{U} \mathcal{P}f\circ T\cdot |\det{J_T}| \cdot \sqrt{\det{G_{n}}}\circ T\cdot dx_1\wedge\ldots\wedge dx_r,
\end{align*}
 where the last line is due to Lemma \ref{thm:app3.1}. Hence, since $T$ is a diffeomorphism and $\omega_m^r=\sqrt{\det{G_m}}\cdot dx_1\wedge\ldots\wedge dx_r$ by \eqref{eq:omega_m^r}, one has
 \begin{align}\label{eq:P-F2}
   \mathcal{P}f
   &=\frac{f\circ T^{-1}}{|\det J_{T}\circ T^{-1}|}\cdot \frac{\sqrt{\det G_m}\circ T^{-1}}{\sqrt{\det G_n}}\notag\\
   &=f\circ T^{-1}\cdot |\det J_{T^{-1}}|\cdot \frac{\sqrt{\det G_m}\circ T^{-1}}{\sqrt{\det G_n}},
 \end{align}
where we have applied the inverse function theorem to obtain the last line. Moreover, setting $f=h_\mu$ in \eqref{eq:P-F2} and using the fact that $\mathcal{P}h_\mu=h_\nu$ (by \eqref{eq:CalP}) yields
 \begin{equation}\label{eq:finiteJT}
 h_\mu=h_\nu\circ T\cdot|\det{J_T}|\cdot \frac{\sqrt{\det{G_{n}}}\circ T}{\sqrt{\det{G_m}}}.
 \end{equation}
Now by assumption, $T$ is a diffeomorphism and the densities $h_\mu$ and $h_\nu$ are uniformly bounded away from zero. Therefore, by \eqref{eq:finiteJT} and the nondegeneracy of the metrics $m$ and $n$, the Jacobian $|\det J_T|$ is bounded above and uniformly away from zero.

Let $\mathbf{1}_V$ denote the characteristic function on a measurable subset $V\subset N$. One has for all $f\in L^1(M, V_m)$
\begin{equation}\label{eq:koopman}
\int_{N} \mathcal{P}f\cdot \mathbf{1}_{V}\cdot \omega_n^r=\int_{V} \mathcal{P}f\cdot \omega_n^r=\int_{T^{-1}V} f\cdot \omega_m^r=\int_M f\cdot \mathbf{1}_{V}\circ T\cdot\omega_m^r.
\end{equation}
Hence, the Koopman operator $\mathcal{K}$ adjoint to $\mathcal{P}$ is given by $\mathcal{K}f=f\circ T$.

Recall the definition of the push-forward operator $\mathcal{L}:L^2(M, m, \mu_r)\to L^2(N, n, \nu_r)$ given by \eqref{eq:pfw}, with $L^2(M, m, \mu_r)$ adjoint $\mathcal{L}^*$.
\begin{lemma}
\label{lem:welldefL}
The operator $\mathcal{L}:L^2(M, m, \mu_r)\to L^2(N, n, \nu_r)$ is well defined, may be expressed as $\mathcal{L}f=f\circ T^{-1}$, and has adjoint $\mathcal{L}^*g=g\circ T$.
\end{lemma}
\begin{proof}
Let $f\in L^2(M, m, \mu_r)$. Due to \eqref{eq:P-F2} and the fact that $h_\mu>0$, one has

\begin{eqnarray}
\label{eq:needed} |\mathcal{P}(f\cdot h_\mu)|^2
 &=&|(f\cdot h_\mu)\circ T^{-1}|^2\cdot \left|\det J_{T^{-1}}\right|^2 \cdot \left|\frac{\sqrt{\det G_m} \circ T^{-1}}{\sqrt{\det G_n}}\right|^2\\
 \nonumber&=&\left|(f^2\cdot h_\mu) \circ T^{-1} \cdot \left|\det J_{T^{-1}}\right| \cdot \frac{\sqrt{\det G_m} \circ T^{-1}}{\sqrt{\det G_n}}\right|\cdot
 \left|h_\mu \circ T^{-1}  \cdot \left|\det J_{T^{-1}}\right| \cdot \frac{\sqrt{\det G_m} \circ T^{-1}}{\sqrt{\det G_n}}\right|\\
\nonumber &=&|\mathcal{P}(f^2\cdot h_\mu)|\cdot |\mathcal{P}h_\mu|\\
\nonumber &=&\mathcal{P}(f^2\cdot h_\mu)\cdot h_\nu.
\end{eqnarray}

Therefore
\begin{align}\label{eq:calLdf}
\int_N |\mathcal{L}f|^2\,d\nu_r
&=\int_N \left|\frac{\mathcal{P}(f\cdot h_\mu)}{h_\nu}\right|^2\cdot h_\nu\omega_n^r\notag\\
&=\int_N \frac{\left|\mathcal{P}(f\cdot h_\mu)\right|^2}{h_\nu}\cdot \omega_n^r\quad \mbox{since $h_\nu>0$}\notag\\
&=\int_N \mathcal{P}(f^2\cdot h_\mu)\cdot  \omega_n^r\notag\\
&=\int_M f^2\cdot h_\mu \cdot \omega_m^r\\
&=\int_M f^2\,d\mu_r,
\end{align}
where the second last line is due to \eqref{eq:P-F}.
Thus,
since $f\in L^2(M, m, \mu_r)$ the RHS of \eqref{eq:calLdf} is bounded and $\mathcal{L}$ is well defined.

To show that $\mathcal{L}f=f\circ T^{-1}$, we use (\ref{eq:needed}) (without the squares) to compute $\mathcal{P}(f\cdot h_\mu)$, and (\ref{eq:P-F2}) to compute $h_\nu=\mathcal{P}h_\mu$, and note that all terms in the quotient $\mathcal{L}=\mathcal{P}(f\cdot h_\mu)/h_\nu$ not involving $f$ cancel to leave $\mathcal{L}=f\circ T^{-1}$.

For all measurable $U\subset M$,
\begin{equation}\label{eq:cov3}
\int_{T(U)} \mathcal{L}f\cdot h_\nu\omega_n^r=\int_{T(U)}\mathcal{P}(f\cdot h_\mu)\cdot \omega_n^r=\int_U f\cdot h_\mu\omega_m^r.
\end{equation}
Let $U\subset M$ be measurable. Since $\mathcal{P}h_\mu=h_\nu$, one has $\mathcal{P}(\mathbf{1}_U\cdot h_\mu)=\mathbf{1}_{T(U)}\cdot h_\nu$. Therefore,
\begin{align}\label{eq:cov4}
\int_{T(U)} g\, d\nu_r
&=\int_{N} \frac{\mathbf{1}_{T(U)}\cdot h_\nu}{h_\nu}\cdot g\, d\nu_r\notag\\
&=\int_{N} \frac{\mathcal{P}(\mathbf{1}_U\cdot h_\mu)}{h_\nu}\cdot g\,d\nu_r\notag\\
&=\int_{N} \mathcal{L}(\mathbf{1}_U)\cdot g\,d\nu_r\notag\\
&=\int_M \mathbf{1}_U\cdot \mathcal{L}^*g\,d\mu_r\quad\mbox{by definition of $\mathcal{L}^*$}\notag\\
&= \int_U \mathcal{L}^*g\,d\mu_r \notag\\
\end{align}
for all $g\in L^2(N, n, \nu_r)$. Therefore, using the fact that $\nu_r=\mu_r\circ T^{-1}$, one has for any measurable $V\subset N$
\begin{equation}\label{eq:L^*}
 \int_M \mathcal{L}^*\mathbf{1}_{V}\,d\mu_r=\int_N \mathbf{1}_{V}\,d\nu_r=\int_M \mathbf{1}_{T^{-1}V}\,d\mu_r=\int_M \mathbf{1}_V\circ T\,d\mu_r.
\end{equation}
Thus, $\mathcal{L}^*g=g\circ T$ for all $g\in L^2(N, n, \nu_r)$.
\end{proof}

The following proposition is immediate in view of Lemma \ref{lem:welldefL}.
\begin{proposition}\label{thm:Lf}
Let $\mathcal{L}:L^2(M, m, \mu_r)\to L^2(N, n, \nu_r)$ be as in \eqref{eq:pfw}, with adjoint $\mathcal{L}^*$. For any $f\in C^1(M, \mathbb{R})\cap L^2(M, m, \mu_r)$, one has
\begin{equation*}
T\{x\in M: f(x)=t\}=\{y\in N: \mathcal{L}f(y)=t\}.
\end{equation*}
\end{proposition}

\begin{lemma}\label{thm:app3.3}
Let $\mathcal{L}:L^2(M, m, \mu_r)\to L^2(N, n, \nu_r)$ be as in \eqref{eq:pfw}, with adjoint $\mathcal{L}^*$. One has
 \begin{enumerate}
  \item $\nabla_n =T_*\nabla_{T^*n}\mathcal{L}^*$,
  \item $\mathcal{L}^*\divg_n T_*=\divg_{T^*n}$,
  \item $\mathcal{L}^*\triangle_n \mathcal{L}f=\triangle_{T^*n}$.
  \end{enumerate}
 \end{lemma}
\begin{proof}

\begin{enumerate}
\item
{Let $g\in C^1(N, \mathbb{R})\cap L^2(N, n, \mu_r)$ and $\mathcal{V}\in \mathcal{F}^1(M)$. One has by \eqref{eq:pbm}
 \begin{align*}
n(T_*\nabla_{T^*n}\mathcal{L}^*g, T_*\mathcal{V})(Tx)
&=T^*n(\nabla_{T^*n}\mathcal{L}^*g, \mathcal{V})(x)\\
&=\restr{\mathcal{V}(\mathcal{L}^*g)}{x}\quad\mbox{by \eqref{eq:grad} with respect to $T^*n$} \\
&=\restr{\mathcal{V}(g\circ T)}{x}\quad\mbox{by \eqref{eq:L^*}}\\
&=\restr{(T_*\mathcal{V})g}{Tx}\\
&=n(\nabla_n g, T_*\mathcal{V})(Tx),
 \end{align*}
 for all $x\in M$.
 }

\item{
Let $\mathcal{V}, \mathcal{V}_1, \mathcal{V}_2,\ldots, \mathcal{V}_{r-1}$ be $r$ vector fields in $\mathcal{F}^1(M)$. One has at each point $x\in M$
\begin{align*}
 [i(T_*\mathcal{V})\omega_n^r)](T_*\mathcal{V}_1, T_*\mathcal{V}_2,\ldots, T_*\mathcal{V}_{r-1})(Tx)
 &=\omega_n^r(T_*\mathcal{V}, T_*\mathcal{V}_1, T_*\mathcal{V}_2, \ldots, T^*\mathcal{V}_{r-1})(Tx)\\
 &=(T^*\omega_n^r)(\mathcal{V}, \mathcal{V}_1, \mathcal{V}_2,\ldots, \mathcal{V}_{r-1})(x)\\
 &=[i(\mathcal{V})\omega_{T^*n}^r](\mathcal{V}_1, \mathcal{V}_2,\ldots, \mathcal{V}_{r-1})(x),
\end{align*}
where we have applied the identity $T^*\omega_n^r=\omega_{T^*n}^r$ in Corollary \ref{thm:app3.2} on the last line. Hence, by the duality of $T_*$ and $T^*$, one has at each point $x\in M$
\begin{equation}\label{eq:app3.3.2}
T^*d[i(T_*\mathcal{V})\omega_n^r)]=d[i(\mathcal{V})\omega_{T^*n}^r].
\end{equation}
Therefore,
\begin{align*}
\int_U \mathcal{L}^*\divg_n(T_*\mathcal{V})\cdot\omega_{T^*n}^r
&=\int_U \divg_n(T_*\mathcal{V})\circ T\cdot\omega_{T^*n}^r\\
&=\int_{TU} \divg_n(T_*\mathcal{V})\cdot\omega_{n}^r\quad\mbox{by \eqref{eq:cov}}\\
&=\int_{TU} d[i(T_*\mathcal{V})\omega_n^r)]\quad\mbox{by \eqref{eq:div2} with respect to $n$}\\
&=\int_{U} T^*d[i(T_*\mathcal{V})\omega_n^r)]\\
&=\int_{U}d[i(\mathcal{V})\omega_{T^*n}^r]\quad\mbox{by \eqref{eq:app3.3.2}}\\
&=\int_{U} \divg_{T^*n}(\mathcal{V}) \cdot\omega_{T^*n}^r.
\end{align*}
}
\item{ Due to 1. and 2. and the fact that $\mathcal{L}^*\mathcal{L}$ is the identity by Lemma \ref{lem:welldefL}, one has
 $\mathcal{L}^*\triangle_n \mathcal{L}f=\mathcal{L}^*\divg_n(\nabla_n \mathcal{L}f)=\mathcal{L}^*\divg_{n}(T_*\nabla_{T^*n}\mathcal{L}^*\mathcal{L}f)=\divg_{T^*n}\nabla_{T^*n}f=\triangle_{T^*n}f$, for all $f\in C^2(M, \mathbb{R})\cap L^2(M, m, \mu_r)$.
 }
 \end{enumerate}

\end{proof}
%

\begin{corollary}\label{thm:dwlp2}
 Let $\mathcal{L}:L^2(M, m, \mu_r)\to L^2(N, n, \nu_r)$ be as in \eqref{eq:pfw}, with adjoint $\mathcal{L}^*$. One has
 \begin{equation}\label{eq:dwc2}
 \triangle^Df=\frac{1}{2}\left(\triangle_m +\mathcal{L}^*\triangle_n \mathcal{L}  \right)f+\frac{1}{2}\left(\frac{m(\nabla_m h_\mu, \nabla_m f)}{h_\mu}+\frac{n(\nabla_n h_\nu, \nabla_n \mathcal{L}f)\circ T}{h_\nu\circ T}\right),
\end{equation}
for all $f\in C^2(M, \mathbb{R})\cap L^2(M, m, \mu_r)$.
\end{corollary}
\begin{proof}
By definition
\begin{equation}\label{eq:dwlp2a}
 \triangle^D=\triangle_\mu+\mathcal{L}^*\triangle_\nu\mathcal{L}.
\end{equation}

 Substituting a straightforward modification of \eqref{def:wlp} into the second term on the RHS of \eqref{eq:dwc}, one has for all $x\in M$ and $f\in C^2(M, \mathbb{R})\cap L^2(M, m, \mu_r)$,
\begin{align*}
 \mathcal{L}^*\triangle_\nu \mathcal{L} f(x)
 &=\triangle_\nu (\mathcal{L} f)(Tx)\\
 &=\triangle_n (\mathcal{L}f)(Tx)+\frac{n(\nabla_n h_\nu, \nabla_n \mathcal{L}f)(Tx)}{h_\nu(Tx)}\\
 &=\mathcal{L}^*\triangle_n \mathcal{L}f(x)+\frac{n(\nabla_n h_\nu, \nabla_n \mathcal{L}f)(Tx)}{h_\nu\circ T(x)}.
\end{align*}
Similarly, one can expand the first term of \eqref{eq:dwlp2a} using \eqref{def:wlp} to obtain the required result.
\end{proof}

\begin{corollary}\label{thm:dwlp3}
Let $\triangle^D$ and $\triangle_\mu$ be defined by \eqref{eq:dwc} and \eqref{def:wlp} respectively. One has
 \begin{equation}
  \triangle^D=\frac{1}{2}(\triangle_\mu+\triangle_{\tilde{\mu}})f,
 \end{equation}
 where $\triangle_{\tilde\mu}$ is given by \eqref{def:wlp} with respect to the metric $T^*n$ and density $h_\nu\circ T$.
\end{corollary}

\begin{proof}

Due to Lemma \ref{thm:app3.3}, one has  $\nabla_n=T_*\nabla_{T^*n} \mathcal{L}^*$. Therefore, by the definition of the gradient \eqref{eq:grad} with respect to the metric $n$, one has for all $x\in M$
\begin{align}\label{eq:dwc3a}
n( \nabla_n h_\nu, \nabla_n \mathcal{L}f)_{Tx}
&=\restr{(\nabla_n\mathcal{L}f)h_\nu}{Tx}\notag\\
&=\restr{(T_*\nabla_{T^*n}f)h_\nu}{Tx}\quad\mbox{by Lemma \ref{lem:welldefL}}\notag\\
&=\restr{(\nabla_{T^*n}f)(h_\nu\circ T)}{x}\notag\\
&=T^*n(\nabla_{T^*n}(h_\nu\circ T), \nabla_{T^*n}f)_x,
\end{align}
where the equality on the last line is due to \eqref{eq:grad} with respect to the metric $T^*n$. 
Moreover, by Lemma \ref{thm:app3.3}, one has the identity $\mathcal{L}^*\triangle_n \mathcal{L}=\triangle_{T^*n}$. Thus, by substituting \eqref{eq:dwc3a} into the fourth term on the RHS of \eqref{eq:dwc2}, one has
\begin{align}
\triangle^D f
&=\frac{1}{2}(\triangle_m+\mathcal{L}^*\triangle_{n}\mathcal{L})f+\frac{1}{2}\left(\frac{m(\nabla_m h_\mu, \nabla_m f)}{h_\mu}+\frac{T^*n(\nabla_{T^*n} (h_\nu\circ T), \nabla_{T^*n} f)}{h_\nu\circ T}\right)\notag\\
&=\frac{1}{2}(\triangle_\mu+\triangle_{\tilde{\mu}})f,
\end{align}
where the second equality is due to the definition of weighted Laplacians \eqref{def:wlp}.
\end{proof}


\subsection{Local properties of charts}

An important analytical tool for reducing a global calculation on $M$ to local calculations on each chart of an atlas on $M$ is the \emph{partition of unity}.
\begin{definition}\label{thm:partu}
 Let $(U_i, \varphi_i)_{i\in I}$ be an atlas on $M$. A \emph{partition of unity} subordinate to the covering $\{U_i\}_{i\in I}$, is the collection of smooth functions $\sigma_i\in C^\infty(M, \mathbb{R})$  such that:
 \begin{enumerate}
  \item $\textup{supp}(\sigma_i)\subset U_i$.
  \item Any point $x\in M$ has a neighbourhood $\mathcal{O}_x$ such that $\mathcal{O}_x\cap \textup{supp}(\sigma_i)=\emptyset$ except for a finite set of $\sigma_i$.
  \item $0\leq \sigma_i\leq 1$ and $\sum_{i\in I}\sigma_i=1$.
 \end{enumerate}
\end{definition}

It is well known that the partition of unity exist for paracompact manifolds (see e.g theorem 1.12 \cite{aubin01}). Since every compact manifold is paracompact, a partition of unity exist for $M$.

Furthermore, for each point $x$ in a compact Riemannian manifold $M$, there exist coordinates on a neighbourhood about $x$, and a constant $c>1$ (depending on the injective radius of $x$, and the dimension of the sectional curvature of $M$), such that
\begin{equation}\label{eq:metricb}
 \frac{1}{c}\delta_{ij}\leq m_{ij}\leq c\delta_{ij},\quad 1\leq i, j\leq r,
\end{equation}
where $\delta_{ij}$ is the Kronecker delta (see e.g p.507 in \cite{jost08}, or Chapter 1 of \cite{hebey96}). The following lemmas are consequences of \eqref{eq:metricb}.

\begin{lemma}\label{thm:volchart}
 Let $(U, \varphi)$ be a chart on $(M, m)$, set $\Omega=\varphi(U)$, and denote by $d\ell$ the density with respect the Lebesgue measure. One has
 \begin{equation}\label{eq:volchart}
  c^{-r/2}\int_{U} |f|^p\,d\mu_r\leq\int_\Omega |f\circ \varphi^{-1}|^p\cdot (h_\mu\circ \varphi^{-1})\,d\ell\leq c^{r/2}\int_U |f|^p\,d\mu_r
 \end{equation}
for some real number $c>1$ and all $f\in L^p(U, m, \mu_r)$, $p\in [1, \infty)$.
\end{lemma}
\begin{proof}
Let $\delta_{ij}$ denote the Kronecker delta, and pick local coordinates on $U$ such that the components of the metric tensor $m$ satisfy $\frac{1}{c}\delta_{ij}\leq m_{ij}(x)\leq c\delta_{ij}$ for all $x\in U$ and $1\leq i, j\leq r$. Due to the inequality $m_{ij}\leq\frac{1}{c}\delta_{ij}$, one has $\sqrt{\det{G_m(x)}}\leq c^{r/2}$ for all $x\in U$. Furthermore, the Riemannian volume form is given by $\omega_m^r=\sqrt{\det G_m}\cdot dx_1\wedge dx_2\wedge\ldots\wedge dx_r$ on $U$, and the Lebesgue density satisfies $d\ell=(\varphi^{-1})^*(dx_1\wedge dx_2\wedge\ldots\wedge dx_r)$ on $\Omega$. Hence by the change of variable formula \eqref{eq:cov}
  \begin{align*}
  c^{-r/2}\int_U |f|^p\,d\mu_r
	&=c^{-r/2}\int_{U} |f|^p\cdot h_\mu \sqrt{\det G_m}\cdot dx_1\wedge dx_2\wedge\ldots\wedge dx_r\\
	&\leq\int_{\varphi^{-1}(\Omega)} |f|^p\cdot h_\mu\cdot dx_1\wedge dx_2\wedge\ldots\wedge dx_r\\
 	&=\int_\Omega |f\circ \varphi^{-1}|^p\cdot (h_\mu\circ \varphi^{-1})\,d\ell,
 \end{align*}
where the final equality is due to \eqref{eq:cov}. The inequality $\int_\Omega |f\circ \varphi^{-1}|^p \cdot (h_\mu\circ \varphi^{-1}) \,d\ell\leq c^{r/2}\int_U |f|^p\,d\mu_r$ is obtained analogously using $\frac{1}{c}\delta_{ij}\leq m_{ij}$.
\end{proof}

\begin{lemma}\label{thm:gradchart}
Let $(U, \varphi)$ be a chart on $(M, m)$, set $\Omega=\varphi(U)$, and denote by $e$ the Euclidean metric on $\Omega$ with respect to the Lebesgue density $d\ell$. One has
\begin{equation*}
 c^{-(r/2+1)}\int_U|\nabla_m f|_m^p\,d\mu_r\leq \int_{\Omega}|\nabla_e(f\circ \varphi^{-1})|_e^p\cdot (h_\mu\circ \varphi^{-1})\, d\ell\leq c^{(r/2+1)}\int_U|\nabla_m f|_m^p\,d\mu_r\
\end{equation*}
for some real number $c>1$ and all $\nabla_m f\in L^p(U,m,  \mu_r)$, $p\in [1, \infty)$.
\end{lemma}
\begin{proof}
We start with the case $p=2$. Let $\delta_{ij}$ denote the Kronecker delta, and pick local coordinates on $U$ such that the components of the metric tensor $m$ satisfy $\frac{1}{c}\delta_{ij}\leq m_{ij}(x)\leq c\delta_{ij}$ for all $x\in U$ and $1\leq i, j\leq r$. Denote by $m^{ij}$ the components of the inverse matrix $G_m^{-1}$. One has the contraction $\sum_k m^{ik}m_{kj}=\delta_{ij}$, so that $\frac{1}{c}\delta_{ij}\leq m^{ij}(x)\leq c\delta_{ij}$. Moreover, due to lemma \eqref{eq:volchart}, the inequality \eqref{eq:volchart} is valid with constant $c$. Hence, by writing $\nabla_m f$ in the given local coordinates via \eqref{eq:gradl}, one has
 \begin{align}\label{eq:gradchart1}
  &c^{-(r/2+1)}\int_{U} |\nabla_m f|^2_m\,d\mu_r\notag\\
	\leq& c^{-1}\int_{\Omega} (|\nabla_m f|_m^2\cdot h_\mu)\circ \varphi^{-1}\,d\ell\notag\\
	=& c^{-1}\int_{\Omega} m( \nabla_m f, \nabla_m f)_{\varphi^{-1}(x)}\cdot h_\mu\circ \varphi^{-1}(x)\,d\ell(x)\notag\\
	=&c^{-1}\int_{\Omega} \sum_{i,j=1}^r m_{ij}\left(\sum_{k=1}^r m^{ki} \frac{\partial (f\circ \varphi^{-1})}{\partial x_k}\right)\left(\sum_{l=1}^r m^{lj} \frac{\partial (f\circ \varphi^{-1})}{\partial x_l}\right)\cdot (h_\mu\circ \varphi^{-1})\,d\ell\notag\\
	=&c^{-1}\int_{\Omega}\sum_{j=1}^r \left( \frac{\partial (f\circ \varphi^{-1})}{\partial x_j}\right)\left(\sum_{l=1}^r m^{lj} \frac{\partial (f\circ \varphi^{-1})}{\partial x_l}\right)\cdot (h_\mu\circ \varphi^{-1})\,d\ell\quad\mbox{by contracting the index $i$}\notag\\
	\leq&\int_{\Omega}\sum_{j=1}^r \left( \frac{\partial (f\circ \varphi^{-1})}{\partial x_j}\right)^2\cdot (h_\mu\circ \varphi^{-1})\,d\ell\quad\mbox{since $m^{lj}\leq c\delta_{lj}$}\notag\\
 =&\int_{\Omega} |\nabla_e (f\circ \varphi^{-1})|_e^2\cdot (h_\mu\circ \varphi^{-1})\,d\ell.
 \end{align}
The inequality $ \int_\Omega|\nabla_e (f\circ \varphi^{-1})|_e^2\cdot (h_\mu\circ \varphi^{-1})\,d\ell\leq c^{(r/2+1)}\int_{U} |\nabla_m f|^2_m\, d\mu_r$ is obtained analogously using $\frac{1}{c}\delta_{ij}\leq m_{ij}$.

The general case $p\in [1, \infty)$ is a straightforward modification of the calculation done to obtain \eqref{eq:gradchart1}.
\end{proof}

\section{Weighted Sobolev spaces}\label{sec:ws}


Let $C^\infty_0(\Omega, \mathbb{R})$ be the space of smooth real-valued functions with compact support on $\Omega\subset \mathbb{R}^r$, and $\ell$ the Lebesgue measure on $\mathbb{R}^r$. For locally integrable functions $f, \tilde{f}\in L^1_{\textup{loc}}(\Omega, \ell)$, we say that $\tilde{f}$ is the first order weak derivative of $f$ if $\int_\Omega f\cdot \partial_i g\, d\ell=-\int_\Omega \tilde{f}\cdot g\, d\ell$ for all $g\in C^\infty_0(\Omega, \mathbb{R})$, and each $1\leq i\leq r$ (see p.21 in \cite{adams03}). We write $\tilde{f}=\partial_i f$, and note that $\partial_i f$ is uniquely determined up to sets of measure zero.
\begin{definition}\label{def:wgrad}
Let $(U, \varphi)$ be a chart on $M$ with corresponding local coordinates $(x_1, x_2, \ldots, x_r)$. Define the first order weak gradient of $f\in L^1_{\textup{loc}}(M, V_m)$ at the point $x\in M$ by
\begin{equation}\label{eq:wgrad}
 \tilde{\nabla}_m f(x)=\sum_{i, j=1}^r m^{ij}(x)\cdot \restr{\frac{\partial (f\circ \varphi^{-1})}{\partial x_i}}{\varphi (x)}\partial_j,
\end{equation}
where the partial derivatives appearing on the RHS exist in the weak sense.
\end{definition}

It is straightforward to extend the operation $T_*$ on weak gradients, and verify that Lemma \ref{thm:app3.3} and \ref{thm:gradchart} hold for weak gradients. In addition, if the density of $\mu_r$ is an $A_p$ weight, then by Proposition \ref{thm:loc}, any $f\in L^p(M, m, \mu_r)$ is also in $L^1_{\textup{loc}}(M, V_m)$. Thus, one can define weak gradients on $L^p(M, m, \mu_r)$ via the Definition \ref{def:wgrad}. The following proposition provides the key motivation behind the construction of the weak gradient given above.

\begin{proposition}\label{thm:wgrad}
 Let $f\in L^2(M, m, \mu_r)$, where the density of $\mu_r$ is an $A_2$ weight. Assume the first order weak gradient of $f$ defined by \eqref{eq:wgrad} exists. One has
 \begin{equation}\label{eq:wgrad1}
  \int_{U} f\cdot \triangle_\mu g\,d\mu_r=-\int_{U} m(\tilde{\nabla}_m f, \nabla_m g)\,d\mu_r,
 \end{equation}
for all measurable $U \subset M$ and $g\in C_0^\infty(M, \mathbb{R}^r)$.
\end{proposition}
\begin{proof}
 Let $(U_k, \varphi_k)_{k\in K}$ be an atlas on $M$, with corresponding local coordinates $(x_1, x_2, \ldots, x_r)$. Due to \eqref{eq:omega_m^r}, one has $d\mu_r=h_\mu\sqrt{\det{G_m}}dx_1\wedge dx_2\wedge\ldots\wedge dx_r$. Additionally $d\ell=(\varphi_k^{-1})^*(dx_1\wedge dx_2\wedge\ldots\wedge dx_r)$. Hence, for each $k\in K$ and any measurable $\Omega_k\subset \varphi_k (U_k)$, one has by the coordinate representation of $\triangle_m$ given by \eqref{eq:laplacel}, for each $k\in K$
 \begin{align}\label{eq:wgrad2}
  \int_{\varphi_k^{-1}(\Omega_k)} f\cdot \triangle_\mu g\,d\mu_r
  &=\int_{\varphi_k^{-1}(\Omega_k)} f\cdot \sum_{i, j=1}^r \partial_i(m^{ij}h_\mu\sqrt{\det{G_m}})\partial_j g\cdot dx_1\wedge dx_2\wedge\ldots\wedge dx_r\notag\\
  &=\sum_{i,j=1}^r\int_{\Omega_k} f\circ \varphi_k^{-1}\cdot  \partial_i[(m^{ij}h_\mu\sqrt{\det{G_m}})\circ \varphi_k^{-1}\cdot \partial_j (g\circ \varphi_k^{-1})]\,d\ell\notag\\
  &=\sum_{i,j=1}^r\int_{\Omega_k}  \partial_i (f\circ \varphi_k^{-1})\cdot [(m^{ij}h_\mu\sqrt{\det{G_m}})\circ \varphi_k^{-1}\cdot\partial_j(g\circ \varphi_k^{-1})]\,d\ell\notag\\
  &=\int_{\varphi_k^{-1}(\Omega_k)} m(\tilde{\nabla}_m f, \nabla_m g)\, d\mu_r,
 \end{align}
 where the last line is due to the fact that $m(\tilde{\nabla}_mf,\nabla_m g)=(\tilde{\nabla}_m f)g=\sum_{i,j=1}^r m^{ij}\partial_i f\partial_jg$.

 Now since $M$ is compact, there exists a smooth partition of unity $\sigma_k$ subordinate to the covering $\{U_k\}_{k\in K}$ (see Definition \ref{thm:partu}). Moreover, since $\varphi_k$ is a diffeomorphism, for any measurable $U\subset M$, there exist $K'\subseteq K$ and countable collection of measurable $\Omega_k\subset \varphi_k (U_k)$, such that $U=\cup_{k\in K'}(\varphi_k^{-1}(\Omega_k))$. Hence, applying \eqref{eq:wgrad2} to each $k\in K'$, one has by setting $\sum_{k\in K'}\sigma_k=1$
 \begin{align*}
  \int_U f\cdot \triangle_\mu g\,d\mu_r\
  &=\sum_{k\in K'} \int_{\varphi_k^{-1}(\Omega_k)} \sigma_k f \cdot \triangle_\mu g \, d\mu_r\\
  &=\sum_{k\in K'}\int_{\varphi_k^{-1}(\Omega_k)} m(\tilde{\nabla}_m (\sigma_kf), \nabla_m g)\, d\mu_r\\
  &=\int_{U}  m\left(\tilde{\nabla}_m \left(\sum_{k\in K'}\sigma_kf\right), \nabla_m g\right)\, d\mu_r\\
  &=\int_U m(\tilde{\nabla}_m f, \nabla_m g)\, d\mu_r,
 \end{align*}
 where we have used the linearity of $\tilde{\nabla}_m$ and the fact that $\textup{supp}(\sigma_k)\subset U_k$ to obtain the penultimate line.
\end{proof}

We introduce the weighted Sobolev space $W^{1,2}(M, m, \mu_r)$ of functions $f\in L^2(M, m, \mu_r)$, whose first order weak gradient exists in $L^2(M, m, \mu_r)$. We equip $W^{1,2}(M, m, \mu_r)$ with the inner-product $\langle f, g\rangle_{W^{1, 2}(M, m,  \mu_r)}=\int_M (m(\tilde{\nabla}_m f,\tilde{\nabla}_m g)+fg)\, d\mu_r$ for all $f, g\in W^{1, 2}(M, m, \mu_r)$, with the norm associated with $\langle \cdot,\cdot \rangle_{W^{1, 2}(M, m, \mu_r)}$ denoted by $\|\cdot\|_{W^{1, 2}(M, m, \mu_r)}$.

There exist embedding theorems and the completeness property for weighted Sobolev spaces on $\mathbb{R}^r$, and for the unweighted Sobolev spaces on Riemannian manifolds (see \cite{turesson00} and \cite{hebey96} respectively). We develop the corresponding results for the weighted Sobolev space $W^{1, 2}(M, m, \mu_r)$ defined as above. Let $(U, \varphi)$ be a chart on $M$. In the following, we first obtain the results of the desired properties in a local setting; i.e.\ the weighted Sobolev space $W^{1, 2}(U, m, \mu_r)$. One can then use the fact that $M$ is compact, and apply the standard partition of unity arguments to extend these local outcomes to global ones for $W^{1,2}(M, m, \mu_r)$.

Given a chart $(U, \varphi)$ on $M$. Set $\Omega=\varphi(U)$, and let $\ell_\mu$ be an absolutely continuous measure with density $h_\mu\circ \varphi^{-1}$ with respect to $\ell$, where $\ell$ is the Lebesgue measure on $\mathbb{R}^r$.  One has the weighted Sobolev space $W^{1, 2}(\Omega, \ell_\mu)$ for the open subset $\Omega\subset\mathbb{R}^r$; that is, the space $W^{1, 2}(\Omega, \ell_\mu)$ is equipped with the norm
\begin{equation}\label{eq:ecnorm}
\|f\circ \varphi^{-1}\|^2_{W^{1, 2}(\Omega, \ell_\mu)}=\int_\Omega \left(|f\circ \varphi^{-1}|^2+|\tilde{\nabla}_e(f\circ \varphi^{-1})|_e^2\right)\cdot (h_\mu\circ \varphi^{-1}) d\ell,
\end{equation}
for all $f\in L^2(U, \mu_r)$, and where $\tilde{\nabla}_e$ is the first order weak gradient with respect to the Euclidean metric $e$. Suppose the density $h_\mu$ of $\mu_r$ is an $A_2$ weight (i.e\ $h_\mu$ satisfies \eqref{eq:apw} when $p=2$). Clearly, $h_\mu$ is an $A_2$ weight restricted to the sub-domain $U$. Moreover, since $U$ has compact closure and $\varphi$ is a diffeomorphism, it is easy to verify that the density $h_\mu\circ \varphi^{-1}$ of $\ell_\mu$ is also an $A_2$ weight. Since the density of $\ell_\mu$ is an $A_2$ weight, the weighted Sobolev space $W^{1, 2}(\Omega, \ell_\mu)$ is a Hilbert space and $C^\infty(\Omega, \mathbb{R})$ is dense in $W^{1, 2}(\Omega, \ell_\mu)$ (see theorem 1 in \cite{goldshtein07}). We show by the following lemma that if $f\in W^{1, 2}(U, m, \mu_r)$, then $f\circ\varphi^{-1}\in W^{1,2}(\Omega,  \ell_\mu)$.

\begin{lemma}\label{thm:ecnorm}
Let $(U, \varphi)$ be a chart on $M$, and set $\Omega=\varphi(U)$. Denote by $\|\cdot \|_{W^{1, 2}(\Omega, \ell_\mu)}$ and $\|\cdot\|_{W^{1, 2}(U, m, \mu_r)}$ the norms on the weighted Sobolev spaces $W^{1, 2}(\Omega, \ell_\mu)$ and $W^{1, 2}(U, m, \mu_r)$ respectively. Then $\|f\circ \varphi^{-1} \|_{W^{1, 2}(\Omega, \ell_\mu)}$ and $\|f\|_{W^{1, 2}(U, m, \mu_r)}$ are equivalent for all $f\in W^{1, 2}(U, m, \mu_r)$
\end{lemma}
\begin{proof}
 This follows immediately from Lemma \ref{thm:volchart} and \ref{thm:gradchart}.
\end{proof}


 Due to Lemma \ref{thm:ecnorm}, one now has global completeness for $W^{1, 2}(M, m, \mu_r)$.

\begin{proposition}\label{thm:comW}
 Assume the density of $\mu_r$ is an $A_2$ weight. The weighted Sobolev space $W^{1, 2}(M, m, \mu_r)$ is complete.
\end{proposition}
\begin{proof}
First we show that the Sobolev spaces on any charts on $M$ are complete. Let $(U, \varphi)$ be a chart on $M$, and $f_j$ a Cauchy sequence
in $W^{1, 2}(U, m, \mu_r)$. Then $f_j\circ \varphi^{-1}$ is Cauchy in $W^{1, 2}(\Omega, \ell_{\mu})$ due to Lemma \ref{thm:ecnorm}, so  by the completeness of $W^{1, 2}(\Omega, \ell_\mu)$, the Cauchy sequence $f_j\circ \varphi^{-1}$ convergences to an element $f\circ \varphi^{-1}\in W^{1,2}(\Omega, \ell_\mu)$. Hence, the Cauchy sequence $f_j$ converges to $f$ in $W^{1, 2}(U, m, \mu_r)$.

Now, let $g_j$ be a Cauchy sequence in $W^{1, 2}(M, m, \mu_r)$, and $(U_i, \varphi_i)_{i\in I}$ an atlas on $M$. Since $M$ is compact, $\{U_i\}_{i\in I}$ is a finite cover for $M$. Hence, there exist a fixed $s\in I$ such that $W^{1, 2}(U_s, m, \mu_r)$ contains infinitely many terms of the sequence $g_j$. Let $g_{j_k}$ be a subsequence of $g_j$ contained entirely in $W^{1, 2}(U_s, m, \mu_r)$, then $g_{j_k}$ is Cauchy in $W^{1, 2}(U_s, m, \mu_r)$, so that $g_{j_k}$ converges to an element $g\in W^{1, 2}(U_s, m, \mu_r)$ by completeness. In particular, the Cauchy sequence $g_j$ converges to $g$ in $W^{1, 2}(M, m, \mu_r)$.
\end{proof}

We proceed to demonstrate that the space $W^{1, 2}(M, m, \mu_r)$ is approximated by smooth functions in $C^\infty(M, \mathbb{R})\cap W^{1, 2}(M, m, \mu_r)$. The key idea is to locally subject the functions in $W^{1, 2}(M, m, \mu_r)$ to \emph{mollification}.

\begin{definition}\label{def:conv}
Let $\Omega$ be an open subset of $\mathbb{R}^r$, and $q\in C_0^\infty(\mathbb{R}^r, \mathbb{R})$ be nonnegative such that $\textup{supp } (q) \subset E_1(0)$ and $\int_{\Omega}q \,d\ell=1$, where $E_1(0)$ is the open unit ball  centered at the origin in $\mathbb{R}^r$. We define a \emph{mollifier} by the function $q_\epsilon:=\epsilon^{-r}q(x/\epsilon)$. For all $f\in L^p(\Omega, \ell)$, $p\in [1, \infty)$, we call the convolution
\begin{equation}\label{eq:conv}
 q_\epsilon \star f(x):=\int_{\Omega}q_\epsilon(x-z) f(z)\,d\ell(z),
\end{equation}
the \emph{mollification} of $f$ by $q_\epsilon$.
\end{definition}

One has the following weighted version of the well known result $\tilde{\nabla}_e(q_\epsilon\star f)=q_\epsilon\star \tilde{\nabla}_e f$, and density theorem (Lemma 7.3 and Theorem 7.9 in \cite{gilbarg77} respectively).

\begin{theorem}[Theorem 2.1.4. \cite{turesson00}]\label{thm:do}
Let $\Omega$ be an open subset of $\mathbb{R}^r$, and $f\in L^p(\Omega, \ell_w)$, where $\ell_w$ is an absolutely continuous measure with respect to Lebesgue. Define $f_\epsilon:=q_\epsilon\star f$, where the mollifier $q_\epsilon$ and $\star$ are as in Definition \ref{def:conv}. For $p\in [1, \infty)$, if the density of $\ell_w$ is an $A_p$ weight, then $f_\epsilon\in C^\infty(\Omega, \mathbb{R})\cap L^p(\Omega, \ell_w)$, $\tilde{\nabla}_e f_\epsilon= q_\epsilon \star \tilde{\nabla}_e f$, and as $\epsilon\to 0$, $f_\epsilon \to f$ in $L^p(\Omega, \ell_w)$.
\end{theorem}



\begin{corollary}\label{thm:lcom}
Let $W^{1, 2}(M, m, \mu_r)$ be a weighted Sobolev space. Assume the density of $\mu_r$ is an $A_2$ weight.  The space $C^\infty(M, \mathbb{R})\cap W^{1, 2}(M, m, \mu_r)$ is dense in $W^{1, 2}(M, m, \mu_r)$.
\end{corollary}
\begin{proof}
 Let $f\in W^{1, 2}(M, m, \mu_r)$ and choose some $\gamma>0$. We will show that there is a $g\in C^\infty(M, \mathbb{R})\cap W^{1, 2}(M, m, \mu_r)$, such that $\|f-g\|^2_{W^{1, 2}(M, m, \mu_r)}<\gamma$. Let $(U_i, \varphi_i)_{i\in I}$ be an atlas on $M$, then pick local coordinates on $M$, such that the components of the metric tensor satisfy $\frac{1}{c}\delta_{jk}\leq m_{jk}(x)\leq c\delta_{jk}$ for some $1<c<\infty$, and each $x\in U_i$, $1\leq j,k\leq r$ (such coordinate exist due to the compactness of $M$, see \eqref{eq:metricb}). For each $1\leq i\leq r$, let $\ell_{\mu, i}$ be an absolutely continuous measure with density $h_\mu\circ \varphi_i^{-1}$ with respect to $\ell$. Since $M$ is compact, there exists a smooth partition of unity $\{\sigma_i\}_{i\in I}$ subordinate to the finite covering $\{U_i\}_{i\in I}$; i.e\ $\sigma_i$ is given by definition \ref{thm:partu}. Therefore, $\sigma_i f$ and its first-order weak gradient vanishes outside of $U_i$, hence $\sigma_if\in W^{1, 2}(U_i, m, \mu_r)$. Set $\Omega_i=\varphi_i(U_i)$ for each $i\in I$, then due to Lemma \ref{thm:ecnorm}, the fact that $\sigma_i f\in W^{1, 2}(U_i, m, \mu_r)$ implies $(\sigma_i f)\circ \varphi_i^{-1}\in W^{1, 2}(\Omega_i, \ell_{\mu,i})$. Consequently, both $(\sigma_if)\circ \varphi_i^{-1}$ and $\tilde{\nabla}_e\left((\sigma_i f)\circ \varphi_i^{-1}\right)$ are in $L^2(\Omega_i, \ell_{\mu,i})$ for each $i\in I$. Let $q_\epsilon$ and $\star$ be as in definition \ref{def:conv}, then by Theorem \ref{thm:do} applied to $(\sigma_i f)\circ\varphi_i^{-1}\in L^2(\Omega_i, \ell_{\mu, i})$ with $p=2$ and $\ell_w=\ell_{\mu,i}$ for each $i\in I$,
\begin{equation}\label{eq:lcom3}
\tilde{\nabla}_e\left(q_\epsilon\star\left((\sigma_if)\circ \varphi_i^{-1}\right)\right)=q_\epsilon\star\left(\tilde{\nabla}_e\left((\sigma_if)\circ \varphi_i^{-1}\right)\right),
\end{equation}
 and there exist $\epsilon_1>0$ such that
\begin{equation}\label{eq:lcom2}
\left\|q_\epsilon\star\left((\sigma_if)\circ \varphi_i^{-1}\right)-(\sigma_if)\circ \varphi_i^{-1}\right\|_{L^2(\Omega_i, \ell_{\mu,i})}< \frac{\gamma}{2c^{r/4}\cdot |I|^2}.
\end{equation}
for all $i\in I$. In addition, applying Theorem \ref{thm:do} to $\tilde{\nabla}_e\left((\sigma_i f)\circ\varphi_i^{-1}\right)\in L^2(\Omega_i, \ell_{\mu, i})$ with $p=2$ and $\ell_\mu=\ell_{\mu, i}$, one has $\epsilon_2>0$ such that
\begin{align}\label{eq:lcom4}
&\left\|\tilde{\nabla}_e\left(q_\epsilon\star\left((\sigma_if)\circ \varphi_i^{-1}\right)\right)-\tilde{\nabla}_e\left((\sigma_if)\circ \varphi_i^{-1}\right)\right\|_{L^2(\Omega_i, \ell_{\mu,i})}\notag\\
=&\left\|q_\epsilon\star\left(\tilde{\nabla}_e\left((\sigma_if)\circ \varphi_i^{-1}\right)\right)-\tilde{\nabla}_e\left((\sigma_if)\circ \varphi_i^{-1}\right)\right\|_{L^2(\Omega_i, \ell_{\mu, i})}\quad\mbox{by \eqref{eq:lcom3}}\notag\\
<& \frac{\gamma}{2c^{r/4+2}\cdot |I|^2},
\end{align}
for all $i\in I$.

Set
\begin{equation}\label{eq:lcom1}
\epsilon=\min\left\{\epsilon_1, \epsilon_2, \textup{dist}\left(\textup{supp}(q_\epsilon\star((\sigma_if)\circ \varphi_i^{-1}), \Omega_i\right)\right\},
\end{equation}
and let $f_{\epsilon, i}:=q_\epsilon\star\left((\sigma_if)\circ\varphi_i^{-1}\right)\circ \varphi_i$. Since $\epsilon$ satisfies \eqref{eq:lcom1}, the function $f_{\epsilon, i}$ and its first-order weak gradient $\tilde{\nabla}_m f_{\epsilon, i}$ vanish outside of $U_i$ for all $i\in I$. Moreover, since $|I|$ is finite, $f_\epsilon:=\sum_{i\in I} f_{\epsilon,i}\in C^\infty(M, \mathbb{R})\cap W^{1,2}(M, m, \mu_r)$.

Set $g=f_\epsilon$, then by Lemma \ref{thm:volchart} and the inequality \eqref{eq:lcom2}
\begin{align}\label{eq:lcom5}
\|g-f\|_{2, m, \mu}
&=\left\|\sum_{i\in I} f_{\epsilon,i}-\sigma_i f\right\|_{2, m, \mu}\notag\\
&\leq \sum_{i\in I}\left(\int_M \left|f_{\epsilon, i}-\sigma_i f \right|^2\, d\mu_r\right)^\frac{1}{2}\quad\mbox{by triangle inequality}\notag\\
&=\sum_{i\in I} \left(\int_{U_i} |f_{\epsilon, i}-\sigma_if|^2\, d\mu_r\right)^\frac{1}{2}\notag\\
&\leq \sum_{i\in I}c^{r/4}\left(\int_{\Omega_i} \left|q_\epsilon\star\left((\sigma_if)\circ \varphi_i^{-1}\right)-(\sigma_if)\circ \varphi_i^{-1}\right|^2\cdot (h_\mu\circ \varphi_i^{-1})\,d\ell\right)^\frac{1}{2}\notag\\
&< \sqrt{\gamma/2}.
\end{align}
Similarly, by Lemma \ref{thm:gradchart}, and the inequality \eqref{eq:lcom4}
\begin{align}\label{eq:lcom6}
\|\tilde{\nabla}_m g-\tilde{\nabla}_m f\|_{2, m, \mu}
&=\left\|\sum_{i\in I}\tilde{\nabla}_m f_\epsilon-\sigma_i f\right\|_{2, m, \mu}\notag\\
&\leq \sum_{i\in I} \left(\int_{U_i} |\tilde{\nabla}_m f_{\epsilon, i}-\tilde{\nabla}_m(\sigma_i f)|_m^2\, d\mu_r\right)^\frac{1}{2}\notag\\
&\leq \sum_{i\in I}c^{r/4+1/2}\left(\int_{\Omega_i}\left|\tilde{\nabla}_e\left(q_\epsilon\star\left((\sigma_if)\circ \varphi_i^{-1}\right)\right)-\tilde{\nabla}_e\left((\sigma_if)\circ \varphi_i^{-1}\right)\right|_e^2\cdot (h_\mu\circ \varphi_i^{-1})\,d\ell\right)^\frac{1}{2}\notag\\
&< \sqrt{\gamma/2}.
\end{align}
Thus, $\|g-f\|^2_{W^{1, 2}(M, m, \mu_r)}<\gamma$.
\end{proof}

As before, let $\ell_\mu$ be an absolutely continuous measure with respect to $\ell$. The \emph{Hardy-Littlewood maximal operator} $\mathcal{H}$ is a non-linear operator on locally integrable functions $f\in L^1_{\textup{loc}}(\mathbb{R}^r, \ell)$ defined by
\begin{equation}\label{eq:lhwf}
\mathcal{H}f(x)=\sup_{\rho>0}\frac{1}{\ell(E_\rho(x))}\int_{E_\rho(x)} |f(y)|\,d\ell(y),
\end{equation}
where $E_\rho(x)$ is the Euclidean ball centered at $x$ with radius $\rho$. If the density of the measure $\ell_\mu$ is an  $A_p$ weight, then $\mathcal{H}$ is bounded as an operator from $L^p(\mathbb{R}^r, \ell_w)$ to  $L^p(\mathbb{R}^r, \ell_w)$ for $1<p<\infty$ (see Theorem 1, p.201 in \cite{stein93}, or Theorem 1.2.3 in \cite{turesson00}). This property of $A_p$ weights forms an essential argument for Theorem \ref{thm:do}.

\begin{lemma}\label{thm:wcom}
Let $K\subset \Omega\subset \mathbb{R}^r$, where $K$ is compact and $\Omega$ is open and bounded. Let $W^{1, 2}(\Omega, \ell_w)$ be a weighted Sobolev space. Assume the density $\ell_w$ is an $A_2$ weight. Suppose $f_i$ is a sequence in $W^{1,2}(\Omega, \ell_w)$ with support $K$. Then there exists a subsequence $f_{i_j}$, and some $f\in L^2(\Omega, \ell_w)$, such that
\begin{equation}\label{eq:wcom}
\int_{K} (f_{i_j}-f)^2\, d\ell_w\to 0
\end{equation}
as $j\to \infty$.
\end{lemma}
\begin{proof}
Let $f_{\epsilon, i}=q_\epsilon\star f_i$, where $q_\epsilon$ and $\star$ are as in Definition \ref{def:conv}. Denote by $\|\cdot\|_{2, \ell_w}$ the $L^2$-norm associated with $L^2(K, \ell_w)$. First we shall show that
\begin{equation}\label{eq:wcom1}
\|f_{\epsilon, i}-f_i\|_{2, \ell_w}\to 0,
\end{equation}
uniformly with respect to $i$. Due to Theorem \ref{thm:do}, it is sufficient to proof \eqref{eq:wcom1} for $f_i\in C^\infty(\Omega, \mathbb{R})\cap W^{1, 2}(\Omega, \ell_w)$. By a change of variable $y=(x-z)/\epsilon$ and using the facts $\int_{\mathbb{R}^r} q\,d\ell=1$, $\textup{supp}(q)\subset E_1(0)$, one has
\begin{align}\label{eq:wcom4}
f_{\epsilon, i}(x)-f_i(x)
&=\int_{\mathbb{R}^r} q_\epsilon(x-z)f_i(z)\,d\ell(z)-f_i(x)\notag\\
&=\epsilon^{-r}\int_{E_\epsilon(x)} q\left(\frac{x-z}{\epsilon}\right)f_i(z)\,d\ell(z)-\int_{|y|< 1}q(y)f_i(x)\,d\ell(y)\notag\\
&= \int_{|y|< 1} q(y)[f_i(x-\epsilon y)-f_i(x)]\,d\ell(y)\notag\\
&\leq \|q\|_{\infty}\cdot\int_{|y|< 1}[f_i(x-\epsilon y)-f_i(x)]\,d\ell(y)\quad\mbox{by H\"older's inequality}\notag\\
&= \|q\|_{\infty}\cdot\epsilon^r\int_{E_\epsilon(x)}[f_i(z)-f_i(z+\epsilon y)]\,d\ell(z).
\end{align}
As a consequence of \eqref{eq:wcom4}, one has
\begin{align}\label{eq:wcom0}
 \|f_{\epsilon, i}-f_i\|^2_{2, \ell_w}
 &\leq \|q\|^2_{\infty}\cdot\int_K\bigg|\epsilon^r\int_{E_\epsilon(x)}\left|f_i(z)-f_i(z+\epsilon y)\right|\,d\ell(z)\bigg|^2\, d\ell_w(x)\notag\\
 &\leq \|q\|^2_\infty\cdot \int_{\mathbb{R}^r}\bigg|\epsilon^r\left[\ell(E_\epsilon(x))\right]\cdot \mathcal{H}(f_i(x)-f_i(x+\epsilon y))\bigg|^2\,d\ell_w(x)\notag\\
 &=\|q\|^2_\infty\cdot \pi^2 \cdot \int_{\mathbb{R}^r}\bigg|\mathcal{H}(f_i(x)-f_i(x+\epsilon y))\bigg|^2\,d\ell_w(x),
 \end{align}
where $\mathcal{H}$ is defined as in \eqref{eq:lhwf}. Furthermore, by using the fact that the density of $\ell_w$ is an $A_p$ weight, the operator $\mathcal{H}: L^2(\mathbb{R}^r, \ell_w)\to L^2(\mathbb{R}^r, \ell_w)$ is bounded. This implies
 \begin{align}\label{eq:wcom0b}
 \textup{RHS of \eqref{eq:wcom0}}
 &\leq \|q\|^2_\infty\cdot \pi^2 \cdot \int_{\mathbb{R}^r}\left|\mathcal{H}(f_i(x)-f_i(x+\epsilon y)\right|^2\,d\ell_w(x)\notag\\
 &\leq \|q\|^2_\infty\cdot \pi^2 \cdot C\int_{\mathbb{R}^r}\left|f_i(x)-f_i(x+\epsilon y)\right|^2\,d\ell_w(x),
\end{align}
where the constant $C$ depends only on $r$, $p$ and the $A_p$ constant of $w$; the constant $C$ is uniform with respect to $i$. Since $f_i$ is continuous independent of $i$, by \eqref{eq:wcom0}-\eqref{eq:wcom0b} one has the convergence $f_{\epsilon, i}\to f_i$ in $L^2(\Omega, \ell_w)$ uniformly with respect to $i$.

Due to \eqref{eq:wcom1}, we can now pick a subsequence of $f_i$, so that for any fixed $\gamma>0$, there exist an $\epsilon$ sufficiently small such that
\begin{equation}\label{eq:wcom2}
\|f_{\epsilon, i_j}-f_{i_j}\|_{2, \ell_w}\leq \gamma/2
\end{equation}
for all $i_j\geq 1$. Furthermore, since the density of $\ell_w$ is an $A_2$ weight, by a straightforward modification of Lemma \ref{thm:loc}, one has $L^2(\Omega, \ell_w)\subset L^1_{\text{loc}}(\Omega, \ell)$. Hence the sequence $f_i$ belongs to $L^1_{\text{loc}}(\Omega
, \ell)$, therefore
\begin{equation*}
\int_\Omega|f_i|\,d\ell=\int_{K} |f_i|\,d\ell< \infty,
\end{equation*}
which implies
\begin{align*}
\sup_{x\in \Omega}|f_{\epsilon, i}(x)|
 &=\sup_{x\in \Omega} \left|q_\epsilon\star f_{i}(x)\right|\\
 &= \sup_{x\in \Omega}\left|\int_\Omega q_\epsilon(x-z)\cdot f_i(z)\,d\ell(z)\right|\\
 &\leq \|q_\epsilon\|_\infty\cdot \int_{\Omega} |f_i|\,d\ell<\infty,
\end{align*}
so that $f_{\epsilon,i}$ is uniformly bounded on $\Omega$. Similarly, by using Leibniz's rule for differentiating under the integral sign, one has
\begin{align*}
 \sup_{x\in \Omega} |\nabla_e f_{\epsilon,i}(x)|_e
 &=\sup_{x\in \Omega}|\nabla_e q_\epsilon\star f_{i}(x)|_e\\
 &=\sup_{x\in \Omega} \left|\nabla_e \left(\int_\Omega q_\epsilon(x-z)\cdot f_i(z)\,d\ell(z)\right)\right|_e\\
 &\leq \sup_{x\in \Omega}\int_{\Omega} |\nabla_e q_\epsilon(x-z)|_e\cdot |f_i(z)|\,d\ell(z)\\
 &\leq \|\nabla_e q_\epsilon\|_\infty\cdot\int_{\Omega} |f_i|\,d\ell<\infty,
\end{align*}
which implies $f_{\epsilon, i}$ is equicontinuous on $\Omega$. Therefore, by the Arzela-Ascoli theorem (Theorem 11.28 in \cite{rudin74}), there exist a subsequence $f_{\epsilon, i_j}$ that convergences uniformly on every compact subset of $\Omega$. In particular,  there is an $f_{\epsilon}$ such that
\begin{align}\label{eq:wcom3}
 \lim_{j, k\to \infty}\|f_{i_j}-f_{i_k}\|^2_{2, \ell_w}
 &\leq  \lim_{i, k\to \infty}\left\{ \|f_{i_j}-f_{\epsilon, i_j}\|_{2, \ell_w}^2+\|f_{\epsilon, i_j}-f_\epsilon\|_{2, \ell_w}^2+\|f_\epsilon-f_{\epsilon, i_k}\|_{2, \ell_w}^2+\|f_{\epsilon, i_k}-f_{i_k}\|_{2, \ell_w}^2\right\}\notag\\
 &\leq \gamma/2+0+0+\gamma/2=\gamma,
\end{align}
where the convergence of the first and last term on the RHS was handled by \eqref{eq:wcom2}.
Hence $f_{i_j}$ is a Cauchy sequence in $L^2(K, \ell_\mu)$. Therefore, by the completeness of $L^2$ spaces, the Cauchy sequence $f_{i_j}$ convergences to some $f$ in $L^2(K, \ell_w)$.
\end{proof}

One now has the weighted version of the well known Sobolev compactness embedding theorem for $\mathbb{R}^r$, which applies to $W^{1, 2}(M, m, \mu_r)$.

\begin{theorem}[Rellich Compactness]\label{thm:rcom}
Let $W^{1, 2}(M, m, \mu_r)$ be a weighted Sobolev space. Assume the density of $\mu_r$ is an $A_2$ weight. Then the embedding $W^{1, 2}(M, m, \mu_r)\hookrightarrow  L^2(M, m, \mu_r)$ is compact.
\end{theorem}
\begin{proof}
Let $f_j$ be a sequence in $W^{1, 2}(M, m, \mu_r)$,  and $(U_i, \varphi_i)_{i\in I}$ an atlas on $M$.  As in the proof of Corollary \ref{thm:lcom}, pick local coordinates on $M$ such that the components of the metric tensor satisfy $\frac{1}{c}\delta_{sp}\leq m_{sp}(x)\leq c\delta_{sp}$ for some $1<c<\infty$, and all $x\in U_i$, $i\in I$, $1\leq s, p\leq r$. Furthermore, let $\{\sigma_i\}_{i\in I}$ be a partition of unity subordinate to the finite covering $\{U_i\}_{i\in I}$. For each $i\in I$, set $\Omega_i=\varphi_i(U_i)$. One has $\sigma_if_j\in W^{1, 2}(U_i, m, \mu_r)$, so by Lemma \ref{thm:ecnorm} the sequence $(\sigma_if_j)\circ\varphi_i^{-1}$ belongs to $W^{1, 2}(\Omega_i, \ell_{\mu,i})$, where $\ell_{\mu, i}$ is an absolutely continuous measure with density $h_\mu\circ \varphi_i^{-1}$ with respect to $\ell$. Moreover, the compactness of the closure of $M$ implies $(\sigma_if_j)\circ\varphi_i^{-1}$ has compact support $C_i\subset \Omega_i$. Therefore, for each $i\in I$ one can apply Lemma \ref{thm:wcom} with $\ell_w=\ell_{\mu, i}$ to obtain a subsequence $(\sigma_i f_{j_k})\circ \varphi^{-1}_i$, and a function $g_i$ in  $L^2(\Omega_i, \ell_{\mu, i})$, such that for any $\gamma>0$, there exist a $K(\gamma)\in \mathbb{N}$ with
\begin{equation}\label{eq:rcom1}
\left(\int_{C_i} \left|(\sigma_if_{j_k})\circ \varphi_i^{-1}-g_i\right|^2\, d\ell_{\mu, i}\right)^\frac{1}{2}<\frac{\gamma}{c^{r/4}\cdot |I|^2}
\end{equation}
for all $k\geq K(\gamma)$.

Since each $\varphi_i$ is a diffeomorphism and $g_i\in L^2(\Omega_i, \ell_{\mu, i})$, each $g_i\circ \varphi_i$ belongs to $L^2(U_i, m, \mu_r)$. Extend $g_i\circ \varphi_i$ to $\tilde{g}_i\in L^2(M, m, \mu_r)$, by setting
\begin{equation*}
\tilde{g}_i(x) :=
\left\{
	\begin{array}{ll}
		g_i\circ \varphi_i(x) &\mbox{} x\in U_i \\
		0 &\mbox{} x\in  M\setminus U_i
	\end{array},
\right.
\end{equation*}
for each $i\in I$. Then by a similar argument as in \eqref{eq:lcom5}
\begin{align*}
\left\|f_{j_k}-\sum_{i\in I}\tilde{g}_i\right\|_{2, m, \mu}
&\leq \sum_{i\in I}\left(c^{r/2}\int_{\Omega_i}\left|\sigma_i f_{j_k}-\tilde{g}_i\right|^2\circ\varphi_i^{-1}\, d\ell_{\mu, i}\right)^\frac{1}{2}\\
&=\sum_{i\in I}c^{r/4}\left(\int_{C_i}\left|(\sigma_if_{j_k})\circ \varphi_i^{-1}-g_i\right|^2\, d\ell_{\mu, i}\right)^\frac{1}{2}\\
&<\gamma,
\end{align*}
where the inequality on the last line is due to \eqref{eq:rcom1}. Since $\tilde{g}_i\in L^2(U_i, m, \mu_r)$ and $|I|$ is finite, we have $\sum_{i\in I} \tilde{g}_i\in L^2(M, m, \mu_r)$; this completes the proof of the theorem.
\end{proof}

\begin{lemma}[Poincar\'e inequality]\label{thm:gpc}
Let $W^{1, 2}(M, m, \mu_r)$ be a weighted Sobolev space, and denote by $\alpha(f)$ the weighted mean of $f$; i.e $\alpha(f)=\int_M f\cdot h_\mu\omega_m^r$. Assume the density of $\mu_r$ is an $A_2$ weight. There is a constant $K$ depending on $r$ and $M$ such that
\begin{equation}\label{eq:gpc}
 \|f-\alpha(f)\|_{2, m, \mu}\leq K\|\tilde{\nabla}_m f\|_{2, m, \mu},
\end{equation}
for all $f\in W^{1,2}(M, m, \mu_r)$.
\end{lemma}

\begin{proof}
We follow a standard argument as in corollary of Theorem 5 on p.194, \cite{mcowen96}. Suppose the inequality \eqref{eq:gpc} is false, then due to Corollary \ref{thm:lcom}, there exists a sequence in $f_k\in C^\infty(M, \mathbb{R})\cap W^{1, 2}(M, m, \mu_r)$, such that $\|f_k-\alpha(f_k)\|_{2, m, \mu}> k\|\nabla_m f_k\|_{2, m, \mu}$ for $k=1, 2, \ldots$. Define
\begin{equation*}
 g_k=\frac{f_k-\alpha(f_k)}{\|f_k-\alpha(f_k)\|_{2,m, \mu}},
\end{equation*}
then $\|g_k\|_{2, m, \mu}=1, \alpha(g_k)=0$ and $\|\nabla_mg_k\|_{2, m,\mu}\leq 1/k$. In particular, $g_k$ is a bounded sequence in $W^{1, 2}(M, m, \mu_r)$. Hence, by Theorem \ref{thm:rcom} there exists a subsequence $g_{k_j}\in W^{1, 2}(M, m, \mu_r)$, which converges to some $g$ in $L^2(M, m, \mu_r)$. One has
\begin{equation}\label{eq:gpc1}
\|g\|_{2, m, \mu}=1,
\end{equation}
and
\begin{equation}\label{eq:gpc2}
\alpha(g)=\int_M g\cdot h_\mu\omega_m^r=\lim_{j\to \infty}\int_M g_{k_j}\cdot h_\mu\omega_m^r=\lim_{j\to\infty}\alpha(g_{k_j})=0
\end{equation}
and $\lim_{j\to \infty}\|\nabla_m g_{k_j}\|_{2, m, \mu}=0$.

Now, for any $\psi\in C^\infty_0(M, \mathbb{R})$, the weak gradient of $g$ satisfies
\begin{align*}
 \int_M m(\tilde{\nabla}_m g, \nabla_m \psi)\, d\mu_r
 &=-\int_M g\triangle_\mu\psi\, d\mu_r\\
 &=-\lim_{j\to \infty}\int_M g_{k_j}\triangle_\mu \psi\, d\mu_r\\
 &=\lim_{j\to \infty}\int_M m(\nabla_m g_{k_j}, \nabla_m \psi)\, d\mu_r\\
 &\leq \lim_{j\to \infty} \|\nabla_m g_{k_j}\|_{2, m, \mu}\cdot \|\nabla_m \psi\|_{2, m, \mu}\\
 &=\left(\lim_{j\to \infty} \frac{1}{k_j}\right)\|\nabla_m \psi\|_{2, m, \mu}=0.
\end{align*}
Therefore
\begin{equation}\label{eq:gpc3}
\tilde{\nabla}_m g=0.
\end{equation}
But since $M$ is connected, \eqref{eq:gpc2} and \eqref{eq:gpc3} implies $g$ is the zero function, which contradicts \eqref{eq:gpc1}.
\end{proof}

\section{The proof of Theorem \ref{thm:dff}}

To obtain the inequality $\mathbf{s}^D\leq\mathbf{h}^D$, let $\Gamma$ be a compact, connected $C^\infty$ hypersurface in $M$ that disconnects $M$ into two open disjoint subsets $M_1$ and $M_2$.  
Let $\textup{dist}_m(x_1,x_2)$ denote the Riemannian distance function with respect to the metric tensor $m$ between the points $x_1$ and $x_2$ in $M$, then define $U_\epsilon:=\{x\in M:\textup{dist}_m(x, \Gamma)<\epsilon\}$ for $\epsilon>0$.

Consider the set of functions
\begin{equation}\label{eq:fe}
f_\epsilon(x) :=
\left\{
	\begin{array}{ll}
		1,  & \mbox{} x \in M_1 \setminus U_\epsilon \\
		-1, & \mbox{} x \in M_2 \setminus U_\epsilon\\
		(1/\epsilon)\textup{dist}_m(x,\Gamma), &\mbox{} x\in M_1\cap U_\epsilon\\
		-(1/\epsilon)\textup{dist}_m(x,\Gamma), &\mbox{} x\in M_2\cap U_\epsilon\\
	\end{array}.
\right.
\end{equation}
In the following, we obtain an upper bound for $\mathbf{s}^D$ by locally approximating functions in $C^\infty(M, \mathbb{R})$ by $f_\epsilon$. 

\begin{lemma}\label{thm:molli}
Let $\mathcal{L}$, $\mathbf{s}^D$ and $f_\epsilon$ be defined by \eqref{eq:pfw}, \eqref{eq:sd} and \eqref{eq:fe} respectively. If the density of $\mu_r$ is $C^1$, then for $\epsilon>0$ sufficiently small, one has
\begin{equation}\label{eq:dff1}
\mathbf{s}^D\leq\frac{\|\nabla_m f_\epsilon\|_{1, m, \mu}+\|\nabla_n \mathcal{L}f_\epsilon \|_{1, n, \nu}}{2\inf_\beta\|f_\epsilon-\beta\|_{1, m, \mu}}.
\end{equation}
\end{lemma}

\begin{proof}
We claim the existence of $g\in C^\infty(M, \mathbb{R})$, such that the terms $\nabla_m f_\epsilon$ and $f_\epsilon-\beta$ are approximated by $\nabla_m g$ and $g-\beta$ respectively in the norm $\|\cdot \|_{1, m, \mu}$, and the term $\nabla_m \mathcal{L}f$ is approximated by $\nabla_n \mathcal{L}g$ in the norm $\|\cdot\|_{1, n, \nu}$. In particular, due to these smooth approximations and the definition of $\mathbf{s}^D$, one immediately obtains the required inequality \eqref{eq:dff1}.

Let $(U_i, \varphi_i)_{i\in I}$ be an atlas of $M$, and set $\Omega_i=\varphi_i(U_i)$ for each $i\in I$. For each $i\in I$, let $\ell_{\mu, i}$ be an absolutely continuous measure with density $h_\mu\circ \varphi_i^{-1}$ with respect to Lebesgue measure $\ell$. Since $M$ is compact, there exist a smooth partition of unity $\{\sigma_i\}_{i\in I}$ subordinate to the finite covering $\{U_i\}_{i\in I}$. Moreover, one can verify that $f_\epsilon$ is a Lipschitz function in $L^1(M, m, \mu_r)$. Therefore $(\sigma_if_\epsilon)\circ \varphi_i^{-1}$ is Lipschitz in $L^1(\Omega_i, \ell_{\mu, i})$ for each $i\in I$. It follows that the restriction of $(\sigma_if_\epsilon)\circ \varphi_i^{-1}$ to any line in $\Omega_i$ is absolutely continuous, which implies all partial derivatives of $(\sigma_if_\epsilon)\circ \varphi_i^{-1}$ exist almost everywhere on $\Omega_i$ (see Theorem 7.20 in \cite{rudin74}). Therefore, the Euclidean gradient $\nabla_e ((\sigma_if_\epsilon)\circ \varphi_i^{-1})\in L^1(\Omega_i, \ell_{\mu, i})$ for each $i\in I$.

Set $f_{\delta, \epsilon}:=q_\delta\star f_\epsilon$, where $q_\delta$ and $\star$ are as in Definition \ref{def:conv}. Then by straightforward modifications to the arguments used in Corollary \ref{thm:lcom} to obtain \eqref{eq:lcom5} and \eqref{eq:lcom6},
one can obtain for any $\gamma>0$, a $\delta>0$ chosen analogously to \eqref{eq:lcom1} such that $f_{\delta, \epsilon}\in C^\infty(M, \mathbb{R})$,
\begin{equation}\label{eq:molli5}
\|\nabla_m f_{\delta, \epsilon}-\nabla_m f_\epsilon\|_{1, m, \mu}<\gamma
\end{equation}
and
\begin{equation*}\label{eq:molli7}
\|(f_{\delta, \epsilon}-\beta)-(f_\epsilon-\beta)\|_{1, m, \mu}
=\|f_{\delta, \epsilon}-f_\epsilon\|_{1, m, \mu}
<\gamma.
\end{equation*}
 Finally, since $T$ is a diffeomorphism and $h_\mu$ is $C^1$, $\mathcal{L}f_\epsilon$ is Lipschitz in $L^1(N, n, \nu_r)$. Thus, the approximation of $\nabla_n \mathcal{L}f_\epsilon$ by $\nabla_n \mathcal{L}f_{\delta, \epsilon}$ in the norm $\|\cdot\|_{1, n, \nu}$ can be obtained analogously to \eqref{eq:molli5}. Thus, setting $g=f_{\delta, \epsilon}$ we are done.
\end{proof}

To complete the proof of Theorem \ref{thm:dff}, we show that the RHS of \eqref{eq:dff1} is bounded above by $\mathbf{h}^D$ as $\epsilon\to 0$. In order to show such convergence holds, we require additional results concerning the connection between $\mu_r(U_\epsilon)$ and $\mu_{r-1}(\Gamma)$.

Suppose $\epsilon$ is smaller than the injectivity radius of each point $x\in \Gamma$, and recall that $U_\epsilon:=\{x\in M:\textup{dist}(x, \Gamma)< \epsilon\}$ are open subsets of $M$. Since $M$ is compact, the closure of $U_\epsilon$ is a compact subset of $M$. Due to the compactness of $\overline{U_\epsilon}$ and the size of $\epsilon$, by the Hopf-Rinow theorem $\overline{U_\epsilon}$ is geodesically complete \cite{carmo76}. 
This implies that the signed distance function  $f:U_\epsilon\to \mathbb{R}$ defined by
\begin{equation}\label{eq:signdist}
f(x) :=
\left\{
	\begin{array}{ll}
		\textup{dist}_m(x, \Gamma) &\mbox{} x\in M_1 \\
		-\textup{dist}_m(x, \Gamma) &\mbox{} x\in M_2 \\
		0						     &\mbox{} x\in \Gamma
	\end{array},
\right.
\end{equation}
is smooth on $U_\epsilon\setminus \Gamma$, and $|\nabla_m f|_m=1$ (Proposition 2.1 \cite{sakai96}).

The following concerns the regularity of the co-dimensional one measure $\mu_{r-1}$ on the level surfaces of $U_\epsilon$.

\begin{lemma}\label{thm:contmu}
Let $\Gamma$ be a $C^\infty$ hypersurface in $M$ that disconnects $M$ into two disjoint open subsets $M_1$ and $M_2$. Define $\Gamma_t:=\{x\in M:\textup{dist}_m(x, \Gamma)=t\}$, and fix $\epsilon$ to be smaller than the injectivity radius of each point $x\in \Gamma$. If the density of $\mu_r$ is continuous, then the real valued function $A$ given by
\begin{equation*}
A(t):=\mu_{r-1}(\Gamma_t),
\end{equation*}
is continuous on the interval $[0, \epsilon]$.
\end{lemma}

\begin{proof}
Let $f:M\to \mathbb{R}$ be the signed distance function as in \eqref{eq:signdist}, and let $U_t:=\{x\in M:\textup{dist}_m(x, \Gamma)< t\}$. Fix $t_0\in (0, \epsilon)$, then $\Gamma_{t_0}$ is in $U_\epsilon\setminus \Gamma$. Hence $f$ is $C^\infty$ restricted to $\Gamma_{t_0}$, and $df(x)\neq 0$ for each $x\in \Gamma_{t_0}$. Therefore, by the implicit function theorem there exist open neighborhoods $\mathcal{O}_x$ about each point $x\in\Gamma_{t_0}$, and local coordinates $(x_1, x_2, \ldots, x_{r-1})$ for $\Gamma_{t_0}$, such that $(x_1, x_2,\ldots, x_{r-1}, f)$ are local coordinates on $\mathcal{O}_x$. Let $G_m$ be the $r\times r$ matrix with entries $m_{ij}$ in the coordinates $(x_1, x_2, \ldots, x_{r-1}, f)$. Then the volume form on $\mathcal{O}_x$ is given by
\begin{equation*}
 \omega_m^r=\sqrt{\det(G_m)}\cdot dx_1\wedge dx_2\ldots\wedge dx_{r-1}\wedge df.
\end{equation*}
Moreover, 
by a combination of the Stokes\rq{} and divergence theorem (see p.122, \cite{spivak65} and p.7, equation (38) \cite{chavel84} respectively), one has
\begin{equation*}
\int_{\Gamma_{t_0}}\omega_m^{r-1}=\int_{\Gamma_{t_0}}m(\mathbf{n}, \mathbf{n})\cdot\omega_m^{r-1}=\int_{U_{t_0}}\divg_m \mathbf{n}\cdot\omega_m^r=\int_{U_{t_0}}d(i(\mathbf{n})\omega_m^r)=\int_{\Gamma_{t_0}}i(\mathbf{n})\omega_m^r,
\end{equation*}
where $\mathbf{n}$ is the unit normal bundle along $\Gamma_{t_0}$. Hence $\omega_m^{r-1}=i(\mathbf{n})\omega_m^r$ for all $x\in \Gamma_{t_0}$.

Now, since $f=t_0$ along $\Gamma_{t_0}$, the vector $\nabla_m f$ is normal to the hypersurface $\Gamma_{t_0}$; which implies $\mathbf{n}=\nabla_m f/|\nabla_m f|_m$, and $dx_i(\nabla_m f)=0$ for $i=1, \ldots r-1$. Therefore
\begin{align}\label{eq:contmu1}
 \restr{\omega_m^{r-1}}{\Gamma_{t_0}}=\restr{i(\mathbf{n})\omega_m^r}{\Gamma_{t_0}}
 &=\restr{\sqrt{\det G_m}\cdot i(\mathbf{n})(dx_1\wedge dx_2\ldots\wedge dx_{r-1}\wedge df)}{{\Gamma_{t_0}}}\notag\\
 &=\restr{(-1)^{r}\sqrt{\det{G_m}}\cdot\frac{df(\nabla_m f)}{|\nabla_m f|_m}\cdot dx_1\wedge dx_2\ldots\wedge dx_{r-1}}{{\Gamma_{t_0}}}\notag\\
 &=\restr{(-1)^{r}\sqrt{\det{G_m}}\cdot |\nabla_m f|_m\cdot dx_1\wedge dx_2\ldots\wedge dx_{r-1}}{{\Gamma_{t_0}}},
\end{align}
where the penultimate equality is due to the Leibniz rule applied to the interior product, and the fact that $dx_i(\nabla_m f)=0$ for $i=1, \ldots, r-1$.

To complete the proof, we note that $|\nabla_m f|=1$ because $f$ is the signed distance function, $h_\mu$ is continuous by assumption, and $G_m$ is smooth since $m$ is smooth. Hence $h_\mu\restr{\omega_m^{r-1}}{\Gamma_t}$ is a continuous density for all $0< t< \epsilon$. Therefore $A(t)=\mu_{r-1}(\Gamma_t)=\int_{\Gamma_t} h_\mu\omega_m^{r-1}$ is continuous on $[0, \epsilon]$.
\end{proof}

\begin{lemma}\label{thm:alw}
Let $\Gamma$ be a compact, connected $C^\infty$ hypersurface in $M$. Define $U_\epsilon:=\{x\in M: \textup{dist}_m(x, \Gamma)<\epsilon\}$ for some $\epsilon>0$. Assume the density of $\mu_r$ is continuous. One has
 \begin{equation}\label{eq:alw1}
  \lim_{\epsilon\to 0} \frac{1}{\epsilon} \mu_{r}(U_\epsilon)=2\mu_{r-1}(\Gamma).
 \end{equation}
\end{lemma}
\begin{proof}

 Let $f$ be the signed distance function as in \eqref{eq:signdist}, and $\Gamma_t=\{x\in M:\textup{dist}_m(x, \Gamma)=t\}$. Then $|\nabla_m f|_m=1$, and $f$ is $C^\infty$ on $U_\epsilon\setminus \Gamma$. Hence, by the co-area formula \eqref{eq:ca}
 \begin{align}\label{eq:alw2}
  \notag \lim_{\epsilon\to 0} \frac{1}{\epsilon}\mu_r(U_\epsilon)
 \notag &=\lim_{\epsilon\to 0} \frac{1}{\epsilon}\int_{U_\epsilon}  |\nabla_m f|_m\cdot h_\mu\omega_m^r\\
 \notag  &=\lim_{\epsilon\to 0} \frac{1}{\epsilon}\int_{-\epsilon}^\epsilon \left(\int_{f^{-1}\{t\}} h_\mu\cdot\omega_m^{r-1}\right)dt\\
 \notag  &=\lim_{\epsilon\to 0} \frac{2}{\epsilon}\int_{0}^\epsilon \left(\int_{\Gamma_t} h_\mu\cdot\omega_m^{r-1}\right)dt\\
  &=\lim_{\epsilon\to 0} \frac{2}{\epsilon}\int_0^\epsilon \mu_{r-1}(\Gamma_t)dt.
 \end{align}
Take $\epsilon$ to be smaller than the injectivity radius of each $x\in \Gamma$. Since $h_\mu$ is continuous, by Lemma \ref{thm:contmu} the function $A(t):=\mu_{r-1}(\Gamma_t)$ is continuous on the interval $[0, \epsilon]$. Thus, one can apply the fundamental theorem of calculus to the last line of \eqref{eq:alw2} to obtain
\begin{equation*}
 \mbox{RHS of \eqref{eq:alw2}}=2\lim_{\epsilon\to 0} \frac{a(\epsilon)-a(0)}{\epsilon}=2A(0)=2\mu_{r-1}(\Gamma),
\end{equation*}
where $a(t)$ is the anti-derivative of $A(t)$.
\end{proof}

Now, to obtain the inequality $\mathbf{s}^D\leq \mathbf{h}^D$ via Lemma \ref{thm:molli}, we start with the term $\|\nabla_m f_\epsilon\|_{1, m, \mu}$ on the numerator of \eqref{eq:dff1}. Note that $f_\epsilon$ is constant on $M\setminus U_\epsilon$, which implies $\nabla_m f_\epsilon(x) = 0$ for all $x\in M\setminus U_\epsilon$. But if $x\in U_\epsilon$, then $|\nabla_m f_\epsilon|_m=\frac{1}{\epsilon}|\nabla_m (\textup{dist}(x, \Gamma))|_m=\frac{1}{\epsilon}$ for $\epsilon$ smaller than as in Lemma \ref{thm:contmu}. Therefore, by Lemma \ref{thm:alw} one has
\begin{equation}\label{eq:dff2}
\lim_{\epsilon\to 0}\|\nabla_m f_\epsilon\|_{1, m, \mu}=\lim_{\epsilon\to 0}\frac{1}{\epsilon}\int_{U_\epsilon} d\mu_r= \lim_{\epsilon\to 0}\frac{\mu_r(U_\epsilon)}{\epsilon}=2\mu_{r-1}(\Gamma).
\end{equation}
Next, we consider the term $\|\nabla_n \mathcal{L}f_\epsilon \|_{1, n, \nu}$ on the numerator of \eqref{eq:dff1}. Observe that at each point $x\in M\setminus U_\epsilon$,
\begin{align*}
 |\nabla_n \mathcal{L}f(Tx)|^2_n
 &=n(\nabla_n\mathcal{L}f, \nabla_n\mathcal{L}f)_{Tx}\\
 &=n(T_*\nabla_{T^*n}f, T_*\nabla_{T^*n}f)_{Tx}\quad\mbox{by Lemma \ref{thm:app3.3}}\\
 &=T^*n(\nabla_{T^*n}f,\nabla_{T^*n}f)_x\quad\mbox{by \eqref{eq:pbm}}\\
 &=m(\nabla_mf, \nabla_{T^*n}f)_x=0,
\end{align*}
 where we have used \eqref{eq:grad} and the fact that $\nabla_mf(x)=0$ for all $x\in M\setminus U_\epsilon$ to obtain the last line. Hence, the integral $\int_{N\setminus T\Omega_\epsilon}|\nabla_n \mathcal{L}f_\epsilon|_n\,d\nu_r$ vanishes. Set $f$ to be the signed distance function defined by \eqref{eq:signdist}, then $f(x)=\epsilon \cdot f_\epsilon(x)$ for all $x\in U_\epsilon$. Thus $\mathcal{L}f=\epsilon\cdot \mathcal{L}f_\epsilon$ on $TU_\epsilon$. 
Let $\Gamma_t$ be the level surfaces of the signed distance function $f$; that is $\Gamma_t=\{x\in M:f(x)=t\}$. Then $T\Gamma_t$ are generated by the level surfaces of $\mathcal{L}f$; that is $T\Gamma_t=\{y\in N: \mathcal{L}f(y)=t\}$. Therefore, by the co-area formula \eqref{eq:ca} one has,
\begin{align}\label{eq:dff5}
 \|\nabla_n \mathcal{L}f_\epsilon\|_{1, n, \nu}
 &=\int_{TU_\epsilon}  \left|\nabla_n\mathcal{L}f_\epsilon \right|_{n} \, d\nu_r\notag\\
 &=\frac{1}{\epsilon}\int_{TU_\epsilon}  \left|\nabla_{n} \mathcal{L}f\right|_{n} \cdot h_\nu\omega_n^r\notag\\
 &=\frac{1}{\epsilon}\int_{-\epsilon}^\epsilon \left(\int_{(\mathcal{L}f)^{-1}\{t\}}h_\nu\cdot\omega_{n}^{r-1}\right)dt\notag\\
 &=\frac{2}{\epsilon}\int_0^\epsilon \left(\int_{T\Gamma_t}h_\nu\cdot\omega_n^{r-1}\right)dt\notag\\
 &=\frac{2}{\epsilon}\int_0^\epsilon \nu_{r-1}(T\Gamma_t)dt.
 \end{align}
By a straightforward modification of Lemma \ref{thm:contmu}, the expression $\nu_{r-1}(T\Gamma_t)$ appearing on the RHS of \eqref{eq:dff5} is continuous as a function of $t$ on the interval $[0, \epsilon]$. Thus by taking the limit of $\epsilon\to 0$ on both sides of \eqref{eq:dff5} and using the fundamental theorem of calculus (see a similar argument used in Lemma \ref{thm:alw}), one arrives at
\begin{equation}\label{eq:dff3}
\lim_{\epsilon\to 0}  \|\nabla_n \mathcal{L}f_\epsilon\|_{1, n, \nu}=2\nu_{r-1}(T\Gamma).
\end{equation}
Finally, for the denominator term of \eqref{eq:dff1}, we may without loss of generality assume that $\mu_{r}(M_1)\leq\mu_{r}(M_2)$. Then
\begin{align}\label{eq:dff4}
\|f_\epsilon-\alpha\|_{1, m, \mu}
\geq&\int_{M\setminus{U_\epsilon}}|f_\epsilon-\alpha|\, d\mu_r\notag\\
=&|1-\alpha|\cdot(\mu_r(M_1)-\mu_r(U_\epsilon))+|1+\alpha|\cdot(\mu_r(M_2)-\mu_r(U_\epsilon))\notag\\
 \geq&2(\mu_r(M_1)-\mu_r(U_\epsilon)),
\end{align}
for each $\epsilon>0$. Hence, by taking the limit of $\epsilon\to 0$ on \eqref{eq:dff4} one has $\inf_\alpha\|f_\epsilon-\alpha\|_{1, m, \mu}\geq 2\mu_r(M_1)= 2\min\{(\mu_r(M_1),\mu_r(M_2)\}$. Combining this inequality with \eqref{eq:dff1}, \eqref{eq:dff2} and \eqref{eq:dff3}, we conclude that $\mathbf{s}^D\leq \mathbf{h}^D$.

\section{The proof of Theorem \ref{thm:spec}}\label{sec:appD}

In this proof we follow the work of \cite{froyland14} and \cite{mcowen96}, and consider a weak formulation of the eigenvalue problem for the weighted dynamic Laplacian $\triangle^D$. One can find a set of weak solution pairs $(\phi_i, \lambda_i)\in L^2(M, m, \mu_r)\times \mathbb{R}$ to the weak formulation of $\triangle^D$ that satisfies the conclusions of Theorem \ref{thm:spec}. Moreover, we show that the operator $\triangle^D$ has the smooth uniformly elliptic property. That is, $\triangle^D$ can be expressed in local coordinates as $\triangle^D=\sum_{i,j=1}^r a_{ij}\partial_i\partial_j+b_{i}\partial_i+c$, where $a_{ij}$, $b_i$ and $c$ are bounded and smooth functions on $M$, and there exists a constant $\gamma>0$ such that $\sum_{i,j=1}^r a_{ij}\varepsilon_i\varepsilon_j\geq \gamma|\varepsilon|^2$ for all $x\in M$ and $\varepsilon \in\mathbb{R}^r$. The elliptic regularity theorem (see Theorem 8.14 in \cite{gilbarg77}) gives the additional regularity of the eigenfunctions $\phi_i$ on $M$ . Thus, the weak solution pairs $(\phi_i, \lambda_i)$ solve the eigenproblem
\begin{equation}\label{eq:eignp}
\triangle^D\phi_i=\lambda_i \phi_i,
\end{equation}
 for each $i$.

\subsection{Weak formulation of the $\triangle^D$ eigenproblem} \label{sec:wf}
Let $f, g\in C^\infty(M, \mathbb{R})$, and note that the smoothness assumption on the density $h_\mu$ implies $f, g\in L^2(M, m, \mu_r)$. Consider the integral $\int_M  g\cdot \triangle^Df \cdot h_\mu\omega_m^r$. Recall from Section \ref{sec:dfop} that the weighted divergence $\divg_\mu$ satisfies
\begin{equation}\label{eq:wdivth2}
\int_{\partial U} m(\mathcal{V}, \mathbf{n})\cdot h_\mu\omega_m^{r-1}=\int_U \divg_\mu\mathcal{V}\cdot h_\mu\omega_m^r,
\end{equation}
for all open $U\subset M$ and continuously differentiable vector fields $\mathcal{V}\in \mathcal{F}^1(M)$, and where $\mathbf{n}$ is the unit normal bundle along $\partial U$. Since $f, g\in C^\infty(M ,\mathbb{R})$, the vector $g\cdot \nabla_m f\in \mathcal{F}^\infty(M)$. Consequently, by taking $U=M$ and $\mathcal{V}=g\cdot\nabla_m f$ in \eqref{eq:wdivth2}, follow by applying the expansion rule \eqref{eq:pdiv2} for the weighted divergence $\divg_\mu$, one has the following weighted Green's identity:
\begin{align}\label{eq:app3.5}
 \int_{\partial M} g\cdot m( \nabla_m f, \mathbf{n}) \cdot h_\mu \omega_m^{r-1}
 &=\int_M \divg_\mu (g\cdot \nabla_m f)\cdot h_\mu\omega_m^r\notag\\
 &=\int_M g\cdot \triangle_\mu f \cdot h_\mu\omega_m^r+\int_M m( \nabla_m g, \nabla_m f)\cdot h_\mu\omega_m^r.
\end{align}
Rearranging \eqref{eq:app3.5} gives
\begin{equation}\label{eq:wf1}
 \int_M g\cdot\triangle_\mu f \cdot h_\mu\omega_m^r
 =-\int_M m(\nabla_m g,\nabla_m f)\cdot h_\mu \omega_m^r+\int_{\partial M} g\cdot m( \nabla_m f, \mathbf{n}) \cdot h_\mu\omega_m^{r-1}.
\end{equation}
Since $\mathcal{L}$ is the adjoint of $\mathcal{L}^*$
\begin{equation*}
\int_M g\cdot \mathcal{L}^*\triangle_\nu \mathcal{L} f\, d\mu_r
 =\int_N \mathcal{L}g \cdot \triangle_{\nu} \mathcal{L}f \, d\nu_r.
\end{equation*}
Therefore, one has analogous to \eqref{eq:wf1}
\begin{equation}\label{eq:wf2}
 \int_M g\cdot \mathcal{L}^*\triangle_\nu \mathcal{L} f\cdot h_\mu\omega_m^r
=-\int_N n( \nabla_n \mathcal{L}g, \nabla_n \mathcal{L}f) \cdot h_\nu\omega_n^r+\int_{\partial N} \mathcal{L}g\cdot n(\nabla_n \mathcal{L}f,\hat{\mathbf{n}})\cdot h_\nu\omega_n^{r-1},
\end{equation}
where $\hat{\mathbf{n}}$ is the unit normal bundle along $\partial N$. Combining \eqref{eq:wf1} and \eqref{eq:wf2}, we arrive at
\begin{align}\label{eq:wf3}
&2\int_M g\cdot\triangle^D f\cdot h_\mu\omega_m^r\notag\\
=&\int_M g\cdot\left(\triangle_\mu f+\mathcal{L}^*\triangle_\nu \mathcal{L} f \right)\cdot h_\mu\omega_m^r\notag\\
=&-\int_M m(\nabla_m f,\nabla_m g) \cdot h_\mu \omega_m^r-\int_N n( \nabla_n \mathcal{L}g, \nabla_n \mathcal{L}f) \cdot h_\nu\omega_n^r+P_1(f, g, \partial M)+P_2(f, g, \partial N),
\end{align}
where
\begin{equation}\label{eq:wf4}
 P_1(f, g, \partial M)=\int_{\partial M} g\cdot m( \nabla_m f, \mathbf{n}) \cdot h_\mu\omega_m^{r-1}
 \quad\mbox{and}\quad
 P_2(f, g, \partial N)=\int_{\partial N} \mathcal{L}g\cdot n(\nabla_n \mathcal{L}f,\hat{\mathbf{n}})\cdot h_\nu\omega_n^{r-1}.
\end{equation}
Next, we demonstrate that if the boundary condition \eqref{eq:bc} in the hypothesis of Theorem \ref{thm:spec} is satisfied for $f$, then the boundary term $P_1(f, g, \partial M)+P_2(f, g, \partial N)$ of \eqref{eq:wf3} vanishes for all $g\in C^\infty(M, \mathbb{R})$.

\begin{proposition}\label{thm:wbc}
Let $f, g\in C^\infty(M, \mathbb{R})$, and define $P_1(f, g, \partial M)$ and $P_2(f, g, \partial N)$ by \eqref{eq:wf4}, where $\partial M$ and $\partial N$ are the boundary of $M$ and $N$ respectively. If the boundary condition
\begin{equation*}
m([\nabla_m +\nabla_{T^*n}]f, \mathbf{n})(x)=0,
\end{equation*}
holds for all $x\in \partial M$, then
\begin{equation}\label{eq:wbc}
 P_1(\partial M)+P_2(\partial N)=0.
\end{equation}
\end{proposition}
\begin{proof}
Let the hypersurface $\partial M$ be generated by the zero level set of $\psi\in C^\infty(M, \mathbb{R})$; i.e\ $\partial M=\{x\in M: \psi(x)=0\}$. Due to Proposition \ref{thm:Lf}, the surface $\partial N$ is generated by the zero level set of $\mathcal{L}\psi$. Now by Lemma \ref{thm:app3.3} and the fact that $\mathcal{L}^*\mathcal{L}$ is the identity,
\begin{align}\label{eq:wf6}
n(\nabla_n \mathcal{L}f, \nabla_n \mathcal{L}\psi)_{Tx}
&=n(T_*\nabla_{T^*n} f, T_*\nabla_{T^*n} \psi)_{Tx}\notag\\
&=T^*n(\nabla_{T^*n} f, \nabla_{T^*n} \psi)_{x}\quad\mbox{by \eqref{eq:pbm}}\notag\\
&=\restr{(\nabla_{T^*n}f)\psi}{x}\quad\mbox{by \eqref{eq:grad}}\notag\\
&=m(\nabla_{T^*n}f, \nabla_m\psi)_x.
\end{align}
Hence,
\begin{align}\label{eq:wf7}
 &\int_{-\infty}^\infty\int_{\mathcal{L}\psi=t} \mathcal{L}g\cdot \frac{n( \nabla_n \mathcal{L}f, \nabla_n\mathcal{L}\psi)}{|\nabla_n \mathcal{L}\psi|_n}\cdot h_\nu\omega_n^{r-1}\,dt\notag\\
 &=\int_{N} \mathcal{L}g \cdot n( \nabla_n \mathcal{L}f, \nabla_n\mathcal{L}\psi)\cdot h_\nu\omega_n^{r}\notag\quad\mbox{by the co-area formula \eqref{eq:ca}}\\
 &=\int_{M} g\cdot n( \nabla_n \mathcal{L}f, \nabla_n \mathcal{L}\psi )\circ T\cdot h_\mu \omega_m^{r}\notag\quad\mbox{by \eqref{eq:cov3}}\\
 &=\int_{M} g\cdot m(\nabla_{T^*n}f,\nabla_m \psi )\cdot h_\mu\omega_m^{r}\quad\mbox{by \eqref{eq:wf6}}\notag\\
 &=\int_{-\infty}^\infty\int_{\psi=t} g\cdot \frac{m(\nabla_{T^*n}f,\nabla_m \psi )}{|\nabla_m \psi|_m}\cdot h_\mu\omega_m^{r-1}\,dt,
 \end{align}
where the last line is due to the application of the co-area formula \eqref{eq:ca}. Differentiating both sides of \eqref{eq:wf7} with respect to $t$, then at $t=0$
\begin{equation}\label{eq:wf8}
 \int_{\mathcal{L}\psi=t} \mathcal{L}g\cdot \frac{n( \nabla_n \mathcal{L}f, \nabla_n\mathcal{L}\psi)}{|\nabla_n \mathcal{L}\psi|_n}\cdot h_\nu\omega_n^{r-1}
 =\int_{\psi=t} g\cdot \frac{m(\nabla_{T^*n}f,\nabla_m \psi )}{|\nabla_m \psi|_m}\cdot h_\mu\omega_m^{r-1}.
\end{equation}
Additionally, the vector $\nabla_m \psi$ is normal to the level surfaces of $\psi$, thus $\mathbf{n}=\nabla_m \psi/|\nabla_m \psi|_m$,  and  similarly $\hat{\mathbf{n}}=\nabla_n\mathcal{L}\psi/|\nabla_n \mathcal{L}\psi|_n$. Hence,
\begin{align*}
 P_2(f,g,\partial N)\
 &=\int_{\partial N} \mathcal{L}g\cdot n( \nabla_n \mathcal{L}f, \hat{\mathbf{n}})\cdot h_\nu\omega_n^{r-1}\\
 &=\int_{\mathcal{L}\psi=0} \mathcal{L}g\cdot \frac{n( \nabla_n \mathcal{L}f, \nabla_n\mathcal{L}\psi)}{|\nabla_n \mathcal{L}\psi|_n}\cdot h_\nu\omega_n^{r-1}\\
 &=\int_{\psi=0} g\cdot \frac{m(\nabla_{T^*n}f,\nabla_m \psi )}{|\nabla_m \psi|_m}\cdot h_\mu\omega_m^{r-1}\quad\mbox{by \eqref{eq:wf8}}\\
 &=\int_{\partial M} g\cdot m(\nabla_{T^*n}f, \mathbf{n})\cdot h_\mu\omega_m^{r-1}.
\end{align*}
We conclude that
\begin{equation*}
P_1(f,g,\partial M)+P_2(f,g,\partial N)=\int_{\partial M} g\cdot m([\nabla_m +\nabla_{T^*n}]f, \mathbf{n})\cdot h_\mu\omega_m^{r-1},
\end{equation*}
which vanishes due to the theorem hypothesis of $m([\nabla_m+\nabla_{T^*n}]f, \mathbf{n})(x)=0$ for all $x\in \partial M$.
\end{proof}

Due to Proposition \ref{thm:wbc} and \eqref{eq:wf3}, one has
\begin{equation}\label{eq:wf3b}
 \int_M g\cdot\triangle^D f\,d\mu_r=-\int_M m(\nabla_m f, \nabla_m g)\, d\mu_r-\int_N n(\nabla_n\mathcal{L}g, \nabla_n \mathcal{L}f)\,d\mu_r,
\end{equation}
for all $f, g\in C^\infty(M, \mathbb{R})$. Note that \eqref{eq:wf3b} is symmetric in $f$ and $g$, hence the operator $\triangle^D$ is self-adjoint in $L^2(M, m, \mu_r)$.

Suppose the solution $(\phi,\lambda)\in C^\infty (M,\mathbb{R})\times \mathbb{R}$ exists for the eigenvalue problem \eqref{eq:eignp}. Then under the boundary condition \eqref{eq:bc}, one has by \eqref{eq:wf3} and Proposition \ref{thm:wbc} the following formulation for the eigenvalue problem $\triangle^D\phi=\lambda \phi$:
\begin{equation}\label{eq:ws1}
\int_M m(\nabla_m g, \nabla_m \phi) \,d\mu_r+\int_{N} n( \nabla_n \mathcal{L}g,\nabla_n \mathcal{L}\phi) \,d\nu_r=-2\lambda\int_M g\phi \,d\mu_r.
\end{equation}
for all $g\in C^\infty(M, \mathbb{R})$.
Equivalently,
\begin{equation}\label{eq:ws2a}
\int_M \left[m(\nabla_m g, \nabla_m \phi)+T^*n( \nabla_{T^*n} g,\nabla_{T^*n} \phi)\right] \,d\mu_r=-2\lambda\int_M g\phi \,d\mu_r.
\end{equation}
for all $g\in C^\infty(M, \mathbb{R})$.


\subsection{Existence of weak solution and variational characterisation of eigenvalues}\label{sec: ws}
Let $S$ be a weighted Sobolev space $W^{1, 2}(M, m, \mu_r)$ with weights $h_\mu$. Recall from Section \ref{sec:ws} that the weak gradient with respect to the metric $m$ is denoted by $\tilde{\nabla}_m$. Due to \eqref{eq:ws1}, the the weak formulation for the eigenproblem \eqref{eq:eignp} is given by
\begin{equation}\label{eq:ws1b}
 \int_M m(\tilde{\nabla}_m g, \tilde{\nabla}_m \phi) \,d\mu_r+\int_{N} n( \tilde{\nabla}_n \mathcal{L}g,\tilde{\nabla}_n \mathcal{L}\phi) \,d\nu_r=-2\lambda\int_M g\phi \,d\mu_r.
\end{equation}
We show existence of solutions $(\phi_i, \lambda_i)\in S\times\mathbb{R}$ for the above weak formulation \eqref{eq:ws1b}, for all $g\in S$. We call such pairs $(\phi_i, \lambda_i)$ weak solutions\footnotemark\ for the eigenvalue problem \eqref{eq:eignp}.

\footnotetext{The weak solution pairs $(\phi_i, \lambda_i)$ does not necessarily solve the eigenvalue problem $\triangle^D\phi_i=\lambda_i\phi_i$, because $\phi_i$ may lack sufficient regularity on $M$, see p.210 in \cite{mcowen96} for a discussion.}

Our approach to finding the weak solutions for $\triangle^D$ is based on the construction of functionals $F$ and $G$, and using the method of Lagrange multipliers. 
For $f\in S$, we define $G(f)=1-\int_M f^2 \,d\mu_r$ and $F(f)=(1/2)(F_1(f)+F_2(f))$, where $F_1(f)=\int_M|\tilde{\nabla}_m f|_m^2\,d\mu_r$ and $F_2(f)=\int_{N}|\tilde{\nabla}_n \mathcal{L}f|_n^2\, d\nu_r$.
 First we list some useful properties of the functionals $F_1$, $F_2$ and $G$.

\begin{lemma}\label{thm:Fnl}
Let $f\in S$, and denote the linear dual of $S$ as $S^\star$. Define the functional $F_2$ as above.
\begin{enumerate}[(i)]
 \item {The functional $F_2:S\to \mathbb{R}$ is well-defined,}
 \item{The derivative $F'_2(f)$ is linear and bounded (hence $F_2'(f)\in S^\star$),}
 \item{$F_2$ is Fr\'{e}chet-differentiable,}
 \item{$f\to F'_2(f)$ is continuous as a map from $S$ to $S^\star$.}
\end{enumerate}
\end{lemma}
\begin{proof}
\begin{enumerate}[(i)]\item{
Let $(U_k, \varphi_k)_{k\in K}$ be an atlas of $M$. Then due to the fact that $T$ is a $C^\infty$-diffeomorphism, there exists a set of finite constants $C_{ij}^k$ such that $(T^*n)^{ij}=C_{ij}^k m^{ij}$ on $U_k$ for each $1\leq i, j\leq r$ and $k\in K$. Hence, by writing $\tilde{\nabla}_{T^*n}$ in coordinates form via \eqref{eq:gradl} (with respect to weak partial derivatives $\tilde\partial$),  one has on all points in $U_k$
\begin{equation}\label{eq:Fnl2}
\tilde{\nabla}_{T^*n}f=\sum_{ij}(T^*n)^{ij}(\tilde\partial_i f)\partial_j =\sum_{ij}C^k_{ij} m^{ij} (\tilde\partial_i f)\partial_j\leq C^k\cdot \tilde{\nabla}_m f,
\end{equation}
for all $k\in K$, where $C_k=\max_{ij} C^k_{ij}$.
Furthermore, since $M$ is compact, there exists a partition of unity $\sigma_k$ subordinate to the covering $\cup_{k\in K}U_k$ (Lemma \ref{thm:partu}). Therefore,
\begin{align}\label{eq:Fnl5}
 F_2(f)&=\int_N |\tilde{\nabla}_n \mathcal{L}f|_n^2\, d\nu_r\notag\\
 &=\int_N n(\tilde{\nabla}_n \mathcal{L}f, \tilde{\nabla}_n \mathcal{L}f) \,d\nu_r\notag\\
 &=\int_N n(T_*\tilde{\nabla}_{T^*n} f, T_*\tilde{\nabla}_{T^*n} f)\,d\nu_r\quad\mbox{by Lemma \ref{thm:app3.3}}\notag\\
 &=\int_N T^*n(\tilde{\nabla}_{T^*n} f, \tilde{\nabla}_{T^*n} f)\circ T^{-1}\,d\nu_r\quad\mbox{by \eqref{eq:pbm}}\notag\\
 &=\int_M T^*n(\tilde{\nabla}_{T^*n} f, \tilde{\nabla}_{T^*n} f)\, d\mu_r\quad\mbox{by \eqref{eq:cov4}}\notag\\
 &=\int_M m(\tilde{\nabla}_{T^*n} f, \tilde{\nabla}_m f)\,d\mu_r\quad\mbox{by \eqref{eq:grad}}\notag\\
 &=\sum_{k\in K}\int_{U_k} \sigma_k\cdot m(\tilde{\nabla}_{T^*n} f, \tilde{\nabla}_m f)\,d\mu_r\notag\\
 &\leq \sum_{k\in K} C_k \int_{U_k} \sigma_k\cdot m( \tilde{\nabla}_m f, \tilde{\nabla}_m f) \,d\mu_r\quad\mbox{by \eqref{eq:Fnl2}}\notag\\
 &\leq C\cdot\int_M |\tilde{\nabla}_m f|_m^2\,d\mu_r=C\cdot F_1(f),
\end{align}
where $C=\max_{k\in K} C_k$. Since $f\in S$, $\tilde{\nabla}_m f\in L^2(M, m, \mu_r)$. It follows that $F_2:S\to\mathbb{R}$ is well defined.
}

\item{For all $f, g\in S$
\begin{align}\label{eq:Fnl4}
 F'_2(f)g
 &=\lim_{\epsilon \to 0}\frac{F_2(f+\epsilon g)-F_2(f)}{\epsilon}\notag\\
 &=\lim_{\epsilon \to 0}\frac{\int_N |\tilde\nabla_n \mathcal{L}(f+\epsilon g)|^2_n\,d\nu_r-\int_N|\tilde\nabla_n \mathcal{L}f|^2_n\,d\nu_r}{\epsilon}\notag\\
 &=\lim_{\epsilon \to 0}\frac{\int_N \left(|\tilde\nabla_n \mathcal{L}f|^2_n+2\epsilon\cdot n(\tilde\nabla_n \mathcal{L}f, \tilde\nabla_n\mathcal{L}g)+\epsilon^2\cdot |\tilde\nabla_n \mathcal{L}g|^2_n-\int_N|\tilde\nabla_n \mathcal{L}f|^2_n\right)\,d\nu_r}{\epsilon}\notag\\
 &=2\int_N n(\tilde\nabla_n \mathcal{L}g,\tilde\nabla_n \mathcal{L}f )\,d\nu_r,
\end{align}
where to obtain the last line, we have used the fact that the coefficient of the $\epsilon^2$ term on the penultimate line is finite from part $(i)$. Clearly $F'_2(f)$ is linear. Furthermore, by the Cauchy-Schwarz inequality, one has
\begin{align*}
\mbox{RHS of \eqref{eq:Fnl4}}
 &\leq 2\|\nabla_n \mathcal{L}f\|_{2, n, \nu} \cdot \|\tilde\nabla_n \mathcal{L}g\|_{2, n, \nu}\\
 &\leq 2C\cdot\left(\int_M |\tilde\nabla_m f|_m^2\,d\mu_r\right)^{\frac{1}{2}}\cdot\left(\int_M |\tilde\nabla_m g|_m^2\,d\mu_r\right)^{\frac{1}{2}}\quad\mbox{by \eqref{eq:Fnl5}}\\
 &\leq 2C\cdot\left(\int_M |\tilde\nabla_m f|_m^2\,d\mu_r\right)^{\frac{1}{2}}\cdot\|g\|_{S},
\end{align*}
where $C$ is the same constant that appeared in part $(i)$. Therefore, $F'_2(f)$ is bounded.
}
\end{enumerate}
\paragraph{}
By using the results of part $(i)$ and $(ii)$, the proof of $(iii)$ and $(iv)$ is similar to the corresponding results of Lemma C.1 in \cite{froyland14}.
\end{proof}

\begin{remark}\label{thm:rm1}
One may obtain analogous results of Lemma \ref{thm:Fnl} for $F_1$ by setting $T$ as the identity map in $F_2$, while the corresponding results for $G$ is a straightforward modification with
\begin{equation}\label{eq:rm1.1}
 G'(f)g=-2\int_M fg\,d\mu_r.
\end{equation}
\end{remark}

An important concept associated with linear functionals is the weak convergence. Let $f_i$ be a sequence in $S$. We say that $f_i\rightharpoonup f$ weakly in $S$, if $H(f_i)\to H(f)$ for all $H\in S^\star$ (where $S^\star$ is the linear dual of $S$). Moreover, since $S$ is a Hilbert space (Proposition \ref{thm:comW}), by the Riez representation theorem, if $f_i\rightharpoonup f$ weakly in $S$ then
$
 \langle g, f_i\rangle_S=\langle g, f\rangle_S,
$
for all $g\in S$. One has the following standard result (see p.174, \cite{mcowen96})
 \begin{lemma}\label{thm:wsc}
  Every bounded sequence in a Hilbert space contains a weakly convergent subsequence.
 \end{lemma}


%
Recall from Section \ref{sec:ws} that the $A_p$ condition on the density $h_\mu$ has important consequences for the weighted Sobolev space $W^{1, 2}(M, m, \mu_r)$.  By assumption, the density $h_\mu$ is smooth and uniformly bounded away from zero. Hence, by Proposition \ref{thm:smoothap}, the density $h_\mu$ is an $A_2$ weight on the the space $S$.

\begin{lemma}\label{thm:Fmin}
$F$ attains its minimum on the constraint set $\mathcal{C}=\{f\in S:G(f)=0\}$.
\end{lemma}
\begin{proof}
Define the inner-product  $\langle h, g\rangle_{S'}=\int_N (n(\tilde\nabla_n \mathcal{L}g, \tilde\nabla_n \mathcal{L} h)+\mathcal{L}(gh))\,d\nu_r$ for all $g, h\in S$, and denote the norm associated with $\langle\cdot ,\cdot\rangle_{S'}$  by $\|.\|_{S'}$. Set $I=\inf\{F(g):g\in \mathcal{C}\}\geq0$, and select a sequence $f_i\in \mathcal{C}$ such that $F(f_i)\to I$ and $F(f_i)\leq I+1$.

First, we show that the sequence $f_i$ is bounded in both $S$ and $S'$. Due to Lemma \ref{thm:gpc}, there exists a constant $K$ (independent of $f_i$) such that $\|f_i-\alpha(f_i)\|_{2, \mu}\leq K\|\tilde\nabla_m f_i\|_{2, \mu}$ for each $i$. Hence,
\begin{align*}
 \|f_i\|^2_S
 &=\|f_i\|^2_{2, m, \mu}+\|\tilde\nabla_m f_i\|^2_{2, m, \mu}\\
 &\leq\|f_i-\alpha(f_i)\|_{2, m, \mu}^2+|\alpha(f_i)|+\|\tilde\nabla_m f_i\|^2_{2, m, \mu}\quad\mbox{by triangle's inequality}\\
 &\leq (1+K^2)\|\tilde\nabla_m f_i\|^2_{2, m, \mu}+|\alpha(f_i)|.
\end{align*}
Moreover, by Cauchy-Schwarz
\begin{align*}
 |\alpha(f_i)|^2=\left|\int_M f_i\,d\mu_r\right|^2\leq\left(\int_M f_i^2\,d\mu_r\right)=1-G(f_i)=1.
\end{align*}
Hence
\begin{equation*}
 \|f_i\|_S\leq (1+K^2)\|\nabla_m f_i\|_{2, m, \mu}+1=(1+K^2)F_1(f_i)+1\leq (1+K^2)(I+1)+1,
\end{equation*}
so that $f_i$ is a bounded sequence in $S$. By applying similar arguments as \eqref{eq:Fnl5} in the proof of Lemma \ref{thm:Fnl}(i), one can verify that $f_i$ is also a bounded sequence in $S'$.

Since $f_i$ is a bounded sequence in $S$, and $S$ a Hilbert space (due to Proposition \ref{thm:comW}), by Lemma \ref{thm:wsc}, there exists a subsequence $f_{i_l}$ such that $f_{i_l}\rightharpoonup f$ weakly in $S$. Moreover, due to Lemma  \ref{thm:rcom}, the embedding $S\hookrightarrow  L^2(M, m, \mu_r)$ is compact, which implies the existence of a subsequence $f_{i_k}$ of $f_{i_l}$, such that $f_{i_k}\rightarrow f$ in $L^2(M, m, \mu_r)$. The strong convergence $f_{i_k}\to f$ in $L^2(M, m, \mu_r)$ implies $\mathcal{L}f_{i_k}\to \mathcal{L}f$ in $L^2(N, n, \nu_r)$, because by the change of variable \eqref{eq:cov3}
\begin{equation*}
 \|f_{i_k}-f\|^2_{2, m, \mu}=\int_M |f_{i_k}-f|^2\, d\mu_r=\int_N |\mathcal{L}f_{i_k}-\mathcal{L}f|^2\,d\nu_r=\|\mathcal{L}f_{i_k}-\mathcal{L}f\|^2_{2, n, \nu}.
\end{equation*}

Next, we use the fact that the subsequence $f_{i_k}$ is bounded in $S'$ together with the weak convergence of $f_{i_k}$ in $S$, to show that $f_{i_k}$ convergences weakly in $S'$. Due to lemma \eqref{thm:Fnl} and Remark \ref{thm:rm1}, one has $F'_2(g)\in S^\star$ and $G'(g)\in S^\star$ for all $g\in S$. Therefore
\begin{align*}
\lim_{i_k\to \infty}\langle f_{i_k}, g\rangle_{S'}
&=\lim_{i_k\to \infty} \int_N n(\tilde\nabla_n \mathcal{L}f_{i_k},\tilde \nabla_n \mathcal{L}g)\,d\nu_r+\lim_{i_k\to \infty}\int_N  \mathcal{L}(f_{i_k} g) \,d\nu_r\\
&=\frac{1}{2}\lim_{i_k\to \infty} F'_2(g)f_{i_k}+\lim_{i_k\to \infty}\int_M  f_{i_k} g \,d\mu_r\quad\mbox{by \eqref{eq:Fnl4} and \eqref{eq:cov3}}\\
&=\frac{1}{2}\lim_{i_k\to \infty} F'_2(g)f_{i_k}+\frac{1}{2}\lim_{i_k\to \infty}G'(g)f_{i_k}\quad\mbox{by \eqref{eq:rm1.1}}\\
&=\frac{1}{2}\left(F_2'(g)f+G'(g)f\right)\\
&=\langle f, g\rangle_{S'},
\end{align*}
where the penultimate line is due to the weak convergence of $f_{ik} \rightharpoonup f$ in $S$.

Now, the weak convergence of $f_{i_k}$ in $S'$ implies
\begin{align}\label{eq:Fmin1a}
\|f\|_{S'}^2=&\langle f, f\rangle_{S'}=\lim_{i_k\to \infty}\langle f_{i_k}, f\rangle_{S'}\notag\\
= &\lim_{i_k\to \infty}\bigg\{\langle \tilde\nabla_n \mathcal{L}f_{i_k}, \tilde\nabla_n \mathcal{L}f\rangle_\nu+\langle \mathcal{L}f_{i_k}, \mathcal{L}f\rangle_\nu\bigg\}\notag\\
\leq &\lim_{i_k\to \infty}
 \bigg\{\|\tilde\nabla_n \mathcal{L}f_{i_k}\|_{2, n, \nu}\cdot\|\tilde\nabla_n \mathcal{L}f\|_{2, n, \nu}+\|\mathcal{L}f_{i_k}\|_{2, \nu}\cdot\|\mathcal{L}f\|_{2, n, \nu}\bigg\},
\end{align}
 where the inequality on the last line is due to Cauchy-Schwarz. Set $a_1=\|\tilde\nabla_n \mathcal{L}f_{i_k}\|_{2, n, \nu}$, $b_1=\|\tilde\nabla_n \mathcal{L}f\|_{2, n, \nu}$, $a_2=\|\mathcal{L}f_{i_k}\|_{2, n, \nu}$ and $b_2=\|\mathcal{L}f\|_{2, n, \nu}$, and consider the inequality
 \begin{align}\label{eq:abab}
  a_1b_1+a_2b_2
  &=\sqrt{a_1^2b_1^2+a_2^2b_2^2+2a_1b_2a_2b_1}\notag\\
  &\leq \sqrt{a_1^2b_1^2+a_2^2b_2^2+a_1^2b_2^2+a_2^2b_1^2}\quad\mbox{since $2cd\leq c^2+d^2, \forall c, d\in \mathbb{R}$}\notag\\
  &=\sqrt{(a_1^2+a_2^2)(b_1^2+b_2^2)}.
 \end{align}
As a consequence of \eqref{eq:abab}, one has
\begin{align*}
\mbox{RHS of \eqref{eq:Fmin1a}}
\leq &\lim_{i_k\to \infty} \bigg\{\sqrt{\left(\|\tilde\nabla_n \mathcal{L}f_{i_k}\|_{2, n, \nu}^2+\|\mathcal{L}f_{i_k}\|^2_{2, n, \nu}\right)\cdot\left(\|\tilde\nabla_n \mathcal{L}f\|^2_{2, n, \nu}+\|\mathcal{L}f\|^2_{2, n, \nu}\right)}\bigg\}
\notag\\
=&\|f\|_{S'}\times \lim_{i_k\to \infty} \|f_{i_k}\|_{S'}.
\end{align*}
Thus, $\| f \|_{S'}\leq\lim_{i_k\to \infty} \|f_{i_k}\|_{S'}$. Furthermore, the subsequence $f_{i_k}$ is bounded in $S'$, and $\liminf_{i_k\to \infty} \|f_{i_k}\|_{S'}$ is the largest number smaller than $\lim_{i_k\to \infty} \|f_{i_k}\|_{S'}$. Thus,
\begin{equation}\label{eq:Fmin1}
\| f \|_{S'}\leq\lim_{i_k\to \infty} \|f_{i_k}\|_{S'}\implies \|f\|_{S'}\leq \liminf_{i_k\to \infty} \|f_{i_k}\|_{S'}.
\end{equation}
Similarly, the weak convergence of the bounded subsequence $f_{i_k}$ in $S$ gives
\begin{equation}\label{eq:Fmin2}
 \|f\|_{S}\leq\lim_{i_k\to \infty} \|f_{i_k}\|_{S} \implies \|f\|_S\leq \liminf_{i_k\to \infty} \|f_{i_k}\|_{S}.
\end{equation}

Finally, due to \eqref{eq:Fmin1} and \eqref{eq:Fmin2}
\begin{align}\label{eq:fmin3}
2F(f)&=\int_M |\tilde\nabla_m f|_m^2\,d\mu_r+\int_{N} |\tilde\nabla_n \mathcal{L}f|_n^2\,d\nu_r\notag\\
&=\|f\|^2_S-\|f\|^2_{2, m, \mu}+\|f\|^2_{S'}-\|\mathcal{L}f\|^2_{2, n, \nu}\notag\\
&\leq\liminf_{i_k\to \infty} \|f_{i_k}\|^2_S-\|f\|^2_{2, m, \mu}+\liminf_{i_k\to \infty}\|f_{i_k}\|^2_{S'}-\|\mathcal{L}f\|^2_{2, n, \nu}\quad\mbox{by \eqref{eq:Fmin2} and \eqref{eq:Fmin1}}\notag\\
&=\liminf_{i_k\to \infty} \{\|f_{i_k}\|^2_S-\|f\|^2_{2, m, \mu}\}+\liminf_{i_k\to \infty}\{\|f_{i_k}\|^2_{S'}-\|\mathcal{L}f\|^2_{2, n, \nu}\}.
\end{align}
By the strong convergence of $f_{i_k}\to f$ in $L^2(M, m, \mu_r)$ and $\mathcal{L}f_{i_k}\to \mathcal{L}f$ in $L^2(N, n, \nu_r)$, one has
\begin{align}\label{eq:fmin4}
\mbox{RHS of \eqref{eq:fmin3}}
&=\liminf_{i_k\to \infty}\{ \|f_{i_k}\|^2_S-\|f_{i_k}\|^2_{2, m, \mu}\}+\liminf_{i_k\to \infty}\{\|f_{i_k}\|^2_{S'}-\|\mathcal{L}f_{i_k}\|^2_{2, n, \nu}\}\notag\\
&\leq \liminf_{i_k\to \infty} \left\{\|f_{i_k}\|^2_S-\|f_{i_k}\|^2_{2, m, \mu}+\|f_{i_k}\|^2_{S'}-\|\mathcal{L}f_{i_k}\|^2_{2, n, \nu}\right\}\notag\\
&=\liminf_{i_k\to \infty}\{\|\nabla_m f_{i_k}\|^2_{2, m, \mu}+\|\nabla_n \mathcal{L}f_{i_k}\|^2_{2, n, \nu}\}=2\liminf_{i_k\to \infty} F(f_{i_k})=2I.
\end{align}
From \eqref{eq:fmin3} and \eqref{eq:fmin4}, we conclude that $F(f)\leq I= \inf\{F(g):g\in \mathcal{C}\}$; thus the minimum of $F$ is attained by $f$. To complete the proof the theorem, it remains to show that $f\in \mathcal{C}$; that is $G(f)=0$. One has
\begin{align*}
G(f)
=&1-\int_M f^2\, d\mu_r\\
=&1-\|f\|^2_{2, m, \mu}\\
=&1-\lim_{i_k\to \infty}\|f_{i_k}\|^2_{2, m, \mu}\\
=&\lim_{i_k\to \infty} G(f_{i_k})=0,
\end{align*}
since $f_{i_k}\in \mathcal{C}$.
\end{proof}

Due to Lemma \ref{thm:Fnl}, the functionals $F$ and $G$ are continuously differentiable. In addition, by Lemma \ref{thm:Fmin} there exists a function $\bar{f}\in S$ which minimises $F$ over the constraint set $\mathcal{C}$. Therefore, using the method of Lagrange multipliers, one has the equation $F'(\bar{f})g=\lambda G'(\bar{f})g$ for some $\lambda\in\mathbb{R}$ and all $g\in S$. Expanding this equation with \eqref{eq:Fnl4} and \eqref{eq:rm1.1} yields
\begin{equation}\label{eq:ws2}
\int_M m(\tilde\nabla_m g,\tilde\nabla_m \bar{f}) \,d\mu_r+\int_{N} n(\tilde\nabla_n \mathcal{L}g ,\tilde\nabla_n \mathcal{L}\bar{f}) \,d\nu_r=-2\lambda\int_M g\bar{f}\,d\mu_r,
\end{equation}
for all $g\in S$, $\bar{f}\in \{f\in S:G(f)=0\}$ and some $\lambda\in \mathbb{R}$. By comparing \eqref{eq:ws1b} and \eqref{eq:ws2}, one sees immediately that $(\bar{f}, \lambda)\in \{f\in S:G(f)=0\}\times \mathbb{R}$ is a solution pair for the weak formulation \eqref{eq:ws1b}.

If we fix $g$ to be $\bar{f}$ in \eqref{eq:ws2}, then
\begin{align}\label{eq:ws3}
2F(\bar{f})&=\int_M |\tilde\nabla_m \bar{f}|_m^2\,d\mu_r+\int_N |\tilde\nabla_n \mathcal{L}\bar{f}|_n^2\,d\nu_r\notag\\
&=-2\lambda\int_M \bar{f}^2\,d\mu_r\notag\\
&=-2\lambda (G(\bar{f})+1).
\end{align}
Moreover, as a consequence of Lemma \ref{thm:Fmin}, $\bar{f}$ is minimising for $F$. Thus rearranging \eqref{eq:ws3} yields
\begin{align}\label{eq:ws4}
  \lambda&=-\inf_{f\in S}\frac{F(f)}{G(f)+1}\notag\\
      &=-\inf_{f\in S}\frac{\int_M |\tilde\nabla_m f|_m^2 \,d\mu_r+\int_{N} |\tilde\nabla_n \mathcal{L}f|_n^2 \,d\nu_r}{2\int_M f^2 \,d\mu_r}.
\end{align}
Let the solution $(\bar{f}, \lambda)$ to \eqref{eq:ws1b} be denoted by $(\phi_2, \lambda_2)$. To find other solution pairs to \eqref{eq:ws1b} of the form $(\phi_i, \lambda_i)$, one follows the standard induction arguments presented in \cite{froyland14} and p.212 in \cite{mcowen96}: One constructs a sequence of decreasing, closed and $L^2(M, m, \mu_r)$-orthogonal subspaces of $S$; that is for $k\geq 1$, a sequence of subspaces of $S$ of the form $S_k=\{f\in S:\int_M f\phi_{i}\,d\mu_r=0, \mbox{ for }i=1, 2,\ldots, k\}$, where $\phi_1$ is constant. One then uses the fact that the solutions $\phi_i$ and $\phi_j$ are $L^2(M, m, \mu_r)$-orthogonal for $\lambda_i\neq \lambda_j$ (this follows immediately from Lemma C.3. in \cite{froyland14}), and the fact that each $S_k$ is complete (closed subspace of a Hilbert space), to apply the variational method on $S_{k-1}$ to obtain
\begin{equation}\label{eq:ws5}
     \lambda_k
      =-\inf_{f\in S_{k-1}}\frac{\int_M |\tilde\nabla_m f|_m^2 \,d\mu_r+\int_{N} |\tilde\nabla_n \mathcal{L} f|_n^2\,d\nu_r}{2\int_M f^2 \,d\mu_r},
\end{equation}
for $k=2, 3,\ldots$. Note that $(\phi_1,0)$ is a solution pair to \eqref{eq:ws1b}, thus $\lambda_1=0$. Additionally, the sequence $\lambda_i$ is monotone decreasing and tends to $-\infty$, with the solution space finite for each $i$ (Lemma C.4. in \cite{froyland14}).

Furthermore, using the identity $\tilde\nabla_n=T_*\tilde\nabla_{T^*n}\mathcal{L}^*$ from Lemma \ref{thm:app3.3},
\begin{align*}
 \int_{N} |\tilde\nabla_n \mathcal{L}f|^2_n \,d\nu_r
 &=\int_{N} n( T_*\tilde\nabla_{T^*n} f ,T_*\tilde\nabla_{T^*n} f) \, d\nu_r\\
 &=\int_M T^*n( \tilde\nabla_{T^*n} f, \tilde\nabla_{T^*n} f) \, d\mu_r=\int_M |\tilde\nabla_{T^*n}f|_{T^*n}^2 \, d\mu_r,
\end{align*}
where the second equality is due to \eqref{eq:pbm} and \eqref{eq:cov4}. Hence, one can write \eqref{eq:ws5} as an integral of $M$ as
\begin{equation}\label{eq:ws6}
 \lambda_k=-\inf_{f\in S_{k-1}}\frac{\int_M \left(|\tilde\nabla_m f|_m^2 + |\tilde\nabla_{T^*n} f|_{T^*n}^2\right) \, d\mu_r}{2\int_M f^2\,d\mu_r}.
\end{equation}

\subsection{Ellipticity and global regularity of weak solutions}\label{sec:sue}

To complete the proof of Theorem \ref{thm:spec}, it remains to verify that the eigenfunctions $\phi_i$ of $\triangle^D$ are smooth and unique for each $i$. For then, the smoothness of $\phi_i$ on $M$ implies that the weak solution pairs $(\phi_i,\lambda_i)$ which solves \eqref{eq:ws1b} are also solution to \eqref{eq:ws1}. Moreover, the uniqueness of $(\phi_i,\lambda_i)$ implies that the solutions of the eigenvalue problem \eqref{eq:eignp} are given by \eqref{eq:ws5} or \eqref{eq:ws6} (with the weak gradients replaced with standard version due to the additional smoothness of $\phi_i$). To determine the regularity and uniqueness of $\phi_i$ on $M$, we utilise the elliptical regularity theorem (see Theorem 8.14 in \cite{gilbarg77}).

We say that an operator $L$ of the form
\begin{equation}\label{eq:sue1}
 L=\sum_{i,j=1}^r a_{ij}\frac{\partial^2}{\partial x^i\partial x^j}+b_{i}\frac{\partial}{\partial x^i}+c,
\end{equation}
is strictly uniformly elliptic if $a_{ij}$, $b_i$ and $c$ are bounded, real-valued functions on $M$, and there exists a constant $\gamma>0$ such that
\begin{equation}\label{eq:sue2}
 \sum_{i,j=1}^r a_{ij}\varepsilon_i\varepsilon_j\geq \gamma|\varepsilon|^2,
\end{equation}
where $\varepsilon \in\mathbb{R}^r$ is non-zero.

As a consequence of the Elliptical Regularity theorem, if $\partial M$ is smooth, and $\triangle^D$ is a strictly uniformly elliptic operator with $a_{ij}, b_i, c\in C^\infty(M, \mathbb{R})$ and $c\leq 0$ in $M$, then there exist unique solutions in $C^\infty(M, \mathbb{R})$ for the eigenproblem \eqref{eq:eignp}.

\begin{lemma}\label{thm:ul}
Let $T:M\to N$ be a $C^\infty$-diffeomorphism, and assume $h_\mu$ is smooth and uniformly bounded away from zero. The weighted Laplacian $\triangle^D$ is a strictly uniformly elliptic operator of the form \eqref{eq:sue1}, with $a_{ij}, b_i, c$ in $C^\infty(M, \mathbb{R})$ and $c\leq 0$ on $M$.
\end{lemma}
\begin{proof}
 For this proof, we say that an operator has property $E$, if it is a strictly uniformly elliptic, with coefficients $a_{ij}, b_i, c$ in $C^\infty(M, \mathbb{R})$ and $c\leq 0$ on $M$. By Lemma \ref{thm:dwlp2}
\begin{equation}\label{eq:ul}
2\triangle^Df
 =\triangle_m f+\mathcal{L}^*\triangle_n \mathcal{L} f+\frac{m( \nabla_m h_\mu, \nabla_m f)}{h_\mu} + \frac{n( \nabla_n h_\nu, \nabla_n \mathcal{L}f)\circ T}{h_\nu\circ T}.
\end{equation}
 Clearly the sum of operators with property $E$ is an operator with property $E$. Additionally, if the second and fourth terms of \eqref{eq:ul} has property $E$, then by setting $T$ as the identity, one immediately see that the first and third terms of \eqref{eq:ul1} also has property $E$. Thus, it is sufficient to show that the second and fourth terms of \eqref{eq:ul} has property $E$. To show that second term $\mathcal{L}^*\triangle_{n} \mathcal{L}$ of \eqref{eq:ul} has property $E$, we note by Corollary \ref{thm:app3.3} that $\mathcal{L}^*\triangle_{n} \mathcal{L}=\triangle_{T^*n}$. Therefore in local coordinates at any point in $M$,
 \begin{equation}\label{eq:ul1}
  \mathcal{L}^*\triangle_n \mathcal{L} f=\triangle_{T^*n} f=\frac{1}{\sqrt{\det{G_{T^*n}}}}\sum_{i,j=1}^r\partial_j\left( \sqrt{\det{G_{T^*n}}}(T^*n)^{ij}\partial_i f\right),
 \end{equation}
 for all $f\in C^\infty(M, \mathbb{R})$. Using Jacobi's formula for differentiating the determinant of a matrix $A$; that is $\partial_k(\det{A})(x)=(\det A)(x)\sum_{ij}(A^{-1})_{ij}(x)\partial_kA_{ij}(x)$ for all $x\in M$, one has
\begin{align}\label{eq:ul1.1}
 \partial_j(\sqrt{\det G_{T^*n}})
 &=\frac{1}{2}\frac{1}{\sqrt{\det G_{T^*n}}}\partial_j(\det G_{T^*n})\notag\\
 &=\frac{1}{2}\frac{\det G_{T^*n}}{\sqrt{\det G_{T^*n}}}\sum_{k,l=1}^r (G_{T^*n}^{-1})_{kl}\partial_j (G_{T^*n})_{kl}\notag\\
 &=\frac{1}{2}\sqrt{\det G_{T^*n}}\sum_{k,l=1}^r (T^*n)^{kl}\partial_j (T^*n)_{kl}.
\end{align}
Therefore, by using the product rule to expand the partial derivative in the summation on the RHS of \eqref{eq:ul1}, and then applying \eqref{eq:ul1.1} to the first term one has
 \begin{align}\label{eq:ul2} \mbox{RHS of \eqref{eq:ul1}}=&
 \begin{aligned}
  \frac{1}{\sqrt{\det G_{T^*n}}}\bigg(\sum_{i, j=1}^r (T^*n)^{ij}\partial_j(\sqrt{\det G_{T^*n}})\partial_i f+\sqrt{\det G_{T^*n}}\partial_j(T^*n)^{ij}\partial_i f\\
  +\sqrt{\det G_{T^*n}}(T^*n)^{ij}\partial_j\partial_i f\bigg)
  \end{aligned}\notag\\
  =&\sum_{i,j=1}^r \frac{1}{2}\left(\sum_{k,l=1}^r (T^*n)^{kl}\partial_j (T^*n)_{kl}\right)(T^*n)^{ij}\partial_if+\partial_j (T^*n)^{ij}\partial_i f +(T^*n)^{ij}\partial_j\partial_if\notag\\
  =&\sum_{i,j=1}^r\left[\frac{1}{2}\left(\sum_{k,l=1}^r(T^*n)^{kl}\partial_j (T^*n)_{kl}\right)(T^*n)^{ij}+\partial_j(T^*n)^{ij}\right]\partial_i f+(T^*n)^{ij}\partial_j\partial_i f.
 \end{align}
 Now the Riemannian metric $n$ is a $C^\infty$ bilinear symmetric form and positive-definite for every $y\in N$. Moreover, the mapping $T$ is a $C^\infty$-diffeomorphism. Hence, the components $(T^{*}n)^{ij}$ and $\partial_i (T^{*}n)^{ij}$ are both bounded and smooth for each $1\leq i, j\leq r$. Therefore, the coefficients  $b_i=\sum_j \frac{1}{2}(T^*n)^{ij}\partial_j(T^*n)_{kl}+\partial_j(T^*n)^{ij}$ and $a_{ij}=(T^*n)^{ij}$ in \eqref{eq:ul2} are both bounded and smooth. Additionally, due to Lemma \ref{thm:app3.1} we have at the point $x\in M$,
 \begin{align}\label{eq:ul2.2}
    \sum_{i,j=1}^r a_{ij}\varepsilon_i\varepsilon_j
    &=\sum_{i, j=1}^r (T^*n)^{ij}\varepsilon_i\varepsilon_j\notag\\
    &=\sum_{i,j=1}^r (J_T^\top\cdot G_{n}\circ T\cdot J_T)^{ij}\varepsilon_i\varepsilon_j\notag\\
    &=\sum_{i,j=1}^r (J_T^{-1}\cdot G_{n}^{-1}\circ T\cdot (J_T^\top)^{-1})_{ij}\varepsilon_i\varepsilon_j\notag\\
    &=\sum_{i,j, k, l=1}^r (J_T^{-1})_{ik}\cdot (G^{-1}_{n}\circ T)_{kl}\cdot (J_T^{-1})_{jl} \varepsilon_i\varepsilon_j\notag\\
   &=\sum_{i,j,k,l=1}^r (J_{T^{-1}}\circ T)_{ik}\cdot (G^{-1}_{n}\circ T)_{kl} \cdot (J_{T^{-1}}\circ T)_{jl}\varepsilon_i\varepsilon_j\quad\mbox{by the inverse function theorem}\notag\\
   &=\sum_{k,l=1}^r\left(J_{T^{-1}}\circ T\cdot \varepsilon\right)_k\cdot (G^{-1}_{n}\circ T)_{kl}\cdot \left(J_{T^{-1}}\circ T\cdot\varepsilon\right)_l\notag\\
   &>0,
 \end{align}
 where we have used the fact that the matrix $G_n^{-1}$ is positive definitive at every $Tx\in N$ to obtain the last inequality. Hence, there is a $\gamma>0$ such that $\sum_{i,j=1}^r a_{ij}(x)\varepsilon_i\varepsilon_j\geq \gamma|\varepsilon|^2$ for all $x\in M$. Thus $a_{ij}$ satisfies the condition \eqref{eq:sue2}, so by \eqref{eq:ul1}-\eqref{eq:ul2} the term $\mathcal{L}^*\triangle_n \mathcal{L}$ has property $E$.

 To show that the fourth term $n( \nabla_n h_\nu, \nabla_n\mathcal{L}f)/\mathcal{L}^*h_\nu$ of \eqref{eq:ul} has property $E$, we consider the numerator term. One has at each point $Tx \in N$,
 \begin{align}\label{eq:ul3}
  n(\nabla_n h_\nu, \nabla_n \mathcal{L} f)
  &=n(\nabla_n h_\nu, T_*\nabla_{T^*n} f)\quad\mbox{by Lemma \ref{thm:app3.3}}\notag\\
  &=(T_*\nabla_{T^*n} f) h_\nu\notag\\
  &=\nabla_{T^*n} f (h_\nu\circ T)\circ T^{-1}\quad\mbox{by \eqref{eq:T_*}}\notag\\
  &=m(\nabla_m(h_\nu\circ T), \nabla_{T^*n} f)\circ T^{-1}.
 \end{align}
 Writing the RHS of \eqref{eq:ul3} in local coordinates, one has at any point $x\in M$
 \begin{align*}
 \mbox{RHS of \eqref{eq:ul3}}
  &=\sum_{i,j=1}^r m_{ij}\left(\sum_{k=1}^r m^{ki}\partial_k (h_\nu\circ T)\right)\left(\sum_{l=1}^r (T^*n)^{jl}\partial_lf \right)\notag\\
  &=\sum_{j=1}^r \partial_j(h_\nu\circ T) \left(\sum_{l=1}^r (T^*n)^{jl} \partial_lf\right)\quad\mbox{on contracting the index $i$}\notag\\
  &=\sum_{j,l=1}^r  \partial_j(h_\nu\circ T) (T^*n)^{jl} \partial_l f.
 \end{align*}
Therefore, at each $x\in M$
\begin{align*}
 \frac{n( \nabla_n h_\nu, \nabla_n \mathcal{L} f)\circ T}{h_\nu\circ T}
 &=\frac{\sum_{j,l=1}^r  \partial_j(h_\nu\circ T) (T^*n)^{jl} \partial_l f}{h_\nu\circ T}\\
 &=\sum_{j,l=1}^r  \partial_j(\ln{(h_\nu\circ T)}) (T^*n)^{jl} \partial_l f
\end{align*}
As before, due to the properties of the metric $m$, the smoothness of $h_\mu$, and the fact that $T$ is a diffeomorphism, the coefficient $b_l=\sum_{j} \partial_j(\ln{(h_\nu\circ T)}) (T^*n)^{jl}$ is bounded and smooth, and so the fourth term of \eqref{eq:ul} has property $E$.
\end{proof}

\section{The proof of Theorem \ref{thm:wci}}

This proof is a straightforward modification of Theorem 3.2 in \cite{froyland14}. Let $g:M\to \mathbb{R}^+$ be nonnegative and smooth. Since
\begin{equation*}
 \int_N \mathcal{L}g\, d\nu_r=\int_M g \,d\mu_r,
\end{equation*}
by \eqref{eq:cov3}, and densities $h_\mu, h_\nu$ are both positive and smooth, the function $\mathcal{L}g$ is also nonnegative and smooth. Denote by $\Gamma_t$ the level surfaces generated by $g$; that is $\{x\in M:g(x)=t\}$. Then the level surfaces of $T\Gamma_t$ are generated by $\mathcal{L}g$. Now, due to the co-area formula given by Lemma \ref{thm:ca}, one has
\begin{align}\label{eq:wci1}
&\int_M |\nabla_m g|_m \cdot h_\mu\omega_m^r+\int_N |\nabla_n \mathcal{L}g|_n \cdot h_\nu \omega_n^r\notag\\
=&\int_0^\infty \left( \int_{\Gamma_t} h_\mu\omega_m^{r-1}+\int_{T\Gamma_t} h_\nu\omega_n^{r-1}\right)dt\notag\\
=&\int_0^\infty \left(\mu_{r-1}(\Gamma_t)+\nu_{r-1}(T\Gamma_t)\right) dt\notag\\
\geq& 2 \inf_{t\in (0, \infty)}\mathbf{H}^D(\{g=t\})\int_0^\infty \min\{\mu_r(g>t), \mu_r(g<t)\}dt.
\end{align}
Let $f:M\to \mathbb{R}$ be smooth, and $\sigma$ the median of $f$ with respect to $\mu_r$; i.e $\mu_r(f\geq \sigma)\geq 1/2$ and $\mu_r(f\leq\sigma)\geq 1/2 $. Set $f_+=\max\{f-\sigma, 0\}$ and $f_{-}=-\min\{f-\sigma, 0\}$, so that $f-\sigma=f_+-f_-$. Observe that for each point $x\in M$, either $|f(x)-\sigma|=f_+(x)$, $|f(x)-\sigma|=f_-(x)$ or $|f(x)-\sigma|=f_+(x)=f_-(x)=0$. Therefore
 \begin{equation}\label{eq:wci9b}
 \inf_{t\in (-\infty, \infty)}\mathbf{H}^D(\{f=t\})= \min\bigg\{ \inf_{t\in (0, \infty)}\mathbf{H}^D(\{f_-^2=t\}), \inf_{t\in (0, \infty)}\mathbf{H}^D(\{f_+^2=t\})\bigg\}.
 \end{equation}
In addition, if $f_+$ is positive then $f> \sigma$, and if $f_-$ is positive then $f<\sigma$. Hence, by using the fact that $\sigma$ is the median of $f$, one has
\begin{equation}\label{eq:wci7}
\mu_r(f^2_+>t)\leq \frac{1}{2}\quad \quad\mbox{and}\quad\quad \mu_r(f^2_->t)\leq \frac{1}{2},
\end{equation}
for all $t\geq 0$. Moreover, if $f_+\neq 0$ then $f_-=0$, and if $f_-\neq 0$ then $f_+=0$. Hence,
\begin{equation}\label{eq:wci2a}
 (f-\sigma)^2=f_+^2+f_-^2,
\end{equation}
and
\begin{align}\label{eq:wci2b}
|\nabla_m (f_+^2+f_-^2)|_m^2
&=m(\nabla_m (f_+^2+f_-^2), \nabla_m (f_+^2+f_-^2))\notag\\
&=|\nabla_m (f_+^2)|_m^2+2m(\nabla_m(f_+^2), \nabla_m (f_-^2))+|\nabla_m (f_-^2)|_m^2\notag\\
&=|\nabla_m(f_+^2)|_m^2+|\nabla_m (f_-^2)|_m^2\notag\\
&=|\nabla_m(f_+^2)|_m^2 +2\cdot |\nabla_m (f_+^2)|_m\cdot |\nabla_m(f_-^2)|_m+|\nabla_m (f_-^2)|_m^2\notag\\
&=\left(|\nabla_m(f_+^2)|_m+|\nabla_m (f_-^2)|_m\right)^2.
\end{align}
Finally,  
by definition $\mathcal{L}f_+=\max\{\mathcal{L}f-\sigma, 0\}$ and $\mathcal{L}f_-=-\min\{\mathcal{L}f-\sigma, 0\}$. Hence analogous to \eqref{eq:wci2a}
\begin{equation}\label{eq:wci2c}
(\mathcal{L}f-\sigma)^2=\mathcal{L}f_+^2+\mathcal{L}f_-^2
\end{equation}
and analogous to \eqref{eq:wci2b}
\begin{equation}\label{eq:wci2d}
|\nabla_n(\mathcal{L}f_+^2+\mathcal{L}f_-^2)|_n^2=(|\nabla_n (\mathcal{L}f_+^2)+\nabla_n(\mathcal{L}f_-^2)|_n)^2.
\end{equation}

Due to \eqref{eq:wci2a}-\eqref{eq:wci2d}, one has
\begin{align}\label{eq:wci3}
&\int_M \big|\nabla_m[(f-\sigma)^2]\big|_m\,d\mu_r+\int_M \big|\nabla_n [(\mathcal{L}f-\sigma)^2]\big|_n\,d\nu_r\notag\\
=&\int_M \big|\nabla_m(f_+^2+f_-^2)\big|_m\,d\mu_r+\int_N \big|\nabla_n(\mathcal{L}f_+^2+\mathcal{L}f_-^2)\big|_n\,d\nu_r\quad\mbox{by \eqref{eq:wci2a} and \eqref{eq:wci2b}}\notag\\
=&\int_M\left( |\nabla_m (f_+^2)|_m+ |\nabla_m (f_-^2)|_m \right)\,d\mu_r+ \int_N \left(|\nabla_n (\mathcal{L}f_+^2)|_n+|\nabla_n (\mathcal{L}f_-^2)|_n\right)\,d\nu_r,
\end{align}
where the last line is due to \eqref{eq:wci2b} and \eqref{eq:wci2d}.

Now, consider the RHS of \eqref{eq:wci3}. Since $f_+^2$ and $f_-^2$ are nonnegative and smooth almost everywhere, one can set $g=f_+^2$ and $g=f_-^2$ independently in \eqref{eq:wci1}, and then apply \eqref{eq:wci9b} to the result to obtain
\begin{align}\label{eq:wci9}
&\mbox{RHS of \eqref{eq:wci3}}\notag\\
\geq&
 2 \inf_{t\in (-\infty, \infty)}\mathbf{H}^D(\{f=t\})\int_0^\infty \min\{\mu_r(f^2_+>t), \mu_r(f^2_+<t)\}+\min\{\mu_r(f^2_->t), \mu_r(f^2_-<t)\}\,dt\notag\\
=&2 \inf_{t\in (-\infty, \infty)}\mathbf{H}^D(\{f=t\})   \int_0^\infty \mu_r(f_+^2> t)+\mu_r(f_-^2> t)\,dt,
\end{align}
where the equality on the last line is due to \eqref{eq:wci7}. Applying the Cavalieri\rq{}s principle (Proposition I.3.3 in \cite{chavel01}) to the RHS of \eqref{eq:wci9} yields
\begin{align}\label{eq:wci10}
\mbox{RHS of \eqref{eq:wci9}}
&=2\inf_{t\in (-\infty, \infty)}\mathbf{H}^D(\{f=t\})  \int_M (f_+^2+ f_-^2) \,d\mu_r\notag\\
&=2\inf_{t\in (-\infty, \infty)}\mathbf{H}^D(\{f=t\})\int_M (f-\sigma)^2 \,d\mu_r.
\end{align}
Next, we consider the LHS of \eqref{eq:wci3}. In local coordinates, one has by \eqref{eq:gradl}
\begin{align*}
\nabla_m[(f-\sigma)^2]
&=\sum_{i, j=1}^r m^{ij}\partial_i (f-\sigma)^2 \partial_j\\
&=2\sum_{i, j=1}^r m^{ij} (f-\sigma)\partial_i f\partial_j\\
&=2(f-\sigma)\nabla_m f.
\end{align*}
Therefore, by Cauchy-Schwarz
\begin{align} \label{eq:wci4}
\int_M \big|\nabla_m[(f-\sigma)^2] \big|_m\, d\mu_r
&= 2\int_M |f-\sigma|\cdot\big|\nabla_m f\big|_m\,d\mu_r\notag\\
&\leq 2\|f-\sigma\|_{2, m, \mu}\cdot \|\nabla_mf\|_{2, m, \mu}.
\end{align}
Also, analogous to \eqref{eq:wci4}
\begin{align}\label{eq:wci5}
\int_N \big|\nabla_n[(\mathcal{L}f-\sigma)^2] \big|_n\,d\nu_r
&\leq 2\|\mathcal{L}f-\sigma\|_{2, n, \nu}\cdot \|\nabla_n \mathcal{L}f\|_{2, n, \nu}\notag\\
&=2\left(\int_N (\mathcal{L}f-\sigma)^2 \,d\nu_r  \right)\cdot \|\nabla_n \mathcal{L}f\|_{2, n, \nu}\notag\\
&=2\left(\int_M (f-\sigma)^2 \,d\mu_r \right) \cdot \|\nabla_n \mathcal{L}f\|_{2, n, \nu}\quad\mbox{by \eqref{eq:cov3}}\notag\\
&= 2\|f-\sigma\|_{2, m, \mu}\cdot \|\nabla_n \mathcal{L}f\|_{2, n, \nu}.
\end{align}
Therefore, by \eqref{eq:wci3}-\eqref{eq:wci5}, one has
\begin{align}\label{eq:wci6}
\inf_{t\in (-\infty, \infty)}\mathbf{H}^D(\{f=t\}) \int_M (f-\sigma)^2\,d\mu_r
&\leq  \|f-\sigma\|_{2, m, \mu}\cdot \left(\|\nabla_mf\|_{2, m, \mu}+ \|\nabla_n \mathcal{L}f\|_{2, n, \nu}\right)\notag\\
\implies \inf_{t\in (-\infty, \infty)}\mathbf{H}^D(\{f=t\})
&\leq \frac{\|\nabla_m f\|_{2, m, \mu}+\|\nabla_n\mathcal{L}f\|_{2, n, \nu}}{\left(\int_M(f-\sigma)^2\, d\mu_r\right)^{1/2}}.
\end{align}

Let $\alpha(f)$ be the mean of $f$ with respect to $\mu_r$; that is $\alpha(f)=\int_M f\, d\mu_r$. Then $\int_M (f-c)^2 \,d\mu_r$ as a function of $c\in \mathbb{R}$ is minimum when $c=\alpha(f)$. Hence, by squaring both sides of \eqref{eq:wci6}, one has
\begin{align}\label{eq:wci8}
\left(\inf_{t\in (-\infty, \infty)}\mathbf{H}^D(\{f=t\})\right)^2
&\leq \frac{\left(\|\nabla_m f\|_{2, m, \mu}+\|\nabla_n\mathcal{L}f\|_{2, n, \nu}\right)^2}{\int_M(f-\alpha(f))^2\,d\mu_r}\notag\\
&\leq 2\frac{\int_M |\nabla_m f|_m^2\,d\mu_r+\int_N |\nabla_n \mathcal{L}f|_n^2\,d\nu_r}{\int_M(f-\alpha(f))^2\,d\mu_r},
\end{align}
for all $f\in C^\infty(M, \mathbb{R})$, where we have used the fact that $(a+b)^2\leq 2(a^2+b^2)$ for $a, b\in \mathbb{R}$ to obtain the inequality on the last line. Furthermore, if $\lambda_2$ is the smallest magnitude nonzero eigenvalue of $\triangle^D$ with corresponding eigenfunction $\phi_2$, then by Theorem \ref{thm:spec}, one has $\phi_2\in C^\infty(M, \mathbb{R})$, $\alpha(\phi_2)=\int_M \phi_2\, d\mu_r=0$, and for $k=2$ the infimum of \eqref{eq:3.7} is attained by $\phi_2$. Thus, by setting $f=\phi_2$ in \eqref{eq:wci8},
\begin{align*}
\left(\inf_{t\in (-\infty, \infty)}\mathbf{H}^D(\{\phi_2=t\})\right)^2
&\leq 2\frac{\int_M |\nabla_m \phi_2|_m^2\,d\mu_r+\int_N |\nabla_n \mathcal{L}\phi_2|_n^2\,d\nu_r}{\int_M |\phi_2-\alpha(\phi_2)|^2\, d\mu_r}\\
&= -4\lambda_2.
\end{align*}
This concludes the proof of the theorem.

\subsection{Time-discrete and time-continuous case}\label{sec:multip}
To generalise Theorem \ref{thm:wci} to the time-continuous dynamic Cheeger inequality, we note that apart from \eqref{eq:wci8} all arguments are applied linearly with respect to time. Hence, the results up to \eqref{eq:wci8} are immediate via the constructions outlined in Sections \ref{sec:mts} and \ref{sec:lts}. To modify the argument $(a+b)^2\leq 2(a^2+b^2)$ used to obtain \eqref{eq:wci8}, we apply Cauchy-Schwarz to obtain
\begin{align*}
\left(\int_0^\tau a_t \,dt\right)^2=\left(\int_0^\tau a_t\cdot 1 \,dt\right)^2 \leq \left(\int_0^\tau a_t^2 \,dt\right)\cdot \left(\int_0^\tau 1^2\,dt\right)=\tau\cdot \int_0^\tau a_t^2 \,dt.
\end{align*}
For the time-discrete case, one applies Cauchy-Schwarz analogously.

\section{The proof of Theorem \ref{thm:3.1}}\label{sec:p3.1}



Recall the definition of the diffusion operator $\mathcal{D}_{X, \epsilon}$ given by \eqref{eq:diffusion1}.
For $f\in C^3(M, \mathbb{R})$, we wish to evaluate the $\epsilon\to 0$ limit of the image of $f$ under the operator $\mathcal{L}_\epsilon^*\mathcal{L}_\epsilon$, where by \eqref{eq:3.2} and \eqref{eq:3.3},
\begin{equation}\label{eq:LeLe}
 \mathcal{L}_\epsilon^*\mathcal{L}_\epsilon f=\mathcal{D}^*_{X_{\epsilon},\epsilon}\circ\mathcal{L}^*\circ\mathcal{D}^*_{Y'_\epsilon,\epsilon}
 \left(\frac{\mathcal{P}_\epsilon(fh_\mu)}{\mathcal{P}_\epsilon h_\mu}\right),
\end{equation}
with
$
 \mathcal{P}_\epsilon =\mathcal{D}_{Y'_\epsilon,\epsilon}\circ\mathcal{P}\circ\mathcal{D}_{X,\epsilon}.
$
Let $(U, \varphi)$ be a chart on $M$ containing the point $x\in M$. Recall normal coordinates at the point $x$, are the local coordinates on $(U, \varphi)$ such that the metric tensor satisfies $m_{ij}(x)=\delta_{ij}$ and $\partial_i m_{jk}(x)=0$ for all $1\leq i, j, k\leq r$.

Introducing standard multi-index notation for $\alpha$; i.e\ $\alpha=(\alpha_1, \alpha_2, \ldots, \alpha_r)$ such that
\begin{equation}\label{eq:stnotation}
\begin{aligned}
|\alpha|&=\alpha_1+\ldots+\alpha_r\\
\alpha!&=\alpha_1!\alpha_2!\ldots \alpha_r!\\
D^\alpha&=\partial_1^{\alpha_1}\partial_2^{\alpha_2}\ldots\partial_r^{\alpha_r}\\
v^\alpha&=v_1^{\alpha_1} v_2^{\alpha_2}\ldots v_r^{\alpha_r},
\end{aligned}
\end{equation}
for a vector $v=(v_1, \ldots, v_r)$.
The following lemmas are well known results regarding normal coordinates
\begin{lemma}\label{thm:laplaciannormal}
Let $(U, \varphi)$ be a chart of $M$ containing the point $x_0\in M$, with corresponding normal coordinates $\{x_1, x_2, \ldots, x_r\}$. The Laplace-Beltrami operator satisfies
\begin{equation*}
\triangle_m f(x_0)=\sum_{i=1}^r\frac{\partial^2 (f\circ \varphi^{-1})}{\partial x_i^2}(\varphi (x_0)).
\end{equation*}
\end{lemma}
\begin{proof}
 See p.90 in \cite{rosenberg97}.
\end{proof}

\begin{lemma}\label{thm:texpdet}
 Let $(U, \varphi)$ be a chart of $M$ containing the point $x_0\in M$ with corresponding coordinates $\{x_1, x_2, \ldots, x_r\}$. The asymptotic expansion of $\sqrt{\det G_m}$ about  $B_\epsilon(x_0)\subseteq U$, centered at $x_0$ is given by
\begin{equation*}\label{eq:texpdet}
\sqrt{\det{G_m}}(x)=1+\sum_{|\alpha|=2}^\infty C_{\mathcal{R}, |\alpha|}(x_0)\cdot (\varphi(x))^\alpha,
\end{equation*}
where $C_\mathcal{R, |\alpha|}(x_0)$ depend only on the  Riemannian curvature tensor $\mathcal{R}$ and covariant derivatives of $\mathcal{R}$ at the point $x_0$. Moreover, if $\mathcal{R}$ is bounded on $B_\epsilon(x_0)$, then
\begin{equation}\label{eq:texpdet2}
\sum_{|\alpha|=2}^\infty |C_{\mathcal{R}, |\alpha|}(x_0)|<\infty
\end{equation}
\end{lemma}

\begin{proof}
See Corollary 2.10 in \cite{gray74}.
\end{proof}

The following lemma generalises Lemma D.1 \cite{froyland14} for flat manifolds to the case of general Riemannian manifolds.
\begin{lemma}\label{thm:app2.1}
 Let $\mathbf{1}$ denote the characteristic function, and $\mathcal{D}_{X, \epsilon}$ be defined as in \eqref{eq:diffusion1}. There exist a constant $c$, such that
 \begin{equation*}
  \lim_{\epsilon\to 0}\sup_{\|f\|_{C^3(M, \mathbb{R})}\leq K}\left \|\frac{\mathcal{D}_{X,\epsilon}f- f}{\epsilon^2}-(c/2)\triangle_m f \right\|_{C^0(M, \mathbb{R})}=0,
 \end{equation*}
for each $K<\infty$.
\end{lemma}
\begin{proof}
Let $f\in C^3(M, \mathbb{R})$ with $\|f\|_{C^3(M, \mathbb{R})}\leq K$, fix $x_0\in X$ and set $\epsilon>0$ to be smaller than the injectivity radius of the point $x_0\in M$. It is well known that the exponential map $\textup{exp}_{x_0}$ at the point $x_0$ is a diffeomorphism of a neighbourhood of $0\in \mathbb{R}^r$ onto $B_\epsilon(x_0)$ (see Theorem 5.11, \cite{aubin01}). Moreover, there exist normal coordinates on the chart $(B_\epsilon(x_0), \exp^{-1}_{x_0})$; that is the components of the metric tensor $m$ satisfy $m_{ij}=\delta_{ij}$, and $\partial_k m_{ij}=0$ at the point $x_0$ for all $1\leq i, j,k\leq r$ (see Corollary 5.12, \cite{aubin01}).

Recall the definition of $Q$ from Section \ref{sec:3}.
By the Gauss lemma for Riemannian manifolds, the exponential map $\exp_{x_0}$ is a radial isometry from $E_\epsilon(0)$ to $B_\epsilon(x_0)$ (see Lemma 3.5, p.69 in \cite{carmo92}). Thus,
\begin{align}\label{eq:diff0}
Q_\epsilon(x_0, z)
&:=\epsilon^{-r}Q\left(\frac{\textup{dist}_m(x_0,z)}{\epsilon}\right)\notag\\
&=\epsilon^{-r}Q\left(\frac{|\exp_{x_0}^{-1} x_0-\exp_{x_0}^{-1} z|}{\epsilon}\right)=\epsilon^{-r}Q\left(\frac{|\exp_{x_0}^{-1} z|}{\epsilon}\right),
\end{align}
for all $z\in B_\epsilon(x_0)\subset M$. Moreover, due to the fact that $\textup{supp } Q\subset E_1(0)$, the function $Q_\epsilon$ vanishes for all $z\in M\setminus B_\epsilon(x_0)$.

Let $\{x_1, \ldots, x_r\}$ denote normal coordinates on $(B_\epsilon(x_0), \exp_{x_0}^{-1})$. Recall that the volume form on $M$ is given by $\omega_m^r=\sqrt{\det{G_m}}\cdot dx_1\wedge dx_2\wedge\ldots \wedge dx_r$, where $G_m$ is a $r\times r$ matrix with entries $m_{ij}$. Hence $(\exp_{x_0})^*\omega_m^r= \sqrt{\det G_m}\circ \exp_{x_0}\,d\ell$, where $\ell$ is the Lebesgue measure on $\mathbb{R}^r$. Moreover, since $\textup{supp }Q_\epsilon (x_0, \cdot)\subset B_\epsilon(x_0)$, one has
\begin{align}\label{eq:diff1}
\mathcal{D}_{X, \epsilon}f(x_0)
&=\int_{B_\epsilon(x_0)} Q_{\epsilon}(x_0, z)f(z)\cdot \omega_m^r(z)\notag\\
 &=\epsilon^{-r}\int_{E_\epsilon(0)} Q\left(\frac{|u|}{\epsilon}\right)\left(f\cdot \sqrt{\det G_m}\right)\circ \exp_{x_0} (u)\cdot  d\ell(u),
\end{align}
where the last line is due to \eqref{eq:diff0}. An application of the change of variable $v=u/\epsilon$ to the RHS of \eqref{eq:diff1} yields
\begin{equation}\label{eq:diff1b}
  \mbox{RHS of \eqref{eq:diff1}}=\int_{E_1(0)} Q\left(v\right)\left(f\cdot \sqrt{\det G_m}\right)\circ \exp_{x_0}(\epsilon v)\,d\ell(v).
\end{equation}
To complete the proof of the lemma from \eqref{eq:diff1}, we follow the proof of Lemma D.1 \cite{froyland14}.  We apply Taylor's theorem to the real-valued function $\bar{f}:=f \circ \exp_x$ on $E_1(0)$, centered at $0$ to obtain
\begin{equation*}
\bar{f}(\epsilon v)=\sum_{|\alpha|=0}^2 (\epsilon v)^\alpha\frac{D^\alpha \bar{f}(0)}{\alpha!}+\sum_{|\alpha|=3} (\epsilon v)^\alpha R_\alpha (\epsilon v)
\end{equation*}
where the remainder term $R_\alpha(\epsilon v)$ is given by
\begin{equation}\label{eq:diffremain}
R_\alpha(\epsilon v)=\frac{3}{\alpha!}\int_0^1(1-t)^{2}D^\alpha \overline{f}(t\epsilon v)\,dt.
\end{equation}
Due to the above Taylor expansion of $\bar{f}$, the RHS of \eqref{eq:diff1b} becomes
\begin{align}\label{eq:diff2}
 \mbox{RHS of \eqref{eq:diff1b}}=\int_{E_1(0)}Q(|v|)\left[\sum_{|\alpha|=0}^2 \epsilon^{|\alpha|} v^\alpha\frac{D^\alpha \bar{f}(0)}{\alpha!}+\sum_{|\alpha|=3}\epsilon^3 v^\alpha R_3(\epsilon v)\right]\cdot\left(\sqrt{\det G_m}\right)\circ \exp_{x_0}(\epsilon v)\,d\ell(v).
\end{align}
We evaluate the above integral term by term.

For the $|\alpha|=0$ term, one has
\begin{align}\label{eq:diff3}
 &\int_{E_1(0)} Q(|v|)\cdot \frac{\bar{f}(0)}{0!}\cdot\left(\sqrt{\det G_m}\right)\circ \exp_{x_0}(\epsilon v)\,d\ell(v)\notag\\
 =&\int_{E_1(0)} Q(|v|)\cdot f(x_0)\cdot\left(\sqrt{\det G_m}\right)\circ \exp_{x_0}(\epsilon v)\,d\ell(v)\notag\\
 =& f(x_0)\int_{B_\epsilon(x_0)}Q_{m, \epsilon}(x_0, z)\cdot\omega_m^r(z)=f(x_0).
\end{align}

For the $|\alpha|=1$ term, 
we note that the real-valued function $Q(|v|)$ is symmetric. Hence, $v_i Q(|v|)$ are odd functions of $v$ for $1\leq i \leq r$. Therefore,
\begin{align}\label{eq:diff4}
  &\int_{E_1(0)} Q(|v|)\cdot \epsilon^1 \sum_{i=1}^r v_i^1\frac{\partial_i\bar{f}(0)}{1!}\cdot \sqrt{\det{G_m}}\circ \exp_{x_0}(\epsilon v)\,d\ell(v)\notag\\
  =&\sum_{i=1}^r \partial_i\bar{f}(0)\left(\int_{E_1(0)} v_iQ(|v|)\cdot \epsilon\left(1+\sum_{|\beta|=2}^\infty C_{\mathcal{R}, |\beta|}(x_0)\epsilon^{|\beta|}v^\beta\right)\,d\ell(v)\right)\notag\\
 = &\sum_{i=1}^r \partial_if(x_0)\cdot  \left(0+\int_{E_1(0)}  Q(|v|)\cdot \sum_{|\beta|=2}^\infty C_{\mathcal{R}, |\beta|}(x_0)v_i v^\beta\epsilon^{|\beta|+1}\,d\ell(v)\right),
\end{align}
where we have applied Lemma \ref{thm:texpdet} to obtain the second equality, with constants $C_{\mathcal{R}, |\beta|}(x_0)<\infty$ depend only on the Riemannian curvature tensor $\mathcal{R}$, and the covariant derivatives of $\mathcal{R}$ at the point $x_0$.
We return to this term later in the proof, but for now we proceed to the $|\alpha|=2$ term.



For the $|\alpha|=2$ term, due to the property \eqref{eq:diffusion4} for $Q$ and the approximation of $\sqrt{G_m}$ by Lemma \ref{thm:texpdet}, one has
\begin{align}\label{eq:diff5}
 &\int_{E_1(0)} Q(|v|)\cdot \left( \epsilon^2 \sum_{i, j=1}^r v_iv_j \frac{\partial_{i}\partial_j\bar{f}(0)}{2!}\right)\cdot \left(1+\sum_{|\beta|=2}^\infty C_{\mathcal{R}, |\beta|}(x_0)\epsilon^{|\beta|}v^\beta \right)\,d\ell(v)\notag\\
=&\sum_{i,j=1}^r\frac{\partial_i\partial_j \bar{f}(0)}{2!}\cdot \left(\int_{E_1(0)} Q(|v|)\cdot \left(v_iv_j\epsilon^2+\sum_{|\beta|=2}^\infty C_{\mathcal{R}, |\beta|}(x_0)v_iv_j v^\beta\epsilon^{|\beta|+2}\right)\,d\ell(v)\right)\notag\\
 =&\frac{c\epsilon^2}{2}\sum_{i=1}^r\partial_i^2 \bar{f}(0)+\sum_{i,j=1}^r\frac{\partial_i\partial_j \bar{f}(0)}{2}\cdot \left(\int_{E_1(0)} Q(|v|)\cdot \sum_{|\beta|=2}^\infty C_{\mathcal{R}, |\beta|}(x_0)v_iv_j v^\beta\epsilon^{|\beta|+2}\,d\ell(v)\right)\notag\\
 =&\frac{c\epsilon^2}{2}\triangle_m f(x_0)+\sum_{i, j=1}^r\frac{\partial_i\partial_j {f}(x_0)}{2}\cdot \left(\int_{E_1(0)}  Q(|v|)\cdot \sum_{|\beta|=2}^\infty C_{\mathcal{R}, |\beta|}(x_0)v_iv_j v^\beta\epsilon^{|\beta|+2}\,d\ell(v)\right),
 \end{align}
 where we have applied Lemma \ref{thm:laplaciannormal} to obtain the last line.

Now set $\epsilon\leq \min\{\rho, 1\}$, where $\rho$ is smaller than the injectivity radius for every $x\in M$, then the approximations \eqref{eq:diff1}-\eqref{eq:diff5} are valid for every point $x\in M$. Moreover, since $M$ is compact $\mathcal{R}$ is bounded on $M$. Therefore, by \eqref{eq:texpdet2}, if $v\in E_1(0)$ then there exists a constant $C_1$ such that
\begin{equation}\label{eq:diff4b}
\sum_{|\beta|=2}^\infty \left|\mathcal{C}_{\mathcal{R}, |\beta|}(x)v_i v^{|\beta|}\epsilon^{|\beta|+1}\right|
\leq \epsilon^3 \left(\sum_{|\beta|=2}^\infty \left| \mathcal{C}_{\mathcal{R}, |\beta|}(x)\right|\right)
=\frac{C_1}{\int_{E_1(0)}Q(|v|)\,d\ell(v)}\epsilon^3,
\end{equation}
for each $i\geq 1$ and all $x\in M$. Similarly, if $v\in E_1(0)$ then there exists a constant $C_2$ such that
\begin{equation}\label{eq:diff5b}
\sum_{|\beta|=2}^\infty \left|\mathcal{C}_{\mathcal{R}, |\beta|}(x)v_iv_j v^{|\beta|}\epsilon^{|\beta|+2}\right|
\leq \epsilon^4 \left(\sum_{|\beta|=2}^\infty \left| \mathcal{C}_{\mathcal{R}, |\beta|}(x)\right|\right)
=\frac{C_2}{\int_{E_1(0)}Q(|v|)\,d\ell(v)}\epsilon^4,
\end{equation}
for each $i, j\geq 1$ and all $x\in M$. Due to \eqref{eq:diff1}-\eqref{eq:diff5b}, one has
  \begin{align}\label{eq:codiff1}
 \left|\mathcal{D}_{X, \epsilon}f(x)-f(x)-\frac{c\epsilon^2}{2}\triangle_m f(x)\right|
 \leq
 \begin{aligned}
 &\left| \left(\sum_{i=1}^r \partial_i f(x)\right)\cdot C_1\epsilon^3\right|+\left|\left(\sum_{i,j=1}^r\partial_i\partial_jf(x)\right)\cdot C_2\epsilon^4\right|\\
& \hfill+\left|\int_{E_1(0)}Q(|v|)\cdot \sum_{|\alpha|=3}\epsilon^3 v^\alpha R_\alpha(\epsilon v)\cdot \sqrt{\det{G_m}}\circ \exp_x (v)\,d\ell(v)\right|.
 \end{aligned}
 \end{align}
for all $x\in M$. Consider the term on the second line of \eqref{eq:codiff1}, one has
\begin{align*}
&\sup_{\|f\|_{C^3(M, \mathbb{R})}\leq K} \left\|\int_{E_1(0)}Q(|v|)\cdot \sum_{|\alpha|=3}\epsilon^3 v^\alpha R_\alpha(\epsilon v)\cdot \sqrt{\det{G_m}}\circ \exp_x (v)\,d\ell(v)\right\|_{C^0(M, \mathbb{R})}\\
\leq &   \sup_{\substack{\|f\|_{C^3(M, \mathbb{R})}\leq K\\ |\alpha|=3, u\in E_\epsilon(0)}}\|R_\alpha(u)\|_{C^0(M, \mathbb{R})} \cdot \epsilon^3 \int_{E_{1}(0)} Q(|v|)\cdot \sum_{|\alpha|=3} v^\alpha \sqrt{\det{G_m}}\circ \exp_x (v)\,d\ell(v)\\
=&\sup_{\substack{\|f\|_{C^3(M, \mathbb{R})}\leq K\\ |\alpha|=3, u\in E_\epsilon(0)}}\|R_\alpha(u)\|_{C^0(M, \mathbb{R})} \cdot C_3 \epsilon^3,
\end{align*}
for some constant $C_3$. Therefore, rearranging \eqref{eq:codiff1} yields
\begin{align}\label{eq:codiff2}
 &\sup_{\|f\|_{C^3(M, \mathbb{R})}\leq K}\left\|\frac{(\mathcal{D}_{X, \epsilon}-I)f}{\epsilon^2}-(c/2)\triangle_m f\right\|_{C^0(M, \mathbb{R})}\notag\\
 \leq&
 \begin{aligned}
  \sup_{\|f\|_{ C^3(M, \mathbb{R})}\leq K} \left\| \left(\sum_{i=1}^r \partial_i f\right)\right\|_{C^{0}(M, \mathbb{R})}\cdot C_1\epsilon^3+\left\|\left(\sum_{i,j =1}^r\partial_i\partial_jf\right)\right\|_{C^{0}(M, \mathbb{R})}\cdot C_2\epsilon^4 &\\
 \hfill+C_3 \epsilon^1 \cdot \sup_{\substack{\|f\|_{C^3(M, \mathbb{R})}\leq K\\ |\alpha|=3, u\in E_\epsilon(0)}}\|R_\alpha(u)\|_{C^0(M, \mathbb{R})}&.
 \end{aligned}
\end{align}

Since the first and second order derivatives of $f$ are bounded for by $K$, the first two terms on the RHS of \eqref{eq:codiff2} converge to $0$ as $\epsilon\to \infty$. Hence, to complete the proof of the theorem it suffices to show that
\begin{equation*}
 R_\alpha(u)=\frac{3}{\alpha!}\int_0^1(1-t)^{2}D^\alpha (f\circ \exp_x)(tu)\,dt,
\end{equation*}
is uniformly bounded on $E_\epsilon(0)$, for $|\alpha|=3$ and every $f\in C^3(M, \mathbb{R})$ with $\|f\|_{C^3(M, \mathbb{R})}\leq K$.

Let $u\in E_\epsilon(0)$ and $|\alpha|=3$. Since $\epsilon$ is less than the injectivity radius of $x$, the exponential map $\exp_{x}^{-1}$ is a $C^\infty$-diffeomorphism from $B_\epsilon(x)$ onto $E_\epsilon(0)$. Thus, if $\|f\|_{C^3(M, \mathbb{R})}\leq K$, then all derivatives of $f\circ \exp_x$ up to order $3$ are bounded above by $K'$ for some $K'<\infty$ on $E_\epsilon(0)$.
Now since $u\in E_\epsilon(0)$, one has $tu\in E_\epsilon(0)$ for all $0\leq t \leq 1$. Hence, the term $D^\alpha (f\circ \exp_x)(tu)$ is uniformly bounded in $u$ for $0\leq t\leq 1$, and all $\|f\|_{C^3(M, \mathbb{R})}\leq K$. It follows that the remainder $R_\alpha$ is uniformly bounded on $E_\epsilon(0)$, for $|\alpha|=3$ and every $\|f\|_{C^3(M, \mathbb{R})}\leq K$.

\end{proof}

\begin{proof}[\textbf{Proof of Theorem \ref{thm:3.1}}]

Let $\|f\|_{C^3(M, \mathbb{R})}\leq 1$, and set $\epsilon>0$ to be smaller than the injectivity radius of each point in $M$. We start with the asymptotic expansions of $\mathcal{P}_\epsilon(fh_\mu)$. Since $\|f\|_{C^3(M, \mathbb{R})}\leq 1$ and $h_\mu$ is bounded in the $C^3$-norm, one has $\|fh_\mu\|_{C^3(M, \mathbb{R})}\leq K$, for some constant $K$. Consider $fh_\mu$ such that $\|fh_\mu\|_{C^3(M, \mathbb{R})}\leq K$.
Lemma \ref{thm:app2.1}  yields $\mathcal{D}_{X,\epsilon}(fh_\mu)=fh_\mu+\frac{c\epsilon^2}{2}\triangle_m(fh_\mu)+\mathcal{O}(\epsilon^3)$, where $\mathcal{O}(\epsilon^3)$ denotes the class of polynomials $a_3\epsilon^3+a_4\epsilon^4+\ldots$, with all coefficients $a_3, a_4, \ldots$ bounded on $M$ and independent of $f$. Combining the expansion of $\mathcal{D}_{X, \epsilon}(fh_\mu)$ with the linearity of $\mathcal{P}$, then $\mathcal{P}\mathcal{D}_{X,\epsilon} (fh_\mu)=\mathcal{P}(fh_\mu) +\frac{c\epsilon^2}{2}\mathcal{P}\triangle_m (fh_\mu)+\mathcal{O}(\epsilon^3)$. Now, since $\mathcal{P}$ is given by \eqref{eq:P-F2} and $T$ is a $C^\infty$-diffeomorphism, one has $\mathcal{P} \mathcal{D}_{X, \epsilon} (fh_\mu)\in F^3(N, \mathbb{R})$. Therefore, by a straightforward modification of Lemma \ref{thm:app2.1}, we have uniformly on $N$
\begin{align}\label{eq:app2.1}
 \mathcal{P}_\epsilon(fh_\mu)\notag
=&\mathcal{D}_{Y'_\epsilon,\epsilon}\mathcal{P}\mathcal{D}_{X,\epsilon} (fh_\mu)\\
=&\mathcal{P}(fh_\mu) +\frac{c\epsilon^2}{2}\mathcal{P}\triangle_m (fh_\mu)+\frac{c\epsilon^2}{2}\left[\triangle_n \mathcal{P}(fh_\mu)+\mathcal{O}(\epsilon^2)\right]+\mathcal{O}(\epsilon^3)\notag\\
=&\mathcal{P}(fh_\mu) +\frac{c\epsilon^2}{2}\left[\mathcal{P}\triangle_m (fh_\mu)+\triangle_n \mathcal{P}(fh_\mu)\right]+\mathcal{O}(\epsilon^3),
\end{align}
where $c$ is the same constant as in Lemma \ref{thm:app2.1} (since the constant $c$ comes from the property \eqref{eq:diffusion4} of $Q$, independent of $f$). Therefore, using the fact that $\mathcal{P}h_\mu=h_\nu$
\begin{equation}\label{eq:app2.2}
 \mathcal{L}_\epsilon f
 =\frac{\mathcal{P}_\epsilon(fh_\mu)}{\mathcal{P}_\epsilon h_\mu}
 =\frac{\mathcal{P}(fh_\mu) +\frac{c\epsilon^2}{2}\left[\mathcal{P}\triangle_m (fh_\mu)+\triangle_n \mathcal{P}(fh_\mu)\right]+\mathcal{O}(\epsilon^3)}
 {h_\nu +\frac{c\epsilon^2}{2}\left[\mathcal{P}\triangle_m h_\mu+\triangle_n h_\nu \right]+\mathcal{O}(\epsilon^3)},
\end{equation}
uniformly on $N$. Next we apply $\mathcal{L}_\epsilon^*$ to $\mathcal{L}_\epsilon f$. According to $\eqref{eq:LeLe}$, the first step is the application of the dual diffusion operator $\mathcal{D}^*_{Y'_\epsilon,\epsilon}$ to \eqref{eq:app2.2}. In preparation for this, we consider a general polynomial quotient of the form
\begin{equation*}
\frac{a+b\epsilon^2+c\epsilon^3}{d+e\epsilon^2+f\epsilon^3}
\end{equation*}
where $a, b,\ldots, f$ are a set of known coefficients. By polynomial long division and truncating at $\epsilon^3$, one has
\begin{equation}\label{eq:app2.3}
\frac{a+b\epsilon^2+c\epsilon^3}{d+e\epsilon^2+f\epsilon^3}= \frac{a}{d}+\frac{bd-ae}{d^2}\epsilon^2 + \mathcal{O}(\epsilon^3).
\end{equation}
Applying \eqref{eq:app2.3} to \eqref{eq:app2.2}, and noting that $\mathcal{L}f=\mathcal{P}(f\cdot h_\mu)/h_\nu$ (see \eqref{eq:pfw}) yields
\begin{align*}
 \mathcal{L}_\epsilon f
 &=\frac{\mathcal{P}(fh_\mu)}{h_\nu}+\frac{c\epsilon^2}{2}\left[\frac{\mathcal{P}\triangle_m(fh_\mu)}{h_\nu}+\frac{\triangle_n\mathcal{P}(fh_\mu)}{h_\nu}
 -\frac{\mathcal{P}(fh_\mu)\cdot \mathcal{P}\triangle_m h_\mu}{h_\nu^2}-\frac{\mathcal{P}(fh_\mu)\cdot\triangle_n h_\nu}{h_\nu^2}\right] +\mathcal{O}(\epsilon^3)\\
  &=\mathcal{L}f+\frac{c\epsilon^2}{2}\left[\frac{\mathcal{P}\triangle_m(fh_\mu)}{h_\nu}+\frac{\triangle_n\mathcal{P}(fh_\mu)}{h_\nu}
 -\frac{\mathcal{L}f\cdot \mathcal{P}\triangle_m h_\mu}{h_\nu}-\frac{\mathcal{L}f\cdot\triangle_n h_\nu}{h_\nu}\right] +\mathcal{O}(\epsilon^3)
\end{align*}
uniformly on $N$. Since $h_\nu$ is uniformly bounded away from zero, one can check that $\mathcal{L}_\epsilon f\in F^3(N, \mathbb{R})$. Hence, it is now straightforward to compute $\mathcal{D}^*_{Y'_\epsilon,\epsilon} \mathcal{L}_\epsilon f$ via Lemma \ref{thm:app2.1} to obtain
\begin{align}\label{eq:app2.4}
  \mathcal{D}^*_{Y'_\epsilon,\epsilon} \mathcal{L}_\epsilon f=\mathcal{L}f+\frac{c\epsilon^2}{2}\left[\frac{\mathcal{P}\triangle_m(fh_\mu)}{h_\nu}+\frac{\triangle_n\mathcal{P}(fh_\mu)}{h_\nu}
 -\frac{\mathcal{L}f\cdot \mathcal{P}\triangle_m h_\mu}{h_\nu}-\frac{\mathcal{L}f\cdot\triangle_n h_\nu}{h_\nu}\right]\notag\\
 +\frac{c\epsilon^2}{2}\triangle_n\mathcal{L}f +\mathcal{O}(\epsilon^3),
\end{align}
uniformly on $N$.
We write
\begin{align*}
 \frac{\mathcal{P}(\triangle_m(fh_\mu))}{h_\nu}
 &=\mathcal{L}\left(\frac{\triangle_m(fh_\mu)}{h_\mu}\right)\\
 &=\mathcal{L}\left(\frac{f\cdot \triangle_m h_\mu+h_\mu\cdot \triangle_m f+2m(\nabla_m h_\mu, \nabla_m f)}{h_\mu}\right)\\
 &=\mathcal{L}\left(\frac{f\cdot \triangle_m h_\mu}{h_\mu}\right)+\mathcal{L}\left(\triangle_m f+\frac{2m(\nabla_m h_\mu, \nabla_m f)}{h_\mu}\right)\quad\mbox{by linearity of $\mathcal{L}$}\\
 &=\frac{\mathcal{L}f\cdot\mathcal{P}\triangle_m h_\mu}{h_\nu}+\mathcal{L}\left(\triangle_m f+\frac{2m(\nabla_m h_\mu, \nabla_m f)}{h_\mu}\right),
\end{align*}
using the fact that $\mathcal{L}f=f\circ T^{-1}$ for the last line.
Thus, the $2^{nd}$ and $4^{th}$ terms of \eqref{eq:app2.4} can be combined to form
\begin{equation}\label{eq:app2.4b}
\frac{c\epsilon^2}{2}\left(\frac{\mathcal{P}\triangle_m(f h_\mu)}{h_\nu}-\frac{\mathcal{L}f\cdot \mathcal{P}\triangle_m h_\mu}{h_\nu}\right)=c\epsilon^2\mathcal{L}\left(\frac{\triangle_m f}{2}+\frac{m(\nabla_m h_\mu, \nabla_m f)}{h_\mu}\right).
\end{equation}
Also,
\begin{align*}
\frac{\triangle_n \mathcal{P}(fh_\mu)}{h_\nu}
=\frac{\triangle_n (\mathcal{L} f \cdot h_\nu)}{h_\nu}
&=\frac{h_\nu\cdot\triangle_n \mathcal{L} f+\mathcal{L}f\cdot\triangle_n h_\nu+2n(\nabla_n h_\nu, \nabla_n \mathcal{L}f)}{h_\nu}\notag\\
&=\triangle_n \mathcal{L}f+\frac{\mathcal{L}f\cdot \triangle_n h_\nu}{h_\nu}+\frac{2n(\nabla_n h_\nu, \nabla_n \mathcal{L}f)}{h_\nu}.
\end{align*}
Thus, the $3^{rd}, 5^{th}$ and $6^{th}$ terms of \eqref{eq:app2.4} can be combined to form
\begin{align}\label{eq:app2.4a}
 &\frac{c\epsilon^2}{2}\left[\frac{\triangle_n \mathcal{P}(fh_\mu)}{h_\nu}-\frac{\mathcal{L}f\cdot\triangle_n h_\nu}{h_\nu}\right]+\frac{c\epsilon^2}{2}\triangle_n\mathcal{L}f\notag\\
 =&\frac{c\epsilon^2}{2}\triangle_n \mathcal{L} f+c\epsilon^2\frac{n(\nabla_n h_\nu, \nabla_n \mathcal{L}f)}{h_\nu}+\frac{c\epsilon^2}{2}\triangle_n\mathcal{L}f\notag\\
 =&c\epsilon^2\left(\triangle_nf + \frac{n(\nabla_n h_\nu, \nabla_n \mathcal{L}f)}{h_\nu}\right)\notag\\
 =&c\epsilon^2(\triangle_\nu \mathcal{L}f),
 \end{align}
where the last line is due to \eqref{def:wlp}. Substituting \eqref{eq:app2.4b} and \eqref{eq:app2.4a} into the RHS of \eqref{eq:app2.4} yields,
\begin{equation}\label{eq:app2.5}
 \mathcal{D}^*_{Y'_\epsilon, \epsilon}\mathcal{L}_\epsilon f=\mathcal{L}f+c\epsilon^2\left(\mathcal{L}\left(\frac{\triangle_m f}{2}+\frac{m(\nabla_m h_\mu, \nabla_m f)}{h_\mu}\right)+\triangle_\nu\mathcal{L}f\right)+\mathcal{O}(\epsilon^3),
\end{equation}
uniformly on $N$. It is straightforward to apply $\mathcal{D}^*_{X_\epsilon, \epsilon}\mathcal{L}^*$ to the RHS of  \eqref{eq:app2.5}via Lemma \ref{thm:app2.1}, which yields
\begin{align}\label{eq:app2.6}
 \mathcal{L}_\epsilon^*\mathcal{L}_\epsilon f
 &=f+\frac{c\epsilon^2}{2}\triangle_m f+c\epsilon^2\left(\left(\frac{\triangle_m f}{2}+\frac{m(\nabla_m h_\mu, \nabla_m f)}{h_\mu}\right)+\mathcal{L}^*\triangle_\nu\mathcal{L}f\right)+\mathcal{O}(\epsilon^3)\notag\\
 &=f+\frac{c\epsilon^2}{2}\left(\triangle_\mu f+\mathcal{L}^*\triangle_\nu\mathcal{L}f\right)+\mathcal{O}(\epsilon^3),
\end{align}
uniformly on $M$, where we have used \eqref{eq:dwc2} to obtain the last line. Since the coefficients of the $\mathcal{O}(\epsilon^3)$ are uniform on $M$ and independent of $f$, rearranging \eqref{eq:app2.6} gives
\begin{equation*}
\lim_{\epsilon\to 0}\sup_{\|f\|_{C^3(M, \mathbb{R})}}\left\|\frac{(\mathcal{L}^*_\epsilon\mathcal{L}-I)f}{\epsilon^2}-c\cdot\triangle^Df\right\|_{C^0(M, \mathbb{R})}=0.
\end{equation*}
\end{proof}

\end{document}